\documentclass[reqno, 11pt]{amsart}

\usepackage[margin=1in]{geometry}

\usepackage{paralist}
\usepackage{amsmath,xcolor, amscd, amsthm, amssymb}
\usepackage[hyperfootnotes=false, colorlinks, linkcolor={blue}, citecolor={magenta}, filecolor={blue}, urlcolor={blue}, pdfpagelabels]{hyperref}
\usepackage{bm} 
\usepackage{chngcntr}
\counterwithin{equation}{section}

\newtheorem{theorem}{Theorem}[section]

\newtheorem*{localtheorem}{Local Main Theorem}

\newtheorem{lemma}[theorem]{Lemma}

\newtheorem{proposition}[theorem]{Proposition}

\newtheorem{corollary}[theorem]{Corollary}

\theoremstyle{remark}
\newtheorem{remark}[theorem]{Remark}
\newcommand{\C}{{\mathbb C}}

\newcommand{\Q}{{\mathbb Q}}
\newcommand{\A}{{\mathbb A}}

\newcommand{\ip}[2]{\langle #1, #2 \rangle}

\newcommand{\gsp}{\mathop{\rm GSp}}
\newcommand{\Sp}{\mathop{\rm Sp}}
\newcommand{\gl}{\mathop{\rm  GL}}
\newcommand{\go}{\mathop{\rm  GO}}
\newcommand{\gso}{\mathop{\rm  GSO}}
\newcommand{\orth}{\mathop{\rm O}}
\newcommand{\oo}{\mathop{\rm O}}
\newcommand{\so}{\mathop{\rm  SO}}

\newcommand{\OF}{{{\mathfrak o}_L}}
\newcommand{\GL}{{\rm GL}}

\newcommand{\SL}{{\rm SL}}
\newcommand{\SP}{{\rm Sp}}
\newcommand{\SO}{{\rm SO}}
\newcommand{\GO}{{\rm GO}}
\newcommand{\OO}{{\rm O}}
\newcommand{\GSO}{{\rm GSO}}
\newcommand{\SSp}{{\rm Sp}}
\newcommand{\GSp}{{\rm GSp}}

\newcommand{\norm}{{\rm  N}}
\newcommand{\trace}{{\rm  Tr}}

\newcommand{\Mat}{{\rm  M}}
\newcommand{\sym}{{\rm Sym}}

\newcommand{\Kl}{{\rm Kl}}
\newcommand{\K}{{\rm K}}

\newcommand{\vol}{{\rm  vol}}
\newcommand{\supp}{{\rm  supp}}

\newcommand{\mat}{{\rm M}}
\newcommand{\Hom}{{\rm Hom}}
\newcommand{\Ind}{{\rm Ind}}

\allowdisplaybreaks
\begin{document}
% \mainmatter              % start of a contribution
% %
\title{An explicit theta lift to Siegel paramodular forms}
\author[J. Johnson-Leung]{Jennifer Johnson-Leung}
\address{Department of Mathematics and Statistical Science\\ University of Idaho\\ Moscow, ID~USA}
\email{jenfns@uidaho.edu}
% \author[B. Roberts]{Brooks Roberts}
% \address{Department of Mathematics and Statistical Science\\ University of Idaho\\ Moscow, ID~USA}
% \email{brooksr@uidaho.edu}
\author[N. Rupert]{Nina Rupert}
\address{Western Governors University\\ Salt Lake City, UT~USA}
\email{nina.rupert@wgu.edu}

\begin{abstract}
Let $E/L$ be a real quadratic extension of number fields. We construct an explicit map from an irreducible cuspidal automorphic representation of $\gl(2,E)$ which contains a Hilbert modular form with $\Gamma_0$ level to an irreducible automorphic representation of $\gsp(4,L)$ which contains a Siegel paramodular form and exhibit local data which produces a paramodular invariant vector for the local theta lift at every finite place, except when the local extension has wild ramification. 
\keywords{Siegel modular forms, theta lift, paramodular group, GSp(4)}
\end{abstract}
\maketitle
\section{Introduction}% (fold)
\label{sec:introduction}
The theta correspondence, introduced by Roger Howe in 1979 \cite{Howe79}, provides a robust and explicit method of relating automorphic forms on pairs of reductive groups. In general, the theta correspondence is a tool for relating automorphic adelic representations (or admissible local representations) of certain pairs of subgroups $(S_1,S_2)$ of a symplectic group $\Sp(X)$ defined over a number field (or a local field of characteristic zero). Such a map is generally called a theta lift, and this theory has been carefully developed over the past forty-five years by the work of many people. See, for example, \cite{ganIHESnotes} for a modern introduction.

The global theta lift from $S_1$ to $S_2$ takes as input a cusp form $f$ on the adeles of $S_1$ and defines an automorphic form on the adeles of $S_2$ by integrating $f$ against a theta kernel; this theta kernel depends on a choice of a certain local Schwartz function $\varphi$. In the archimedean case Schwartz functions are rapidly decreasing away from $0$, and the in non-archimedean case Schwartz functions are locally constant and compactly supported. Determining a choice for $\varphi$, which gives automorphic forms with desirable qualities, requires a theory of local theta lifts that are commensurable with the global theory.

While the global theta lift has a natural  construction as an integral operator derived from the work of Weil \cite{Weil64}, there is nothing so ubiquitous in the local theory. In \cite{wald1}, the author studies the local and global theta lift when $W$ is a 2-dimensional symplectic space, $V$ is a rank $3$ quadratic space, and $X = W \otimes V$. In \cite{wald2}, the author investigated the correspondence between modular forms of half integral weight and those with integral weight and computed special values of $L$-functions in terms of Fourier coefficients of half integral weight modular forms. 

In this work, we consider an analogous problem when $W$ is a 4-dimensional symplectic space, $V$ is a rank $4$ quadratic space, and $X = W \otimes V$. The key technical result is the construction an explicit theta lift from $\GO(X)$ to $\GSp(4)$. In this setting, the local theta correspondence was studied by Roberts in a series of papers, culminating in \cite{Roberts2001}. An integral map for the local theta lift was studied in \cite{Johnson-Leung_Roberts2012}, and the authors proved that for every Hilbert cusp form of level $\Gamma_0({\mathfrak N})$ there exists a Siegel paramodular newform with weight, level, Hecke eigenvalues, epsilon factor and $L$-function determined explicitly by the data of the Hilbert modular form. 

We continue this work by calculating explicit input data for the local theta lifts. In particular, we study an integral realization for the local theta lifts and give an explicit recipe for Schwartz functions at all places that are not wildly ramified. This allows for explicit construction of the Siegel paramodular forms appearing in this correspondence as lifts of Hilbert modular forms over any real quadratic field, except $\mathbb{Q}(\sqrt{2})$.  These paramodular forms correspond to abelian surfaces which are Weil restrictions of elliptic curves over real quadratic fields as a consequence of the results of \cite{RLS}. This provides further evidence for the paramodular conjecture of Brumer and Kramer \cite{BK}. Our result opens up a new approach to calculate Fourier coefficients of these paramodular lifts.

This paper is organized as follows: in Section \ref{sec:Main Theorem} we set up the language of theta correspondence with similitudes which is required to state our main results. Section \ref{sec:notation} contains known results about Weil representations and Whittaker models which are required for the proof. In Section \ref{sec:intertwiningmaps} we study the intertwining maps which give the local integral map of the theta lift, and finally in Section \ref{sec:testdata} we complete the proof of the Main Theorem by specifying specific Schwartz functions to produce the prescribed paramodular vector at every prime.

\subsection*{Acknowledgments}
This project was initiated as part of the second author's PhD thesis at the University of Idaho. Both authors would like to thank Brooks Roberts for providing key insights through many conversations, sharing unpublished work, and providing detailed feedback on the initial form of these results as a member of the dissertation committee.

\section{Statement of the Main Theorem}\label{sec:Main Theorem}
We define the algebraic group $\gsp(4)$ with respect to the symplectic form $J=[\begin{smallmatrix}&I\\-I&\end{smallmatrix}]$ to be the subgroup of $g \in \GL(4)$ such that ${}^tgJg=\lambda(g) J$, for some $\lambda(g)\in\GL(1)$. The Siegel upper half space $\mathcal H$ is the space of elements of $[\begin{smallmatrix}\tau&z\\z&\tau'\end{smallmatrix}]\in\mathrm{M}(2, \mathbb{C})$ whose imaginary part is positive definite. For natural numbers $N$ and $n$ and a rational prime $p$, we define global and local paramodular groups of level $N$ and $p^n$, respectively, to be
\begin{equation*}
  \K(N) = \begin{bmatrix}
    \mathbb{Z}&\mathbb{Z}&N^{-1}\mathbb{Z}&\mathbb{Z}\\
    N\mathbb{Z}&\mathbb{Z}&\mathbb{Z}&\mathbb{Z}\\
    N\mathbb{Z}&N\mathbb{Z}&\mathbb{Z}&N\mathbb{Z}\\
    N\mathbb{Z}&\mathbb{Z}&\mathbb{Z}&\mathbb{Z}
  \end{bmatrix} \cap \Sp(\mathbb{Q})\quad   \K(p^n) = \begin{bmatrix}
    \mathbb{Z}_p&\mathbb{Z}_p&p^{-n}\mathbb{Z}_p&\mathbb{Z}_p\\
    p^n\mathbb{Z}_p&\mathbb{Z}_p&\mathbb{Z}_p&\mathbb{Z}_p\\
    p^n\mathbb{Z}_p&p^n\mathbb{Z}_p&\mathbb{Z}_p&p^n\mathbb{Z}_p\\
    p^n\mathbb{Z}_p&\mathbb{Z}_p&\mathbb{Z}_p&\mathbb{Z}_p
  \end{bmatrix} \cap \GSp(\mathbb{Q}_p).
\end{equation*}
A Siegel paramodular form is a holomorphic function on $\mathcal{H}$ such that $F(\gamma \langle Z \rangle) = j(\gamma, Z)^{-k} F(Z)$ for all $\gamma \in \K(N)$ and $Z \in \mathcal H$, where $\gamma \langle Z \rangle$ is the action by fractional linear transformation. By Theorem 1.4.1 of \cite{JohnsonLeung2023} there is a bijection between cuspidal automorphic representations of $\GSp(4,\A)$ of weight $k$ which are paramodular with conductor $N$ and cuspidal Siegel eigenforms with respect to $\K(N)$ that are new of level $N$.

 For any even-dimensional non-degenerate symmetric bilinear space $X$ over a field of characteristic not equal to 2, we let $\GO(X)$ be the group of
$h \in \GL(X)$ such that there exists a $\lambda \in \GL(1)$ such that $\ip {hx} {hy} = \lambda \ip xy$ 
for all $x,y \in X$. The scalar $\lambda$ is unique, and will be denoted by $\lambda(h)$.  We let $\OO(X)$ be the subgroup of $h \in \GO(X)$ such that $\lambda(h)=1$, and we let $\SO(X)$ be the subgroup of $h \in \OO(X)$ such that $\det (h)=1$. We see that for $h \in \GO(X)$ we have $\det(h)^2 = \lambda (h)^{\dim X}$. We set $\GSO(X)$ to be the subgroup of $h \in \GO(X)$ such that $\det(h)=\lambda(h)^{\dim X/2}$ and $\SO(X) = \OO(X) \cap \GSO(X)$.

Let $L$ be a number field, let $E$ be a real quadratic extension of $L$, let $\A_L$ be the adeles of $L$, and let $\A_E$ be the adeles of $E$. Let $\pi$ be a cuspidal automorphic representation of $\GL(2, \A_E)$ that has trivial central character and is not Galois invariant. The Jacquet-Langlands correspondence \cite{Jacquet1970} produces a cuspidal automorphic representation $\pi'$ of a quaternion algebra $B^\times(\A_E)$, from the data of $\pi$. Let $X(\A_L)$ be a certain symmetric bilinear 4-dimensional $\A_L$-subspace of $B(\A_E)$. For $x , y \in X$, the symmetric product on $X$ is given by $\langle x , y \rangle = \frac{1}{2}\trace(xy^*)$. The  following sequence is exact (\cite{Knus}):
\begin{equation}
\label{intro:ES}
		1 \to \A_E^\times \to \A_L^\times \times B(\A_E)^\times \xrightarrow{\ \ \rho\ \ } \gso(X(\A_E)) \to 1.
\end{equation}
So, from the data of $\pi'$, it is simple to produce a cuspidal automorphic representation of $\GSO(X(\A_L))$, also denoted by $\pi'$. Let $\sigma$ be the associated cuspidal automorphic representation of $\go(X)$, as in Theorem 7 of \cite{Roberts2001}.  Since $\GO(X)$ and $\GSp(4)$ form a dual reductive pair, we have the existence of the Weil representation $\omega$ on the group $R = \{(g,h) \in \GSp(4,\A_L) \times \GO(X(\A_L)) \mid \lambda(g) = \lambda(h)\}$ and hence the existence of the theta lift, which is the focus of this paper. 

The space of $\omega$ is the space of Schwartz functions on $X(\A_L)^2$, which we denote by $\mathcal S(X(\A_L)^2)$. The key to determining an explicit theta lift is choosing a $\varphi \in \mathcal S(X(\A_L)^2)$ wisely, but we can define the theta lift for any choice of $\varphi$. Let $f$ be a cusp form on $\GO(X,\A_L)$ and $\varphi \in \mathcal S(X(\A_L)^2)$. Let $\GSp(4 , \A_L)^+$ be the subgroup of $g \in \GSp(4 , \A_L)$ such that $\lambda(g) \in \lambda (\GO(X (\A_L)))$. For $g \in \GSp(4 , \A_L)^+$ define:			\begin{equation*}
	\label{eq:globalthetaliftintro}
		\theta(f , \varphi)(g) = \int_{\orth(X , \Q) \setminus \orth(X , \A_L)} \vartheta(g , h_1h; \varphi) f(h_1 h )dh_1 
	\end{equation*}	
where $h \in \GO(X , \A_L)$ is any element such that $(g , h) \in R(\A_L)$ and $\vartheta$ is the global theta kernel given by
	\begin{equation*}
			\vartheta( g , h ; \varphi) = \sum_{x \in X(L)^2} (\omega(g , h)\cdot \varphi)(x).
	\end{equation*}
Then $\theta(f,\varphi)$ can be extended uniquely to all of $\GSp(4 , \A_L)$ which is left invariant under $\GSp(4 , L)$ and is, in fact, an automorphic form of $\GSp(4 , \A_L)$. 

In order to make a good choice for $\varphi \in \mathcal S(X(\A_L)^2)$ we need to examine the problem locally. Now let $L_v$ be a local  field and let $E_w$ be a quadratic extension of $L_v$. Using the path outlined in the global case one may determine a $\GO(X, L_v)$ representation $(\sigma_v, \mathcal W)$ from the data of an infinite-dimensional irreducible admissible representation $\pi_v$ of $\gl(2,E_w)$. Let $R_v = \{(g,h) \in \GSp(4,L_v) \times \GO(X(L_v)) \mid \lambda(g) = \lambda(h)\}$ where the similitude factor of each coordinate matches and let $\mathcal S(X(L_v)^2)$ be the Schwartz functions of $X(L_v)^2$. When $\sigma_v$ is irreducible as a $\GO(X)$-representation, it is the direct sum of two isomorphic $\GSO(X)$-representations. Throughout this work, we choose a distinguished $\GSO(X)$-representation, which we denote with a $+$. We provide data so that the local theta correspondence can be explicitly realized. Define the intertwining map
\begin{equation}
	\label{eq1}
	B(g, \varphi, W, s) = \int\limits_{H\backslash \SO(X)} \big(\omega(g, hh') \varphi \big)(x_1,x_2) Z(s, \pi(hh') W) \, dh
\end{equation}
where $g \in \gsp(4,L_v), \varphi \in \mathcal S(X(L_v)^2), W \in \mathcal W, \Re(s) \gg 0$, $Z(s, W)$ is the zeta integral as defined in \eqref{eq:split_zeta_def} or \eqref{eq:nonsplit_zeta_def}, and $x_1,x_2 \in X$ are certain specified elements and $H$ is their stabilizer in $\SO(X)$. Indeed, the space of  functions $\{ B(\cdot, \varphi, W,s)  \mid \varphi \in \mathcal S(X^2), W \in V\}$ is an irreducible admissible  representation $\Theta(\pi^+_v) = \Pi_v$ of $\gsp(4,L_v)$. The representation $\Pi_v$ has a canonical paramodular level ${\mathfrak p}_L^N$. Furthermore, by \cite{Roberts2001} the intertwining map is unique. The main result of this work is summarized below.
\begin{localtheorem}
Let $L$ be a non-archimedean local field of characteristic $0$ and let $E$ be either a real quadratic field extension of $L$ or let $E = L\times L$. If $E$ is a field let ${\mathfrak P}$ be the unique maximal ideal of the ring of integers of $E$, ${\mathfrak o}_E$,  and let $\varpi_E$ be the uniformizer of ${\mathfrak P}$. Assume that if the residual characteristic of $L$ is even then $E/L$ is unramified. 

If $E$ is a field let $\tau_0$ be an infinite-dimensional, irreducible, admissible representation of $\gl(2,E)$ with trivial central character.  If $E = L\times L$ let $\tau_1$ and $\tau_2$ be infinite-dimensional, irreducible, admissible representations of $\gl(2,L)$ with trivial central character. For $i \in \{0,1,2\}$, we assume that the space of $\tau_i$ is its unique Whittaker model $\mathcal W(\tau_i)$. If $E = L\times L$ let $(\pi(\tau_1,\tau_2), V)$, where $V = \mathcal W(\tau_1) \times \mathcal W(\tau_2)$ is the associated $+$-representation of $\gso(X)$.  

If $E$ is a field let $(\pi(1, \tau_0), V)$, where $V = \mathcal W(\tau_0)$ is the associated $+$-representation of $\gso(X)$, and let $W\in V$ be a local $\gl(2,E)$-newform with $\Gamma_0({\mathfrak p}^n)$-invariance. If $E = L\times L$ then let $W_i \in V$ be local $\gl(2,L)$-newforms with $\Gamma_0({\mathfrak p}^{n_i})$-invariance, for $i \in \{1,2\}$ and set $W = (W_1, W_2)$.

If $E = L\times L$ then set $N = n_1 + n_2$ and let $\varphi \in \mathcal S(X^2)$ be given by as in Theorem \ref{MT_split},
if $E/L$ is inert then set $N = 2n$ and let $\varphi \in \mathcal S(X^2)$ be given as in Theorem \ref{MT_inert}, and 
if $E/L$ is tamely ramified then set $N = n + 2$, let $\chi$ be the non-trivial quadratic character of $E/L$, and let $\varphi \in \mathcal S(X^2)$ be given as in Theorem \ref{MT_ramified}.
Finally, let $s \in \mathbb{C}$ be such that $\Re(s) \gg 0$, and let $B(\cdot, \varphi, W, s): \gsp(4, L)\to \C$ be the local intertwining map as in \eqref{eq1}. 

Then $B(\cdot, \varphi, W, s): \gsp(4,L) \to \mathbb{C}$ is non-zero and is invariant under right translation by elements of $\mathrm{K}({\mathfrak p}^N)$. 
\end{localtheorem}
The compatibility of this local map with the global theta lift follows from Corollary \ref{cor:piplusmap_arch} and \cite{Roberts2001}. The remainder of the paper is devoted to the proof of of these results.
\section{Preliminaries}
\label{sec:notation}
In this section, we collect some results on local representations that will be needed in the sequel.
\subsection{Whittaker Models}
\label{sec:whittakermodels}

Let $\psi$ be a non-trivial additive character of $L$. If $L$ is non-archimedean we let $\mathcal{W}(\GL(2,L),\psi)$ be the $\mathbb{C}$-vector space of functions $W:\GL(2,L) \to \C$ such that 
\begin{align} \label{eq:whittaker_character}
W([\begin{smallmatrix}1&x\\ &1 \end{smallmatrix}] g) = \psi(x) W(g)
\end{align}
for $x \in L$ and $g \in \GL(2,L)$, and there exists a compact, open subgroup $K$ of $\GL(2,L)$ such that $W(gk)=W(g)$ for $g \in \GL(2,L)$ and $k \in K$. If $L = \mathbb{R}$ we let $\mathcal W (\gl(2,L) , \psi)$ be the space of smooth functions with rapid decay away from $0$ that satisfy \eqref{eq:whittaker_character}.  Evidently, $\mathcal{W}(\GL(2,L),\psi)$ is a smooth $\GL(2,L)$-representation under right translation. We say that a representation $(\tau, V)$ of $\GL(2,L)$ admits a \emph{Whittaker model}, denoted by $\mathcal W(\tau,\psi)$, if it is isomorphic to a subrepresentation of $\mathcal W(\GL(2,L),\psi)$. 
\begin{theorem}
	Let $L$ be a non-archimedean local field and let $\psi$ be a non-trivial additive character of $L$. Let $(\tau , V)$ be an irreducible admissible representation of $\GL(2,L)$. Then $\tau$ is infinite-dimensional if and only if $\tau$ admits a unique Whittaker model. 
\end{theorem}
\begin{proof}
	See Theorem 3.5.3 of \cite{bump}, for example. 
\end{proof}
We note that if $(\tau,V)$ is a finite-dimensional irreducible admissible representation of $\gl(2,L)$, then $\tau \cong \beta \circ \det$ for some character $\beta:L^\times \to \mathbb{C}^\times$. In the case that $L = \mathbb{R}$ we have a similar result.
\begin{theorem}
  Let $(\pi, V)$ be an irreducible admissible $({\mathfrak g} , K)$-module for $\gl(2,\mathbb{R})$. Then there exists at most one space $\mathcal W(\pi, \psi) \subset \mathcal W(\gl(2,\mathbb{R}) , \psi)$ that is invariant under the actions of $U({\mathfrak g})$ and $K$ on $C^\infty(\gl(2,\mathbb{R}))$ such that $\mathcal W(\pi,\psi)$ is isomorphic to $(\pi,V)$ as a $({\mathfrak g},K)$-module. 
\end{theorem}
\begin{proof}
  See Theorem 2.8.1 of \cite{bump}, for example. 
\end{proof}

\subsection{The split case}

Let $L$ be a local field. Define the following symmetric bilinear space of L:
\begin{equation*}
    X_M = \mathrm{M}(2, L).
\end{equation*}
Define an action of $\GL(2, L) \times \GL(2, L)$ on $X_M$ by $\rho(d_1,d_2)x=d_1xd_2^*$. Then $\rho(d_1,d_2) \in \GSO(X_M)$ for $d_1,d_2 \in \GL(2, L)$, and the sequence
\begin{equation}
\label{eq:Base_exact_sequence}
1 \to L^\times \to \GL(2,L) \times \GL(2, L) \xrightarrow{\ \ \rho\ \ } \GSO(X_M) \to 1
\end{equation}
is exact. Let $(\tau_i, \mathcal{W}(\tau_i, \psi))$ be the Whittaker model for an infinite-dimensional, irreducible, admissible representation of $\GL(2,L)$ for $i\in \{1, 2\}$. Assume that $\tau_1$ and $\tau_2$ have equal central character. Then $\pi = \rho(\tau_1, \tau_2)$ is an infinite-dimensional, irreducible, admissible representation of $\GSO(X_M)$, let $\sigma$ be the induced representation of $\GO(X_M)$, and let $\pi^+$ be a particular irreducible subprepresentation of $\sigma$. Let $W_i \in \mathcal W(\tau_i, \psi)$ for $i \in \{1, 2\}$. Let $s \in \mathbb{C}$ and set
\begin{align} \nonumber
	Z(s,W_i) &= \int\limits_{L^\times} W_i([\begin{smallmatrix}
		a & \\ & 1
	\end{smallmatrix}])|a|^{s-\frac 12} \, d^\times a.
	\intertext{We define $Z(s,W)$ by setting} \label{eq:split_zeta_def}
	Z(s, W) & = Z(s,W_1) \cdot Z(s, W_2)
\end{align}
for $W = W_1 \otimes W_2 \in V$.
\begin{lemma}
\label{zeta_convergence_1}
	There exists a positive real number $M$ such that $Z(s,W_i)$ converges absolutely, for all $W_i \in W(\tau_i, \psi)$, for $\Re(s) > M$. If $L$ is non-archimedean then $Z(s,W_i)$ converges to a rational function in $q^{-s}$. 
\end{lemma}
\begin{proof}
	See Theorem 6.12 and Remark 6.13 of \cite{gelbart}, for example.
  \end{proof}

\subsection{The non-split case}
Let $\delta \in L$ be square free and set $E = L(\sqrt{\delta})$. Define the symmetric bilinear space of $L$:
\begin{equation}\label{eq:X_standard_model}
	X_{ns} = \left \{ [\begin{smallmatrix}
			a & b \sqrt{ \delta}\\
			c \sqrt{\delta} & \alpha(a)
		\end{smallmatrix}] \mid a \in E, \ (b,c) \in L\times L \right \}.
\end{equation}
We define an action $\rho$ of $L^\times \times \GL(2,E)$ on $X$ by
$\rho(t,b)\cdot x = t^{-1}bx\alpha(b)^*$ for $t \in L^\times, b \in B^\times$ and $x \in X$. The sequence 
\begin{align}
\label{eq:extended_exact_sequence}
1 \to E^\times \to L^\times \times \GL(2,E) \xrightarrow{\ \ \rho\ \ } \gso(X_{ns}) \to 1
\end{align}
is exact. Let $(\tau, \mathcal{W}(\tau, \psi_E))$ be the Whittaker model for an infinite-dimensional, irreducible, admissible representation of $\GL(2,E)$ with trivial central character.  Then $\pi = \rho(1, \tau)$ is an infinte-dimensional, irreducible, admissible representation of $\GSO(X_{ns})$ and let $\sigma$ be the induced representation of $\GO(X_{ns})$, or a particular irreducible subprepresentation.  We let
\begin{equation}\label{eq:nonsplit_zeta_def}
	Z(s,W) = \int\limits_{E^\times} W([\begin{smallmatrix}
		a & \\ & 1
	\end{smallmatrix}])|a|_E^{s-\frac 12} \, d^\times a
\end{equation}
for $W \in \mathcal{W}(\tau, \psi_E)$ and $s \in \mathbb{C}$.
\begin{lemma}
\label{zeta_convergence_2}
	There exists a positive real number $M$ such that $Z(s,W)$ converges absolutely, for all $W \in W(\tau_0, \psi_E)$, for $\Re(s) > M$ to a rational function in $q_E^{-s}$.

\end{lemma}
\begin{proof}
	See Theorem 6.12 and Remark 6.13 of \cite{gelbart}, for example. 
  \end{proof}

We specify the following elements of $X$
\begin{equation}
	x_1 = \begin{cases}
		[\begin{smallmatrix}
			0 & 1\\ 0& 0
		\end{smallmatrix}] & $X$ \  \text{split}\\ 
		[\begin{smallmatrix}
			0 & \sqrt{\delta}\\ 0& 0
		\end{smallmatrix}] & $X$ \  \text{non-split}\\
	\end{cases} \qquad \text{and}, \qquad
	x_2 = \begin{cases}
		[\begin{smallmatrix}
			0 & 0\\ -2& 0
		\end{smallmatrix}] & $X$ \  \text{split}\\ 
		[\begin{smallmatrix}
			0 & 0\\ \frac{-2 \sqrt{\delta}}{\delta}& 0
		\end{smallmatrix}] & $X$ \  \text{non-split}.\\
	\end{cases}
\end{equation}
Let $H$ be the stabilizer in $\so(X)$ of $(x_1, x_2) \in X^2$. We make the following observation. 

\begin{lemma} \label{lemma:stab}
Let $H$ be as above. In the non-split case assume that $L$ is non-archimedean. Then,
	$H = \{ \rho(1, [\begin{smallmatrix}
		1&\\&u
	\end{smallmatrix}]) \mid u \in E^1\}$
and in the split case $H = \{ \rho([\begin{smallmatrix}
		a & \\ & 1
	\end{smallmatrix}], [\begin{smallmatrix}
		a^{-1} & \\& 1
	\end{smallmatrix}]) \mid  a \in L^\times\}$.
\end{lemma}

\begin{lemma}\label{zetaintegrals}
Let $H$ be as in Section \ref{sec:notation}, let $X = X_M$ or $X = X_{ns}$, and let $\tau_1,\tau_2$, and $\tau_0$ be as above.
\begin{enumerate}
	\item[(a)] Let $s \in \mathbb{C}$ with $\Re(s) \gg M$. Assume that $X = X_M$. If $W \in V$ and $t_1,t_2 \in L^\times$, then $Z(s, \pi(\rho([\begin{smallmatrix}
				t_1&\\&t_2
			\end{smallmatrix}] , 1))W) = |t_2/t_1|^{s - \frac 12}Z(s, W)$.
	\item[(b)] Let $s \in \mathbb{C}$ with $\Re(s) \gg M$. Assume that $X = X_{ns}$. If $W \in V$ and $t_1,t_2 \in E^\times$, then $Z(s, \pi( \rho(1, [\begin{smallmatrix}
				t_1&\\&t_2
			\end{smallmatrix}]))W) = |t_2/t_1|_E^{s - \frac 12}Z(s, W)$.
	\item[(c)] Let $s \in \mathbb{C}$ with $\Re(s) \gg M$, $W \in V$, and $h \in H$. Then, 
	\begin{equation}
		Z(s, \pi(h)W) = Z(s, W).
	\end{equation}  
\end{enumerate}
\end{lemma}
\begin{proof}
\begin{enumerate}
	\item[(a)] Let $t_1,t_2 \in L^\times$, let $\tau = \tau_i$ for $i \in \{1,2\}$ and let $W = W_1 \otimes W_2 \in V$. Then,
	\begin{align*}
		Z(s, \tau_2([\begin{smallmatrix}
			t_1&\\&t_2
		\end{smallmatrix}])W_2) & = \int_{L^\times} W_2([\begin{smallmatrix}
			a&\\&1
		\end{smallmatrix}][\begin{smallmatrix}
			t_1&\\&t_2
		\end{smallmatrix}]) |a|^{s - \frac 12} \,d^\times a\\
		& = \int_{L^\times} W_2([\begin{smallmatrix}
			a&\\&t_2
		\end{smallmatrix}])|a/t_1|^{s - \frac 12} \,d^\times a\\
		& = \int_{L^\times} W_2(t_2 [\begin{smallmatrix}
			a&\\&1
		\end{smallmatrix}]) |a|^{s - \frac 12}|t_2/t_1|^{s - \frac 12} \,d^\times a\\
		& = |t_2/t_1|^{s - \frac 12}Z(s, W_2).
	\end{align*}
	Therefore, $Z(s, \pi(\rho([\begin{smallmatrix}
			t_1&\\&t_2
		\end{smallmatrix}] , 1))W)  = Z(s, W_1) Z(s, \tau_2([\begin{smallmatrix}
			t_1&\\&t_2
		\end{smallmatrix}])W_2) = |t_2/t_1|^{s - \frac 12} Z(s, W)$.
	The result follows for a general $W \in V$.
	\item[(b)] The calculation is the same as the calculation in part (a), except the factor is $|t_2/t_1|_E^{s - \frac{1}{2}}$.
\item[(c)] The proof of (c) follows simply from part (a) and (b) and Lemma \ref{lemma:stab}.
\end{enumerate}
  \end{proof}

\begin{corollary}
	\label{locallyconstantzeta}
	Assume that $L$ is non-archimedean and fix a $W \in V$. The function $H\backslash \SO(X) \to \mathbb{C}$ defined by $Hh\mapsto Z(s,\pi(h)W)$ is well defined and locally constant.
\end{corollary}
\begin{proof}
	Part (b) of Lemma \ref{zetaintegrals} proves the map is well defined. Evidently, the smoothness of $\pi$ guarantees that the map is locally constant.
  \end{proof}

\subsection{Additional Results}
\label{sec:additional_results}

\begin{lemma}\label{simplegausschilemma}
Let $\psi$ be a character of $L^\times$ with conductor ${\mathfrak o}_L$. Let $dx$ be the Haar measure on $L$ so that the volume of $\OF$ is $1$. Let $c \in L^\times$. Then
	\begin{align*}
		\int\limits_{\varpi^{n}\OF} \psi(c \alpha)\, d \alpha & = \begin{cases}
			q^{-n} & \nu(c) \ge -n\\
			0 & \nu(c) < -n,
		\end{cases}
	\intertext{and}
		\int\limits_{\varpi^{n}\OF^\times} \psi(c \alpha)\, d \alpha & = \begin{cases}
			q^{-n} - q^{-n  - 1} & \nu(c) > -n - 1\\
			- q^{-n - 1} & \nu(c) = - n - 1\\
			0 & \nu(c) < -n - 1.
		\end{cases}
	\end{align*}
\end{lemma}

\begin{lemma}
	\label{lemma:paramodular_generators}
	Let $N$ be a non-negative integer. Then $\K({\mathfrak p}^N)$ is generated by the following:
	\begin{enumerate}
		\item[(a)] $[\begin{smallmatrix}
			A & 0\\ 0 & \lambda^t A^{-1}
		\end{smallmatrix}]$ for all $A \in \Gamma_0({\mathfrak p}^N)$ and $\lambda \in {\mathfrak o}_L^\times$,
		\item[(b)] $\left[\begin{smallmatrix}
			1&&\varpi^{-N}b_1&b_2\\
			&1&b_2&b_3\\
			&&1&\\
			&&&1
		\end{smallmatrix}\right]$ for $b_1,b_2,b_3 \in {\mathfrak o}_L$,
	\item[(c)] $s_2 = \left[\begin{smallmatrix}
		1&&&\\
		&&&1\\
		&&1&\\
		&-1&&
	\end{smallmatrix}\right]$, an element of the Weyl group, and

	\item[(d)] $t_N = \left[\begin{smallmatrix}
		&&\varpi^{-N}&\\
		&1&&\\
		-\varpi^N&&&\\
		&&&1
	\end{smallmatrix}\right]$.
	\end{enumerate}
\end{lemma}
\begin{proof} This follows from the Iwahori factorization and the decomposition of $\K({\mathfrak p}^N)/ \Kl({\mathfrak p})$. These results can be found in  \cite{Roberts_Schmidt2007} in (2.7) and Lemma 3.3.1, respectively. We reproduce the later in this text for easy reference.
\begin{equation} \label{eq:KKlcosets}
    \K({\mathfrak p}^N) = \bigsqcup_{u \in {\mathfrak o}_L/{\mathfrak p}^N} \left[\begin{smallmatrix}
      1&&u\varpi^{-N}&\\
      &1&&\\
      &&1&\\
      &&&1
    \end{smallmatrix}\right] \Kl({\mathfrak p}^N) \sqcup \bigsqcup_{o \in {\mathfrak o}_L/{\mathfrak p}^{N - 1}} t_N \left[\begin{smallmatrix}
      1&&u \varpi^{-N + 1}&\\
      &1&&\\
      &&1&\\
      &&&1
    \end{smallmatrix}\right]\Kl({\mathfrak p}^N)
\end{equation}
for $N$ a positive integer. Notice that if $N = 0$ that both sides are equal to $\gsp(4, {\mathfrak o})$.
  \end{proof}

We now demonstrate a more user friendly, but equivalent, Haar measure in the split and non-split cases. In this we will need the following simple observations
\begin{equation}
\label{eq:ip_expand_split}
	2 \langle [\begin{smallmatrix}
		x_1 & x_2 \\ x_3 & x_4
	\end{smallmatrix}], [\begin{smallmatrix}
		y_1 & y_2 \\ y_3 & y_4
	\end{smallmatrix}] \rangle = y_1x_4 - y_2x_3 - y_3x_2 + y_4x_1.
\end{equation}
\begin{equation}
\label{eq:ip_expand_nonsplit}
		2\langle [\begin{smallmatrix}
			x_1 + x_2\sqrt{\delta} & x_3 \\ x_4 & x_1 - x_2 \sqrt{\delta}
		\end{smallmatrix}] , [\begin{smallmatrix}
			y_1 + y_2 \sqrt{\delta} & y_3 \\ y_4 & y_1 - y_2 \sqrt{\delta} 
		\end{smallmatrix}] \rangle = 2x_1 y_1 - 2x_2y_2 \delta - x_3y_4 - x_4y_3.
	\end{equation}
\begin{lemma}
\label{Haar_measure_constant_lemma}
	If $y = \left[\begin{smallmatrix}
		y_1 & y_2 \\ y_3 & y_4
	\end{smallmatrix}\right] \in X_M = \mat(2,L)$ then $dy_1 \, dy_2 \, dy_3 \, dy_4$ is a Haar measure on $X_M$. Since $X_M$ is locally compact there must exist a non-zero constant $c$ such that $dy = c \cdot dy_1 \, dy_2 \, dy_3 \, dy_4$. The claim is that $c = 1$.

	Assume that $E = L(\sqrt{\delta})$ where $\delta \in L^\times$ is squarfree and let $X = X_{ns}$. We denote the maximal ideal of ${\mathfrak o}_E$ by ${\mathfrak P}$. If $y = \left[\begin{smallmatrix}
		y_1 + y_2 \sqrt{\delta} & y_3\\ y_4 & y_1 - y_2 \sqrt{\delta}
	\end{smallmatrix}\right]$ then $dy_1 \, dy_2 \, dy_3 \, dy_4$ is a Haar measure on $X_{ns}$. Since $X_{ns}$ is locally compact there must exist a positive constant $c$ such that $dy = c \cdot dy_1 \, dy_2 \, dy_3 \, dy_4$. The claim is that $c = |\delta|_{{\mathfrak p}} = q_E^{-\nu_L(\delta)}$.
\end{lemma}
\begin{proof}
	Assume that the residual characteristic of $L$ is odd. 

	To align with some future notation chose $\varphi_2 \in \mathcal S(X)$ to be
	\begin{equation*}
		\varphi_2 = \begin{cases}
			f_{M(2,{\mathfrak o}_L)}, & E/L \ \text{is split} \\
			f_{\left[\begin{smallmatrix}
				{\mathfrak o}_L & {\mathfrak P}^{-r}\\
				{\mathfrak P}^{r} & {\mathfrak o}_L
			\end{smallmatrix}\right] \cap X}, & E/L \ \text{is non-split, and} \ \nu_L(\delta) = \nu_E(\sqrt{\delta}) = r. 
		\end{cases}
	\end{equation*}

	We choose $dy$ so that $\mathcal F_1(\mathcal F_1(\varphi_2))(x) = \varphi_1(-x) = \varphi_2(x)$. On the other hand, we can compute this directly using Lemma \ref{simplegausschilemma} and \eqref{eq:ip_expand_split} or \eqref{eq:ip_expand_nonsplit}:
	\begin{align*}
		\mathcal F_1(\varphi_2)& = \int\limits_X \varphi_2(y) \psi(2 \langle x , y \rangle) \,dy = \int\limits_{[\begin{smallmatrix}
			{\mathfrak o}_L& {\mathfrak o}_L\\ {\mathfrak o}_L& {\mathfrak o}_L 
		\end{smallmatrix}]} \psi(2 \langle x , y \rangle) \,dy\\
		& =c \int\limits_{y_1 \in {\mathfrak o}_L} \psi(y_1x_4) \,dy_1 \int\limits_{y_2 \in {\mathfrak o}_L} \psi(y_2x_3) \,dy_2 \int\limits_{y_3 \in {\mathfrak o}_L} \psi(y_3x_2) \,dy_3 \int\limits_{y_4 \in {\mathfrak o}_L} \psi(y_4x_1) \,dy_4 = c \varphi_2(x),
  \end{align*}
\begin{align*}
		\mathcal F_1(\varphi_2)(x) &= \int\limits_{X} \varphi_2(y) \psi(2 \langle x , y \rangle \, dy = \int\limits_{[\begin{smallmatrix}
			{\mathfrak o}_E & {\mathfrak P}^{-r}\\
			{\mathfrak P}^r & {\mathfrak o}_E
		\end{smallmatrix}] \cap X} \psi(2 \langle x , y \rangle) \, dy\\
		& = c \int\limits_{{\mathfrak o}_E} \psi(2x_1y_1) \, dy_1 \int\limits_{{\mathfrak P}^{-r}} \psi(-2x_2y_2) \delta \,dy_2 \int\limits_{{\mathfrak P}^{-r}} \psi(-x_4y_3) \,dy_3 \int\limits_{{\mathfrak P}^r} \psi(-x_3y_4) \,dy_4= c q_E^r  \cdot\varphi_2(x).
	\end{align*}
	Therefore $\varphi_2(x) = \mathcal F_1(\mathcal F_1 \varphi_2)(x) = c^2 q_E^{2r} \varphi_2(x)$ and so $c = q_E^{-r}$, in particular $c$ is equal to 1 when $E/L$ is unramified.
  \end{proof}

\subsection{Weil Representations}
\label{sec:weil_reps}
Let $\psi:L \to \mathbb C^\times$ be a non-trivial continuous  unitary character. For $c \in L$ set the notation $\psi^c(x) = \psi(cx)$ for $x \in L$.  Assume $L$ is a local field. If $L$ is non-archimedean We define $\mathcal S(X)$ the space of Schwartz functions on $X$ to be locally constant and compactly supported. If $L$ is archimedean then $\mathcal S(X)$ will be functions that decay rapidly away from 0. The following formulas determine a unique Weil representation $(\omega_1, \mathcal S(X))$ of $\gl(2,L) \times \oo(X)$ with respect to $\psi$: 
\begin{align*}
		\big(\omega_1(1,h)\varphi\big)(x) &= \varphi(h^{-1}x),\\
		\big( \omega_1(
		    [\begin{smallmatrix} a & \\ & a^{-1} \end{smallmatrix}], 1) \varphi\big) (x) & = \chi_{E/L}(a)|a|^{\dim X/2} \varphi(ax),\\
		\big(\omega_1( [\begin{smallmatrix} 1  &b \\ &1 \end{smallmatrix}],1)\varphi \big)(x)
		    &=\psi(b\langle x,x \rangle) \varphi(x),\\
		\big(\omega_1([\begin{smallmatrix}&1  \\ -1& \end{smallmatrix}],1)\varphi\big)(x_1,x_2) & =\gamma_1(X) \big(\mathcal{F}_1(\varphi)\big)(x).
\end{align*}
Here $\gamma_1(X)$ is the fourth root of unity discussed in \cite{JacquetLanglandsgl2} and $\mathcal F_1(\varphi)$ is the Fourier transform of $\varphi$ given by 
\begin{equation}
\label{eq:sl2_fourier}
\mathcal{F}_1(\varphi)(x)=\int_{X} \varphi(y) \psi( 2 \langle x,y\rangle)\, dy,	
\end{equation}
where the Haar measure on $X$ is the unique Haar measure such that $\mathcal F_1(\mathcal F_1(\varphi))(x) = \varphi(-x)$ for $x \in X$ and all $\varphi \in \mathcal S(X)$. We can extend $\omega_1$ to all of $R = \{ (g,h) \in \gl(2,L) \otimes \go(X) \mid \det(g) = \lambda(h)\}$ via the formula

\begin{equation}
		\label{omega1extendeq}
		\omega_1(g,h) \varphi = |\lambda(h)|^{-1} \omega_1\big( g [\begin{smallmatrix} 1 & \\ & \det(g)^{-1} \end{smallmatrix}],1\big)  ( \varphi \circ h^{-1} ).
\end{equation} 

The following formulas determine a unique \emph{Weil representation} $\omega$ of $\SP(4,L) \times \OO(X)$ on  $\mathcal S^2(X^2) = \mathcal S^2(X \times X)$ with respect to $\psi$:
\begin{align}
		&\big(\omega(1,h)\varphi\big)(x_1,x_2) = \varphi(h^{-1}x_1,h^{-1}x_2),\label{weilrep1}\\
		&\big( \omega(
		    \begin{bmatrix} 
		    [\begin{smallmatrix} a_1 & a_2 \\ a_3 & a_4 \end{smallmatrix}]&  \\ 
		    &{[\begin{smallmatrix} a_1 & a_2 \\ a_3 & a_4 \end{smallmatrix}]^t}^{-1} 
		    \end{bmatrix}, 1) \varphi\big) (x_1,x_2)\label{weilrep2}\\ \nonumber
		&\qquad=\chi_{E/L}(a_1a_4-a_2a_3)|a_1a_4-a_2a_3|^{\dim X/2}
		    \varphi(a_1x_1+a_3x_2,a_2x_1+a_4x_2),\\
		&\big(\omega(\left[\begin{smallmatrix} 1 & & b_1 &b_2 \\ &1& b_2& b_3 \\&&1& \\ &&&1 \end{smallmatrix}\right],1)\varphi \big)(x_1,x_2)\label{weilrep3}\\ \nonumber
		&\qquad=\psi(b_1\langle x_1,x_1 \rangle+ 2b_2\langle x_1,x_2 \rangle  +b_3 \langle x_2,x_2 \rangle) \varphi(x_1,x_2),\\
		&\big(\omega(\left[\begin{smallmatrix}&&1& \\ &&&1 \\ -1&&& \\ &-1&& \end{smallmatrix}\right],1)\varphi\big)(x_1,x_2)  =\gamma(X) \mathcal{F}(\varphi)(x_1,x_2)\label{weilrep4}
		\end{align}
where $g \in \gsp(4), h \in \so(X), x_1,x_2 \in X,$ and $\varphi \in \mathcal L^2(X^2)$. Here $\chi_{E/L} : L^\times \to \C^\times$ is the quadratic character associated to the discriminant of $X$. That is, $\chi_{E/L}$ is the unique quadratic character on $L$ that is trivial on $\norm_{L}^{E}(E^\times)$. In particular if $\det(X) = 1 \in L^{\times}/L^{\times 2}$ then $\chi_{E/L}$ is trivial. Also, $\gamma(X)$ is a particular fourth root of unity as in remark 1 in \cite{Yoshida1979}. In \eqref{weilrep4} the Fourier transform $\mathcal F(\varphi)$ is defined by $\mathcal{F}(\varphi)(x_1,x_2)=\int_{X^2} \varphi(y_1,y_2) \psi( 2 \langle x_1,y_1\rangle +2  \langle x_2,y_2 \rangle)\, dy_1\,dy_2$,
where $dy_1\,dy_2$ is the unique Haar measure so that $\mathcal F( \mathcal F( \varphi ) )(x) = \varphi(-x)$. Let $R(L)$ be the subgroup of  $\GSp(4,L) \times \GO(X)$ given by
	\begin{equation}
		\label{eq:Rdef}
		R(L)=\{(g,h) \in \GSp(4,L) \times \GO(X): \lambda (g) = \lambda (h) \}.
	\end{equation}
 Then $\omega$ extends to $R(L)$ via the formula
	\begin{equation}
		\label{weilrep5}
		\omega(g,h) \varphi = |\lambda(h)|^{-\frac{\dim X}{2}} \omega \big( g \left[\begin{smallmatrix} 1 &&& \\ &1&& \\ &&\lambda(g)^{-1} & \\ &&& \lambda (g)^{-1} \end{smallmatrix}\right],1 \big)(\varphi \circ h^{-1}) 
	\end{equation}
for $\varphi \in \mathcal S^2(X^2)$ and $(g,h) \in R$. 

\begin{lemma}
\label{seesaw_lemma}
	the map 
	\begin{equation*}
		T: \mathcal S(X) \otimes \mathcal S(X) \to \mathcal S(X^2),
	\end{equation*}
	determined by the formula 
	\begin{equation*}
		T(\varphi_1 \otimes \varphi_2)(x_1 , x_2) = \varphi_1(x_1) \varphi_2(x_2)
	\end{equation*}
	for $\varphi_1$ and $\varphi_2$ in $\mathcal S(X)$ and $x_1$ and $x_2$ in $X$, is a well defined complex linear isomorphism such that 
	\begin{align}
		&T\circ (\omega_1\big([\begin{smallmatrix} a_1 &b_1 \\ c_1 & d_1 \end{smallmatrix}],h) \otimes \omega_1(]\begin{smallmatrix} a_2 &b_2 \\ c_2 & d_2 \end{smallmatrix}],h)\big) 
		=\omega(\left[\begin{smallmatrix} a_1&&b_1&\\&a_2&&b_2\\c_1&&d_1&\\&c_2&&d_2\end{smallmatrix}\right],h)\circ T \label{gsp4gl2eq3}
	\end{align}
  for $g_1 = [\begin{smallmatrix}
    a_1&b_1\\ c_1&d_1
  \end{smallmatrix}]$ and $g_2 = [\begin{smallmatrix}
    a_2&b_2\\ c_2&d_2
  \end{smallmatrix}]$ in $\gl(2,L)$ and $h \in \go(X)$ such that 
  \begin{equation*}
    \det(g_1) = \det(g_2) = \lambda(h).
  \end{equation*}
\end{lemma}
	\begin{proof}
		It is not hard to demonstrate that $T$ is an isomorphism. It suffices to prove \eqref{gsp4gl2eq3} holds for $(g_1 , g_2) = (g,g) \in \SL(2,L)\times \SL(2,L)$, $(g_1 , g_2) = (1,g) \in \SL(2,L)\times \SL(2,L)$ for $g$  a generator for $\SL(2,L)$, and $(g_1,g_2) = ([\begin{smallmatrix}
    1& \\ & \lambda  
    \end{smallmatrix}] , [\begin{smallmatrix}
    1&\\&\lambda  
    \end{smallmatrix}])$ and $h \in \GO(X)$ with $\lambda(h) = \lambda$. 
	  \end{proof}
\section{Intertwining Maps} \label{sec:intertwiningmaps}
In this section, we study the intertwining maps which define an explicit integral model of the local theta lift.

\subsection{The non-Archimedean Case}
Let the notation be as in Section \ref{sec:notation}. For this section, make the further assumption that $L$ is non-archimedean, though we will occasionally note when results and their proofs are identical in the case that $L = \mathbb{R}$. Let $R$ be as in Section \ref{sec:introduction}. Let $(\omega, \mathcal S(X^2))$ be the Weil representation of $R$ associated to $X$ and $\psi$ as in Section \ref{sec:weil_reps} and let $V$ be the space of the Whittaker models, as in Section \ref{sec:whittakermodels}.

The space $ \mathcal S(X^2) \otimes V $ is an  $R' = R\cap \big(\GSp(4,L) \times \GSO(X)\big)$-space with the action being determined by
	$(g , h) \cdot (\varphi \otimes v) = (\omega(g , h) \cdot \varphi) \otimes (\pi(h) \cdot v)$
for $(g,h) \in R', \varphi \in \mathcal S(X^2)$ and $v \in V$.
Similarly we find that $ \mathcal S(X^2) \otimes (V \times V) $ is also an $R$-space with the actions determined by
	$(g , h) \cdot (\varphi \otimes (v_1 , v_2)) = (\omega(g , h)\varphi) \otimes (\sigma(h) \cdot (v_1 , v_2))$
with $(g,h) \in R, \varphi \in \mathcal S(X^2)$ and $v_1,v_2 \in V$. 

Let $H, x_1, x_2$ be as in Section \ref{sec:notation}. Assume $s \in \mathbb{C}$ be such that $\Re(s) > M$. Assume that $X$ is split. Let $g \in \GSp(4,L)^{+} = \gsp(4,L)$ and $\varphi \in \mathcal S(X^2)$. Let $W \in V$ and let $Z(s,W)$ be as in \eqref{eq:split_zeta_def}. We define
\begin{equation}
\label{eq:splitbesselintegral}
	B(g, \varphi, W, s) = \int\limits_{H\backslash \SO(X)} \big(\omega(g, hh') \varphi \big)(x_1,x_2) Z(s, \pi(hh') W) \, dh
\end{equation}
where $h' \in \GSO(X)$ is chosen so that the similitude factor $\lambda(h') = \lambda(g)$. We use the Haar measures on $\so(X)$ and $H$ for which 
$\SO(X) \cap \rho(\GL(2,{\mathfrak o}_L) \times \GL(2, {\mathfrak o}_L))$ and $H\cap \rho(\GL(2,{\mathfrak o}_L) \times \GL(2, {\mathfrak o}_L))$ both have measure 1, and the  integral in \eqref{eq:splitbesselintegral} uses their quotient measure.

Assume that $X$ is non-split. Let $W \in V$ and let $Z(s,W)$ be as in \eqref{eq:nonsplit_zeta_def}. For $g\in \GSp(4,L)^{+}$ and $\varphi \in \mathcal S(X^2)$ we define  
\begin{equation}
	\label{eq:nonsplitbesselintegral}
	B(g, \varphi, W, s) = \int\limits_{H\backslash \SO(X)} \big(\omega(g, hh') \varphi \big)(x_1,x_2) Z(s, \pi(hh') W) \, dh
\end{equation}
where $h = \rho(t,h_0)$, $h' \in \GSO(X)$ is chosen so that similitude factor $\lambda(h') = \lambda(g)$. 
Furthermore we use the Haar measures on $\so(X)$ and $H$ for which 
$\SO(X) \cap \rho({\mathfrak o}_E^\times \times \GL(2, {\mathfrak o}_E))$ and $H\cap \rho({\mathfrak o}_E^\times \times \GL(2, {\mathfrak o}_E))$ both have measure 1 and the integral in \eqref{eq:nonsplitbesselintegral} uses their quotient measure. We will occasionally refer to $B(g,\varphi,W,s)$ as the \emph{Bessel integral}.

\begin{lemma} \label{lemma:Bessel_abs_convergence}
	Let $s \in \mathbb{C}$ be such that $\Re(s) > M$. For $g \in \GSp(4,L)^+, \varphi \in \mathcal S(X^2)$, and $W \in V$ the integral defining $B(g,\varphi,W,s)$ is well defined and converges absolutely. 
\end{lemma}
\begin{proof}
	Fix a $(g,h')\in R, \varphi \in \mathcal S(X^2),$ and $s \in \mathbb{C}$ so that $\Re(s) > M$. We consider the function $f: H\backslash \so(X) \to \mathbb{C}$ defined by $f(Hh) = (\omega(g , hh') \varphi) (x_1,x_2)$ for $h \in \so(X)$. We claim that $f$ is compactly supported and locally constant. To see this, we note that $f = f_1 \circ f_2$, where $f_2: H \backslash \so(X) \to \so(X)(x_1,x_2)$ is defined by $f_2(Hh) = (h^{-1} x_1, h^{-1}x_2)$ for $h \in \so(X)$, and $f_1: \so(X)(x_1,x_2) \to \mathbb{C}$ is defined by $f_1(x) = \omega(g,h') \varphi(x)$ for $x \in \so(X)(x_1,x_2) \subset X^2$. The map $f_2$ is a homeomorphism by 5.14 of \cite{Bernshtein_Zelevinskii_1976}. The map $f_1$ is compactly supported and locally constant because the map $X^2 \to \mathbb{C}$ defined by $x \mapsto (\omega(g,h') \varphi)(x)$ is compactly supported and $\so(X)(x_1,x_2)$ is a closed subset of $X^2$. Indeed, $\so(X)(x_1,x_2)$ is the inverse image of $[\begin{smallmatrix}
		0&1/2\\1/2&0
	\end{smallmatrix}]$ under the map $(z_1,z_2) \mapsto [\begin{smallmatrix}
		(z_1,z_1)& (z_1,z_2)\\
		(z_2,z_1) & (z_2,z_2)
	\end{smallmatrix}]$ for $(z_1,z_2) \in \so(X)(x_1,x_2)$. In particular the function $|f_1 \circ f_2|: h \mapsto |\omega(g,hh')\varphi(x_1,x_2)|$ is locally constant and compactly supported on $H\backslash \SO(X)$. From Lemma \ref{locallyconstantzeta} we have that $h \mapsto Z(\Re(s), |\pi(h)W|)$ is also locally constant. Therefore, for $(g,h') \in R, h \in H, \varphi \in \mathcal S(X^2), W \in V$ and $s\in \mathbb{C}$ so that $\Re(s)>M$ we have that
	\begin{equation*}
		\int\limits_{H\backslash \SO(X)} |\big(\omega(g, hh') \varphi \big)(x_1,x_2)| \cdot  Z(\Re(s), |\pi(hh') W|) \, dh
	\end{equation*}
	is finite. By the triangle inequality we have that integral defining $B(g,\varphi,W,s)$ converges absolutely. 
  \end{proof}

\begin{lemma}
\label{bessel_diag}
	Let $\varphi \in \mathcal S(X^2), W \in V$ and suppose that $t = \left[\begin{smallmatrix}
		t_1&&&\\
		&t_2&&\\
		&&t_2&\\
		&&&t_1
	\end{smallmatrix}\right] \in \gsp(4,L)^+$. Let $s \in \mathbb{C}$ such that $\Re(s) > M$ and set $B(\cdot) = B(\cdot, \varphi, W, s)$. Then we have the following transformation property for the Bessel integral
	\begin{equation*}
		B(tg) = |t_1/t_2|^{\frac{1}{2} - s} B(g)
	\end{equation*}
	for $g \in \gsp(4,L)^+$.
\end{lemma}
\begin{proof}
	First let us handle the case when $t =  z \cdot 1_4$. Then $\lambda(t) = z^2$. Set 
	\begin{align*}
		h_{z} = \begin{cases}
			\rho(z^{-1},1), & X = X_{ns}\\
			\rho(z,1), & X = X_M,
		\end{cases}  	
	\end{align*}  
	which acts as multiplication by $z$ and has similitude factor $\lambda(h_z) = z^2$. Using the formulas \eqref{weilrep2} and \eqref{weilrep5} we calculate that for $(y_1,y_2) \in X^2$,
	\begin{align*}
		\omega(z,h_z) \varphi'(y_1,y_2) & = |\lambda(h_z)|^{-2} \omega(z \cdot \left[\begin{smallmatrix}
			1&&&\\
			&1&&\\
			&&z^{-2}&\\
			&&&z^{-2}
		\end{smallmatrix}\right],1)\varphi'(\frac{1}{z} y_1, \frac{1}{z}y_2)\\
		& = |\lambda(h_z)|^{-2} \omega( \left[\begin{smallmatrix}
			z&&&\\
			&z&&\\
			&&z^{-1}&\\
			&&&z^{-1}
		\end{smallmatrix}\right],1)\varphi'(\frac{1}{z} y_1, \frac{1}{z}y_2)\\
		& = |\lambda(h_z)|^{-2} \chi_{E/L}(z^2) |z^2|^2 \varphi'(y_1,y_2)\\
		& = \varphi'(y_1,y_2)
	\end{align*}
	for all $\varphi' \in \mathcal S(X^2)$. We have $\pi(z)W = W$, for all $W \in V$, so that 
	\begin{align*}
		B(zg) & = \int\limits_{H\backslash \so(X)} (\omega(zg,hh_zh') \varphi)(x_1,x_2)  Z(s,\pi(hh') \big( \pi(z)W \big)) \, dh \\
		& = \int\limits_{H\backslash \so(X)} \omega(g,hh') \big( \omega(z,h_z) \varphi\big) (x_1,x_2)  Z(s,\pi(hh') \big( \pi(z)W \big)) \,dh\\
		& =  B(g)
	\end{align*}
	for all $g \in \gsp(4)^+$. For $g \in \gsp(4,L)^+$ we have 
	\begin{align*}
	 	B(\left[\begin{smallmatrix}
	 		t_1&&&\\&t_2&&\\&&t_2&\\&&&t_2
	 	\end{smallmatrix}\right]g) & = B(\left[\begin{smallmatrix}
	 			t_1&&&\\&t_2&&\\&&t_2&\\&&&t_2
	 	 	\end{smallmatrix}\right] \left[\begin{smallmatrix}
	 	 		t_1t_2^{-1}&&&\\&1&&\\&&1&\\&&&t_1t_2^{-1}
		 	\end{smallmatrix}\right] g)  B(\left[\begin{smallmatrix}
	 	 		t_1t_2^{-1}&&&\\&1&&\\&&1&\\&&&t_1t_2^{-1}
		 	\end{smallmatrix}\right] g).
	 \end{align*} 
	Thus, we may assume that $t_2 = 1$.
	
	We determine how $t$ acts on $\varphi$ and $Z(s,W)$, individually. We will need to treat the case when $X$ is split and when $X$ is non-split separately. Assume that $X$ is split. Let $g \in \gsp(4,L) = \gsp(4,L)^+$ and choose $h' \in \gso(X)$ such that $\lambda(h') = \lambda(g)$. Set $h_t = \rho([\begin{smallmatrix}
		t_1& \\&1
	\end{smallmatrix}],1)$. Then $\lambda(h_t) = t_1 = \lambda(t)$. Also, 
		$\omega(t,h_t)\varphi'(x_1,x_2) = |t_1|^{-2}\omega(\left[\begin{smallmatrix}
			t_1&&&\\&1&&\\&&t_1^{-1}&\\&&&1
		\end{smallmatrix}\right], 1)\varphi'(\frac{1}{t_1}x_1, x_2)\varphi'(x_1,x_2)$,	for any $\varphi' \in \mathcal S(X^2)$. By Lemma \ref{zetaintegrals} we have that
		$Z(s, \pi(h_t)W')  = |1/t_1|^{s - \frac{1}{2}} Z(s, W')$, for any $W' \in V$. 

	Assume that $X$ is non-split. Let $g \in \gsp(4,L)^+$ and let $h'\in \gso(X)$ with $\lambda(h') = \lambda(g)$. Suppose that $t = \left[\begin{smallmatrix}
		t_1&&&\\
		&1&&\\
		&&1&\\
		&&&t_1
	\end{smallmatrix}\right] \in \gsp(4,L)^+$. Then there exists some $a \in E$ such that $\norm_L^E(a) = \lambda(t) = t_1$. Set $h_t = \rho(1, [\begin{smallmatrix}
		a&\\&1
	\end{smallmatrix}])$ and notice that $\lambda(h_t) = \norm_L^E(a) = \lambda(t).$ We calculate that
	\begin{align*}
	 	\omega(\left[\begin{smallmatrix}
	 		t_1&&&\\
			&1&&\\
			&&1&\\
			&&&t_1
	 	\end{smallmatrix}\right], h_t) \varphi'(x_1,x_2) & = |t_1|^{-2} \omega(\left[\begin{smallmatrix}
	 		t_1&&&\\&1&&\\&&t_1^{-1}&\\&&&1
	 	\end{smallmatrix}\right], 1) \varphi'(\norm_L^E(a^{-1}) x_1, x_2)\\
	 	& = \chi_{E/L}(t_1) |t_1|^{2}|t_1|^{-2} \varphi'(t_1 \norm_E^L(a^{-1}) x_1, x_2)\\
		& = \varphi'(x_1,x_2),
	\end{align*} 
	for any $\varphi' \in \mathcal S(X^2)$. From Lemma \ref{zetaintegrals} we know that 
		$Z(s, \pi(h_t)W') = |1/t_1|^{\frac{1}{2} - s}Z(s,W')$, for any $W' \in V$. 
	
	We have determined that $\omega(t,h_t)\varphi(x_1,x_2) = \varphi(x_1,x_2)$ and that $Z(s,\pi(h_t)W) = |1/t_1|^{s - \frac{1}{2}}Z(s, W)$ when $X$ is split or non-split. Hence, we find that for any $X$
	\begin{align*}
		B(tg) & = \int\limits_{H \backslash \so(X)}(\omega(tg,hh'h_t) \varphi)(x_1,x_2) Z(s,\pi(hh'h_t) W) \, dh\\
		& = \int\limits_{H \backslash \so(X)} \omega(t,h_t) (\omega(g, h_t^{-1}hh'h_t)\varphi) (x_1,x_2) Z(s, \pi(hh'h_t) W) \, dh \\
		& = \int\limits_{H \backslash \so(X)} \omega(t,h_t) (\omega(g, h_t^{-1}hh_th_t^{-1}h'h_t) \varphi) (x_1,x_2) Z(s, \pi(hh'h_t) W) \, dh \\
		& = \int\limits_{H \backslash \so(X)}  (\omega(g, hh_t^{-1}h'h_t) \varphi) (x_1,x_2) Z(s, \pi(h_thh_t^{-1}h'h_t) W) \, dh \\
		& = |1/t_1|^{\frac{1}{2} - s} \int\limits_{H \backslash \so(X)} \omega(g, hh^{-1}_t h'h_t)\varphi(x_1,x_2) Z(s, \pi(hh_t^{-1}h'h_t)W) \, dh\\
		& = |1/t_1|^{\frac{1}{2} - s} B(g) 
		% B(tg) & = \int\limits_{H\backslash \so(X)} \omega(g,hh') \big( \omega(t,h_t) \varphi(x_1,x_2) \big) Z(s,\pi(hh') \big( \pi(h_t)W \big))\\
		% & = |t_1|^{\frac{1}{2} - s}B(g).
	\end{align*}
	since $\lambda(h') = \lambda(h_t^{-1} h'h_t)$. Here we have used that $\int_{H\backslash \SO(X)} f(h_t^{-1} x h_t) \,dx = \int_{H\backslash \SO(X)} f(x) \,dx$, for all measurable function on $H\backslash \so(X)$. This completes the proof. 
  \end{proof}
\begin{corollary}\label{cor:Bessel_Extension}
	Let $X$ be non-split. Let $\varphi \in \mathcal S(X^2), W \in V, g \in\gsp(4,L)^+$ and let $g_1$ be as in \eqref{eq:g0}. Let $s \in \mathbb{C}$ be such that $\Re(s) > M$. Then $B(g, \varphi, W, s) = |\lambda(g)|^{-s + \frac{1}{2}} B(g_1, \varphi, W, s).$ 
\end{corollary}

We extend $B(\cdot) = B(\cdot, \varphi,W,s)$  to all of $\GSp(4,L)$ via the formula 
\begin{equation}\label{eq:B_extend}
	B(g) = |\lambda(g)|^{-s + \frac 12}B(g_1)
\end{equation}
for
\begin{equation}
\label{eq:g0}
	g_1 = \left[\begin{smallmatrix}
		\lambda(g)^{-1} &  &  & \\
		  & 1 &  & \\
		  &  & 1 & \\
		  &  & & \lambda(g)^{-1}
	\end{smallmatrix}\right]g
\end{equation}
and $g \in \GSp(4,L)$. Clearly, Corollary \ref{cor:Bessel_Extension} indicates that this extension is well defined.

Part \textit{(a)} of the following lemma uses the choice of the canonical $\GO(X)$-representation $\pi^+$ made in Section \ref{sec:notation}. Let $s\in \OO(X)$ be defined by $s(x) = x^*$ and $T:V \to V$ be defined by $T(v_1, v_2) = (v_2, v_1)$ in the split case and by $T(W)(g) = W(\alpha (g))$ for $g \in \GL(2,E)$, $W \in \mathcal W(\tau ,\psi_E)$, and $\alpha$ is the non-trivial element of the Galois group of $E/L$ in the non-split case. 

\begin{lemma}
\label{bessle_trasformation_lemma}
	Let $(\pi, V)$ be the infinite-dimensional, irreducible, admissible representation of $\gso(X)$ obtained from either $\tau_0$ or the pair $\tau_1,\tau_2$, as in Section \ref{sec:notation}. Assume that the space $V$ of $\pi$ is either $\mathcal W(\tau_0, \psi_E)$ or $\mathcal W(\tau_1, \psi) \otimes \mathcal W(\tau_2, \psi)$, respectively. Let $z \in \mathbb{C}$ be such that $\Re(z) > M$. Then, for all $g \in \gsp(4,L) \varphi' \in \mathcal S(X^2), W' \in V$ we have
	\begin{enumerate}
		\item[(a)]   $B(g, \omega (1 , s) \varphi',  T(W'), z) = B(g, \varphi' ,W', z)$,
		\item[(b)] if $b = [\begin{smallmatrix}
			1 & B \\ & 1
		\end{smallmatrix}]$ where $B = [\begin{smallmatrix}
			b_1 & b_2\\ b_2& b_3
		\end{smallmatrix}] \in \mat (2,L)$, we have $B(bg, \varphi', W' , z) = \psi(b_2) B(g, \varphi', W', z)$, and
		\item[(c)] if $t = \left[\begin{smallmatrix}
			t_1&&&\\&t_2&&\\&&t_2&\\&&&t_1
		\end{smallmatrix}\right]$ for $t_1,t_2 \in L^\times$ we have $B(tg, \varphi', W', z) = |t_2/t_1|^{z - \frac 12} B(g, \varphi', W', z)$.
	\end{enumerate}
	
\end{lemma}
\begin{remark} \label{remark:arch}
  In Section \ref{sec:intertwining_maps_archimedean} we will define a similar intertwining map in the case that $L = \mathbb{R}$. In order to refrain from reproducing the same work multiple times we note that Lemma \ref{bessle_trasformation_lemma} will have an identical proof to the analogous result.
\end{remark}
\begin{proof}
	We first prove (a). Let $\varphi' \in \mathcal S(X^2) , W' \in V$. First see that $Z(z, (s\cdot \pi)(h), W') = Z(z, \pi(h)T(W'))$ for all $h \in \gso(X)$ and $W' \in V$. This is because in the split case $\pi(s) \rho(h_1,h_2) = \rho(h_2, h_1)$, for $h_1,h_2 \in \gl(2,L)$ and in the non-split case $\pi(s) \rho(t,h) = \rho(t, h^*)$ for $t \in L^\times$ and $h \in \gl(2,E)$. Suppose that $g\in \gsp(4,L)^+$.  We calculate that
	\begin{align*}
		B(g, \omega(1 , s) \varphi', W', z)
		 & = \int\limits_{H\backslash\SO(X)}(\omega(g , hh')\omega(1 , s) \varphi')(x_1 , x_2)Z(z , \pi(hh')  W') \, dh \\
		& = \int\limits_{H\backslash\SO(X)} (\omega(g, s s^{-1} h ss^{-1} h's) \varphi')(x_1 , x_2)Z(z , \pi(hh')    W') \, dh \\
		& = \int\limits_{H\backslash\SO(X)} (\omega(1 , s)\omega(g,  s^{-1} h ss^{-1} h's)   \varphi')(x_1 , x_2)Z(z , \pi(hh')    W') \, dh \\
		& = \int\limits_{H\backslash\SO(X)} (\omega(g , s^{-1} h ss^{-1} h's)   \varphi')(s    x_1 ,s    x_2)Z(z , \pi(hh')    W') \, dh \\
		& = \int\limits_{H\backslash\SO(X)} (\omega(g , s^{-1} h ss^{-1} h's)   \varphi')(- x_1 ,- x_2)Z(z , \pi(hh')    W') \, dh. \\
		\intertext{As in the proof of Lemma \ref{bessel_diag} we change variables $h \mapsto shs^{-1}$. Also, we replace $h'$ with $h'' = s h' s^{-1}$. We may write that the above is}
		& = \int\limits_{H\backslash\SO(X)} (\omega(g , hh'')   \varphi'(-x_1 , -x_2) Z(z , s\cdot\pi(hh'')W')  \,dh.
		\intertext{Denote by $-1 \in \SO(X)$ the action which sends $x$ to $-x$ so that we have $-1 = \rho(-1 , 1)$. Therefore the above is}
		& =  \int\limits_{H\backslash\SO(X)} (\omega(g , hh''(-1))   \varphi'(x_1 , x_2) Z(z , s\cdot\pi(hh'')W') \,dh\\
		& = \int\limits_{H\backslash\SO(X)} (\omega(g , hh'')   \varphi'(x_1 , x_2) Z(z , s\cdot\pi(-hh'')W')  \,dh\\
		& = \int\limits_{H\backslash\SO(X)} (\omega(g , hh'')   \varphi'(x_1 , x_2) Z(z , s\cdot \pi(hh'')W')  \,dh\\
		& =  B(g, \varphi', T(W'), z).
	\end{align*}	

	Hence we have shown that 
	\begin{equation*}
		B(g, \omega (1 , s)  \varphi',  W', z) = B(g,  \varphi',  T(W'), z)
	\end{equation*}
	so that
	\begin{equation*}
		B(g, \omega (1 , s)  \varphi',  T(W'), z) = B(g,  \varphi', W', z).
	\end{equation*}

	Now for all $g \in \gsp(4,L), \varphi' \in \mathcal S(X^2),$ and $W' \in V$ we use \eqref{eq:g0} to see that
	\begin{align*}
		B(g, \omega(1,s) \varphi',T(W),z) & = |\lambda(g)|^{\frac{1}{2} - z} B(g_1, \omega(1,s) \varphi', T(W), z)\\
		& = |\lambda(g)|^{\frac{1}{2} - z}  B(g_1, \varphi', W', z)\\
		& =  B(g, \varphi', W', z)
	\end{align*}
	which proves \textit{(a)}.

	\textit{Proof of (b).} Set $b = [\begin{smallmatrix}
			1 & B \\ & 1
		\end{smallmatrix}]$ where $B = [\begin{smallmatrix}
			b_1 & b_2\\ b_2& b_3
		\end{smallmatrix}] \in \mat (2,L)$ is symmetric. First notice that $\lambda(b) = 1$ so that $B(bg, \varphi', W', z)$  only differs from $B(g, \varphi', W', z)$ in the Weil representation factor of the integrand. Recall formula \eqref{weilrep3} to see that 
	\begin{align*}
		(\omega(b,1)\varphi')(x_1,x_2) & = \psi(b_1 \langle x_1, x_1 \rangle) + 2b_2 \langle x_1,x_2 \rangle + b_3\langle x_2, x_2 \rangle) \, \varphi'(x_1,x_2)\\
		& = \psi(b_2) \varphi'(x_1,x_2)
		\intertext{for all $\varphi' \in \mathcal S(X^2)$. Therefore it is evident that if $g \in \gsp(4,L)^+$ then by \eqref{eq:g0} }
		B(bg, \varphi', W' , z) & = \psi(b_2) B(g, \varphi', W', z).
		\intertext{If $g \in \gsp(4,L)$ then}
		B(bg)& = |\lambda(g)|^{\frac{1}{2} - s} B(bg_1)\\
		& = |\lambda(g)|^{\frac{1}{2} - s} \psi(b_2) B(g_1)\\
		& = \psi(b_2) B(g).
	\end{align*}
	This finishes the proof of \textit{(b)}.

	\textit{Proof of (c)} The case when $g\in \gsp(4,L)^+$ is proved in Lemma \ref{bessel_diag}. Assume that $g \in \gsp(4,L)$ and set 
	\[
		t = \left[\begin{smallmatrix}
			t_1&&&\\
			&t_2&&\\
			&&t_2&\\
			&&&t_1
		\end{smallmatrix}\right].
	\] Then $\lambda(t) = t_1t_2$. Similar to \eqref{eq:g0} set
	\[
		(tg)_1 = \left[\begin{smallmatrix}
			\lambda(tg)^{-1}&&&\\
			&1&&\\
			&&1&\\
			&&&\lambda(tg)^{-1}
		\end{smallmatrix}\right] tg = \left[\begin{smallmatrix}
			t_2^{-1} &&&\\
			&t_2&&\\
			&&t_2&\\
			&&&t_2^{-1}
		\end{smallmatrix}\right]g_1.
	\]
	Therefore,
	\begin{align*}
		B(tg) & = |\lambda(tg)|^{\frac{1}{2} - s} B((tg)_1)\\
		& = |\lambda(g)|^{\frac{1}{2} - s}|t_1t_2|^{\frac{1}{2} - s} B(\left[\begin{smallmatrix}
			t_2^{-1} &&&\\
			&t_2&&\\
			&&t_2&\\
			&&&t_2^{-1}
		\end{smallmatrix}\right]g_1).
  \end{align*}
	Finally, since $\left[\begin{smallmatrix}
			t_2^{-1} &&&\\
			&t_2&&\\
			&&t_2&\\
			&&&t_2^{-1}
		\end{smallmatrix}\right]g_1 \in \SSp(4,L)$ we get that 
	$B(tg)  = |t_1/t_2|^{\frac{1}{2} - s} |\lambda(g)|^{\frac{1}{2} - s} B(g_1)$, which is just $|t_1/t_2|^{\frac{1}{2} - s} B(g)$,
	completing the proof. 
  \end{proof}

\begin{proposition} \label{lemma:Bessel_nonzero} 
	Assume that $L$ is non-archimedean. Let $s \in \mathbb{C}$ be such that $\Re(s) \gg 1$. Assume that $X$ is split. Suppose that $W_i \in \mathcal W(\tau_i, \psi)$ are such that $Z(s,W_i)$ are not zero for $i \in \{1,2\}$. Set $W = W_1\otimes W_2$. Then there exists $\varphi \in \mathcal S(X^2)$ such that $B(1,\varphi,W,s) \ne 0$. 
	Assume that $X$ is non-split. Suppose that $W \in \mathcal W(\tau_0 ,\psi_E)$ is such that $Z(s,W)$ is not zero, then there exists $\varphi \in \mathcal S(X^2)$ such that $B(1,\varphi,W,s) \ne 0$. 
\end{proposition}
\begin{proof}
	Assume that $X$ is non-split.  Denote $h = \rho(t,h_0) \in \GSO(X)$, for $t \in L^\times$ and $h_0 \in \gl(2,E)$. By Lemma \ref{locallyconstantzeta} the function $h \mapsto Z(s, \pi(h)W)$ is locally constant on $H\backslash\SO(X)$ and, furthermore, is nonzero at the point $H\cdot 1$. Recall that $f_2: H\backslash \SO(X) \to \SO(X)(x_1,x_2)$, defined by $f_2(Hh) = h^{-1}(x_1,x_2)$, is a homeomorphism and $\SO(X)(x_1,x_2)$ is closed. Let $A \subset H\backslash \SO(X)$ be a compact open neighborhood of $H\cdot 1$ so that function $h \mapsto Z(s, \pi(h)W)$ is constant on $A$. Now, $f_2(A)$ is a compact open subset of $\SO(X)(x_1,x_2)$ and there exists some open $U \subset X^2$ such that $f_2(A) = U \cap \SO(X)(x_1,x_2)$. For each $u \in U$ choose an open compact neighborhood $Y_u$ of $u$ such that $Y_u \subset U$. Clearly $\{Y_u \cap \SO(X)(x_1,x_2)\}_{u \in U}$ is an open cover of $f_2(A)$ so there exists a finite subset $I \subset U$ such that $\{Y_u\cap \SO(X)(x_1,x_2)\}_{u \in I}$ is a finite cover of $f_2(A)$. Therefore $f_2(A) = \cup_{u \in I} (Y_u \cap \SO(X)(x_1,x_2))$.  Set $Y  = \bigcup_{u \in I} Y_u$. Then the characteristic function $\varphi = \mathbf{1}_{Y}$ is in $\mathcal S(X^2)$ and fulfills the requirements. Indeed, we calculate that
	\begin{align*}
	 	B(1,\mathbf{1}_Y,W,s) & = \int\limits_{H\backslash \SO(X)} \omega(1,h) \mathbf{1}_Y(x_1,x_2) Z(s, \tau_0(h_0)W) \, dh\\
	 	& = \vol(A) Z(s,W).
	 \end{align*}

	Assume that $X$ is split. The argument is similar. Let $h = \rho(h_1,h_2)$ for $h_1, h_2 \in \gl(2,L)$. Choose compact neighborhood $A$, of $1$ so that the functions $h \mapsto Z(s,\tau_1(h_1)W_1)$ and $ h \mapsto Z(s,\tau_1(h_2)W_2)$ are both constant on $A$, so that $h \mapsto Z(s, \pi(h)W)$ is also constant on $A$, and the argument runs verbatim. 
  \end{proof}

Now that we have proved that $B(\cdot, \varphi, W, s)$ is a well defined and non-zero map, for some choices of $\varphi, W$ and $s$ we define $\Theta(V)$ to be the space generated by all such functions for all choices of $\varphi \in \mathcal S(X^2)$ and $W \in V$. Fix $s \in \mathbb{C}$ such that $\Re(s) \gg M$. We define the map
\begin{align}\label{eq:localthetalift}
	\vartheta: \mathcal S(X^2) \otimes V &\to \Theta(V)\\
	\varphi \otimes W & \mapsto B(\cdot, \varphi, W, s). \nonumber
\end{align}
It is easy to verify that if $B \in \Theta(V)$, then by Lemma \ref{bessle_trasformation_lemma} 
\begin{align}\label{eq:extended_bessel1}
B(tg)  &= |t_1/t_2|^{\frac{1}{2} - s}B(g)	\quad \text{for $t = \left[\begin{smallmatrix}
	t_1&&&\\&t_2&&\\&&t_2&\\&&&t_1
\end{smallmatrix}\right]$ for $t_1,t_2 \in L^\times$ and}\\
\label{eq:extended_bessel2}
B(bg) & = \psi(b_2)B(g)\quad \text{for $b =[\begin{smallmatrix}
	1 & B \\ &1
\end{smallmatrix}]$ where $B = [\begin{smallmatrix}
	b_1 & b_2 \\ b_2 & b_3
\end{smallmatrix}] \in \mat(2,L)$.}  
\end{align}

\begin{theorem} \label{Besselrepresentation}
	Let $s \in \mathbb{C}$ with $\Re(s) \gg M$. Let $V = \mathcal W(\tau_1, \psi)\otimes \mathcal W(\tau_2, \psi)$ or $V = \mathcal W(\tau_0, \psi_E)$ in the split case or non-split case, respectively. Then
	\begin{enumerate}
		\item[(a)] the map $\vartheta$ from \eqref{eq:localthetalift} is a non-zero $R'$-map,
		\item[(b)] for each $f \in \Theta(V)$ there is an open and compact $K \subset \gsp(4,L)$ such that, for each $k \in K$ we have $f(g) = f(gk)$ for all $g \in \gsp(4,L)$, and
		\item[(c)] the image of $\vartheta$ lies inside $\mathcal B(\gsp(4,L) , \psi)$.
	\end{enumerate}
	It then follows that $\Theta(V)$ is a smooth representation of $\gsp(4,L)$ sitting inside $\mathcal B(\gsp(4,L), \psi)$.

\end{theorem}
\begin{proof} 
	\textit{Proof of (a).} Let $(g_0, h_0) \in R'$, let $g \in \gsp(4,L)^+$,  and let $h' \in \go(X)$ so that $(g,h')\in R$. We see that
	\begin{align*}
		B(g, \omega(g_0,h_0) \varphi, \pi(h_0) W, s) & = \int\limits_{H \backslash \so(X)} \omega(gg_0, h(h'h_0)) \varphi(x_1,x_2) Z(s, \pi(h (h'h_0))W) \, dh\\
		& = B(gg_0, \varphi, W, s)
	\end{align*}
	since $(gg_0, h'h_0) \in R$. Now suppose that $g \in \gsp(4,L)$ and let $g_1 \in \gsp(4,L)^+$ be as in \eqref{eq:g0}. Then using the above result, we find that 
	\begin{align*}
		B(g,  \omega(g_0, h_0) \varphi, \pi(h_0) W, s) 
		& = |\lambda(g)|^{-s + \frac{1}{2}} B(g_1,  \omega(g_0, h_0) \varphi, \pi(h_0) W, s) \\
		& = |\lambda(g)|^{-s + \frac{1}{2}} B(g_1g_0, \varphi,  W, s) \\
		& = B(gg_0, \varphi,  W, s).
	\end{align*}
	By Lemma \ref{lemma:Bessel_nonzero} we have that $B(\cdot)$ is nonzero as long as there is some $W \in V$ such that $Z(s, W)$ is not zero for $s \gg M$. Since $\tau_0, \tau_1, \tau_2$ have Whittaker models it is clear that $Z(s,W)$ is not zero. 

	\textit{Proof of (b).} Let $\varphi \otimes W \in \mathcal S(X^2) \otimes V$. It suffices to prove $(b)$ for $B = B(\cdot, \varphi, W, s)$. Given the following exact sequence
	\[
		1 \xrightarrow{\ \ \ } L^\times \cdot \SSp(4,L) \hookrightarrow \gsp(4,L) \xrightarrow{\ \ \lambda \ \ } L^\times / L^{\times 2} \xrightarrow{ \ \ \ } 1
	\]
	it suffices to find a compact open subgroup of $\SSp(4,L)$ which stabilizes $B(\cdot)$. Since the Weil representation is smooth there exists some $N \in \mathbb{Z}_{>0}$ so that the open and compact full congruence subgroup 
	\begin{equation*}
		\Gamma({\mathfrak p}^N) = \{ k \in \Sp(4,{\mathfrak o}_L) \mid k\equiv 1_4 \pmod{{\mathfrak p}^N} \}
	\end{equation*}
	fixes $\varphi$ under the action of the Weil representation. Since $\Gamma({\mathfrak p}^N)$ is also contained in $\SSp(4,L)$ we determine that for all $k \in \Gamma({\mathfrak p}^N)$ and all $g \in \gsp(4,L)$  that
	\begin{align*}
		B(gk) & = |\lambda(g)|^{-s + \frac{1}{2}}\int \limits_{H \backslash \so(X)} \omega(g_1k, hh') \varphi(x_1,x_2) Z(s, \pi(hh')W) \, dh \\
		& = |\lambda(g)|^{-s + \frac{1}{2}}\int \limits_{H \backslash \so(X)} \omega(g_1, hh') \big(\omega(k,1)\varphi(x_1,x_2)) Z(s, \pi(hh')W) \, dh \\
		& = |\lambda(g)|^{-s + \frac{1}{2}}\int \limits_{H \backslash \so(X)} \omega(g_1, hh') \varphi(x_1,x_2) Z(s, \pi(hh')W) \, dh \\
		& = B(g).
	\end{align*}
	\textit{Proof of (c)} This follows from part (b) of this lemma and from \eqref{eq:extended_bessel1} and \eqref{eq:extended_bessel2}. 
  \end{proof}

\begin{lemma}\label{lemma:rho_maximal_compact} Assume that the residual characteristic of $L$ is odd. Let $h \in \so(X)$. If $X = X_{ns}$ is non-split and $E/L$ is unramified then, $h \in \rho ({\mathfrak o}_L^\times \times \GL(2,{\mathfrak o}_{E}))H$ if and only if $h(x_1,x_2) \in [\mat(2, {\mathfrak o}_{E}) \cap X_{ns}]^2$. If $X = X_M$ is split then, $h \in \rho(\gl(2 ,{\mathfrak o}_L), \gl(2 ,{\mathfrak o}_L))H$ if and only if $h(x_1,x_2) \in \mat(2, {\mathfrak o}_L)^2$.
\end{lemma}
\begin{proof} In both the split and non-split case the forward direction is clear. 

Assume that $X = X_{ns}$ is non-split and that $E/L$ is unramified. 
	For the converse notice that for any $[\begin{smallmatrix}
		a&b\\c&d
	\end{smallmatrix}] \in \GL(2,E)$ we can find $[\begin{smallmatrix}
		a'&b'\\c'&d'
	\end{smallmatrix}] \in \GL(2,{\mathfrak o}_E)$ and $[\begin{smallmatrix}
		x&y\\0&z
	\end{smallmatrix}] \in \GL(2,E)$ so that
		$[\begin{smallmatrix}
		a&b\\c&d
	\end{smallmatrix}] = [\begin{smallmatrix}
		a'&b'\\c'&d'
	\end{smallmatrix}] [\begin{smallmatrix}
		x&y\\0&z
	\end{smallmatrix}]=  [\begin{smallmatrix}
		a'&b'\\c'&d'
	\end{smallmatrix}] [\begin{smallmatrix}
		x/z&y/z\\0&1
	\end{smallmatrix}][\begin{smallmatrix}
		z&\\&z
	\end{smallmatrix}]$.
	Set $h = \rho(t, [\begin{smallmatrix}
		a&b\\c&d
	\end{smallmatrix}])$, for some $t \in L$ and $[\begin{smallmatrix}
		a & b \\ c & d
	\end{smallmatrix}] \in \gl(2,E)$ for which $h(x_i) \in X_{ns}\cap \mat(2,{\mathfrak o}_E)$ for $i \in \{1,2\}$. 
 
 Since $\rho(\norm_L^E(z), [\begin{smallmatrix}
		z&\\&z
	\end{smallmatrix}]) = 1$ for any $z \in E^\times$ we have that 
	 	$h  = \rho(t \norm_L^E(z)^{-1},1)\rho(1,[\begin{smallmatrix}
		a'&b'\\c'&d'
	\end{smallmatrix}] [\begin{smallmatrix}
		x&y\\0&1
	\end{smallmatrix}])= \rho(u, [\begin{smallmatrix}
		a'&b'\\c'&d'
	\end{smallmatrix}]) \rho(\varpi^k, [\begin{smallmatrix}
		x&y\\0&1
	\end{smallmatrix}])$,
	for some $u \in {\mathfrak o}_E^\times$ and $k \in \mathbb{Z}$. Therefore, $\rho(\varpi^k, [\begin{smallmatrix}
		x&y\\0&1
	\end{smallmatrix}])(x_i) \in X_{ns} \cap \mat(2,{\mathfrak o}_E)$ for $i =1,2$. 
 
 It suffices to show that $k = 0, x\in {\mathfrak o}_E^\times$ and $y \in {\mathfrak o}_E$. Since $h \in \SO(X_{ns})$ we have that $1 = |\lambda(h)| = |\varpi^{-2k} \norm_L^E(x)|$ so that $\norm_L^E(x) = |\varpi|^{2k}$. Since $E/L$ is unramified, we know $\sqrt{\delta} \in {\mathfrak o}_E^\times$. Furthermore, we have that
		$$\rho(\varpi^k, [\begin{smallmatrix}
			x & y\\0&1
		\end{smallmatrix}])(x_1) = \varpi^{-k}[\begin{smallmatrix}
			0&\norm_L^E(x)\sqrt{\delta}\\0&0
		\end{smallmatrix}] \quad\text{and}\quad\rho(\varpi^k, [\begin{smallmatrix}
			x & y\\0&1
		\end{smallmatrix}])(x_2)  = 2\varpi^{-k}/\delta [\begin{smallmatrix}
			-y\sqrt {\delta} &\norm_L^E(y)\sqrt{\delta}\\
			-\sqrt{\delta} & \alpha(y)\sqrt(\delta)
		\end{smallmatrix}]$$
	are both matrices with integral entries. Therefore $\varpi^{-k}, \varpi^{-k}\norm_L^E(x), \varpi^{-k}\norm_L^E(y), \varpi^{-k}y \in {\mathfrak o}_E$. Clearly $k\le 0$. Furthermore, since $|\varpi^{-k}\norm_L^E(x)| = |\varpi^{k}|$ it follows that $k \ge 0$, so in fact $k = 0$. It follows that $x \in {\mathfrak o}_E^\times$ and $y \in {\mathfrak o}_E$.

	Assume that $X = X_M$ is split and set $h = \rho([\begin{smallmatrix}
		a_1 & b_1\\ c_1& d_1
	\end{smallmatrix}], [\begin{smallmatrix}
		a_2 & b_2\\ c_2& d_2
	\end{smallmatrix}])$ for some $[\begin{smallmatrix}
		a_1 & b_1\\ c_1& d_1
	\end{smallmatrix}]$ and $ [\begin{smallmatrix}
		a_2 & b_2\\ c_2& d_2
	\end{smallmatrix}]$ for which  $h(x_i) \in \mat(2,{\mathfrak o}_L)$ for $i \in \{1 , 2\}$. For $i \in \{1, 2\}$ we can write $$[\begin{smallmatrix}
			a_i & b_i \\ c_i & d_i
		\end{smallmatrix}] = [\begin{smallmatrix}
			a_i' & b_i' \\ c_i' & d_i'
		\end{smallmatrix}] [\begin{smallmatrix}
			w_i & y_i\\0 & z_i 
		\end{smallmatrix}]=[\begin{smallmatrix}
			a_i' & b_i' \\ c_i' & d_i'
		\end{smallmatrix}] [\begin{smallmatrix}
			w_i/z_i & y_i/z_i\\ & 1
		\end{smallmatrix}] [\begin{smallmatrix}
			z_i & \\ & z_i
		\end{smallmatrix}],$$
  for some $[\begin{smallmatrix}
      a_i' & b_i' \\ c_i' & d_i'
    \end{smallmatrix}] \in \gl(2,{\mathfrak o}_L)$ and $[\begin{smallmatrix}
      w_i & y_i \\ 0 & z_i        
      \end{smallmatrix}] \in \gl(2,L)$.
	Therefore for some, possibly different, $w_i,y_i,z_i \in L$ we have that
	\begin{align*}
		h &= \rho([\begin{smallmatrix}
			a_1' & b_1' \\ c_1' & d_1'
		\end{smallmatrix}], [\begin{smallmatrix}
			a_2' & b_2' \\ c_2' & d_2'
		\end{smallmatrix}]) \cdot \rho([\begin{smallmatrix}
			w_1 & y_1\\ & 1
		\end{smallmatrix}], [\begin{smallmatrix}
			w_2 & y_2\\ & 1
		\end{smallmatrix}]) \cdot \rho([\begin{smallmatrix}
			z_1 & \\ & z_1
		\end{smallmatrix}], [\begin{smallmatrix}
			z_2 & \\ & z_2
		\end{smallmatrix}])\\
		& = z_1z_2 \cdot \rho([\begin{smallmatrix}
			a_1' & b_1' \\ c_1' & d_1'
		\end{smallmatrix}], [\begin{smallmatrix}
			a_2' & b_2' \\ c_2' & d_2'
		\end{smallmatrix}]) \cdot \rho([\begin{smallmatrix}
			w_1 & y_1\\ & 1
		\end{smallmatrix}], [\begin{smallmatrix}
			w_2 & y_2\\ & 1
		\end{smallmatrix}])\\
		& = \rho([\begin{smallmatrix}
			a_1' & b_1' \\ c_1' & d_1'
		\end{smallmatrix}], [\begin{smallmatrix}
			a_2' & b_2' \\ c_2' & d_2'
		\end{smallmatrix}]) \, h'
	\end{align*}
	for an appropriate choice of $h'\in \gso(X)$. Set $z_1z_2 = u \varpi_L^k$ for some $u \in {\mathfrak o}_L^\times$ and $k \in \mathbb{Z}$. 
 
  Since $h \in \so(X)$ we calculate that that $1 = |\lambda(h)| = |w_1w_2 \varpi^{2k}|$. Furthermore we have that $h'(x_1) = u\varpi^k [\begin{smallmatrix}
			0 & w_1w_2\\
			0 & 0
		\end{smallmatrix}]$ and $h'(x_2)  = u \varpi^k [\begin{smallmatrix}
			-2y_1 & 2y_1y_2\\ -2 & 2y_2
		\end{smallmatrix}]$
	are both matrices with integral entries. Since $-2u \varpi^k \in {\mathfrak o}_E$, we know $k \ge 0$. We also have $|w_1w_2 \varpi^{k}| = |\varpi^{-k}|$, and since $w_1w_2\varpi^{k} \in {\mathfrak o}_E$, we conclude that $k = 0$. It follows that $w_1w_2 \in {\mathfrak o}_L$. By Lemma \ref{lemma:stab}, the observation that 
  \begin{align*}
    h = \rho([\begin{smallmatrix}
      a_1' & b_1' \\ c_1' & d_1'
    \end{smallmatrix}], [\begin{smallmatrix}
      a_2' & b_2' \\ c_2' & d_2'
    \end{smallmatrix}]) \rho(z_1z_1 [\begin{smallmatrix}
      1 & u_1 \\ & 1
    \end{smallmatrix}], [\begin{smallmatrix}
      w_2w_1 & u_2 \\ & 1
    \end{smallmatrix}]) 
      \rho([\begin{smallmatrix}
        w_1 & \\ & 1
      \end{smallmatrix}], [\begin{smallmatrix}
        w_1^{-1} & \\ & 1 
      \end{smallmatrix}])
  \end{align*}
  completes the proof.  \end{proof}

 The following theorem specifies a choice for $\varphi \in \mathcal S(X^2)$ so that when $W \in V$ is unramified, so is the local theta lift $B(\cdot, \varphi, W, s)$.
\begin{theorem} \label{thm:B1}
	If $X = X_{ns}$ is non-split then assume that $E/L, \tau_0$, and $\psi_E$ are all unramified, and set $Z = \Mat(2,{\mathfrak o}_{E}) \cap X(E)$. If  $X = X_M$ is split then assume that $\tau_1, \tau_2$, and $\psi$ are all unramified and set $Z = \mat(2, {\mathfrak o}_{L})$.

	Let $\varphi = f_{Z^2}$ be the characteristic function of $Z^2 \subset X^2$. Let $W_{un} \in  V$ be the standard unramified vector as in Theorem 11 of \cite{godement}. Then $B(\cdot, f_{Z^2}, W_{un},s)$ satisfies $B(gk, f_{Z^2},W_{un},s) = B(g, f_{Z^2}, W_{un},s)$ for $g \in \GSp(4,L_v)$ and $k \in \GSp(4,{\mathfrak o}_{L_v})$. Also
	\begin{equation}
		B(1,f_{Z^2}, W_{un},s) = \begin{cases}
			L(s, \tau_0) & $X$ \ \text{is non-split,}\\
			L(s, \tau_1) \cdot L(s, \tau_2) & $X$ \ \text{is split.}
		\end{cases}
	\end{equation}
\end{theorem}
\begin{proof}
	Let $g \in \GSp(4,L_v)$, and $k \in \GSp(4, {\mathfrak o}_{L_v})$. Suppose that $g_1$ is as in \eqref{eq:g0}. Assume that $X = X_{ns}$ is non-split. Because $E_v/L_v$ is unramified, $\norm_{L_v}^{E_v}({\mathfrak o}_{E,v}^\times) = {\mathfrak o}_{L,v}^\times$. This allows one to choose $j \in \GSO(X,L_v)$ such that $\lambda(j) = \lambda(k)$ and $jM = M$, using Lemma \ref{lemma:rho_maximal_compact}. We calculate
	\begin{align*}
		B(gk,\varphi,W_{un},s) & = |\lambda(gk)|^{-(s-\frac12)}B(\left[\begin{smallmatrix}
		\lambda(gk)^{-1} &  &  & \\
		  & 1 &  & \\
		 &  & 1 & \\
		  &  &  & \lambda(gk)^{-1}\end{smallmatrix}\right]gk, \varphi,W_{un},s)\\
		 & = |\lambda(g)|^{-(s-\frac12)}B(g_1k, \varphi, W_{un}, s)\\
		 & =  |\lambda(g)|^{-(s-\frac12)} \int\limits_{H\backslash \SO(X)} \big(\omega(g_1k, hj) \varphi \big)(x_1,x_2) Z(s, \pi(hj) W_{un}) \, dh\\
		 & =  |\lambda(g)|^{-(s-\frac12)} \int\limits_{H\backslash \SO(X)} \big(\omega(g_1, h)\omega(k,j) \varphi \big)(x_1,x_2) Z(s, \pi(h)\pi(j) W_{un}) \, dh\\
		 & =  |\lambda(g)|^{-(s-\frac12)} \int\limits_{H\backslash \SO(X)} \big(\omega(g_1, h) \varphi \big)(x_1,x_2) Z(s, \pi(h) W_{un}) \, dh\\
		 & = B(g,\varphi,W_{un},s).
	\end{align*}

	Let $h \in \SO(X)$. By Lemma \ref{lemma:rho_maximal_compact}, $h(x_1,x_2) \in Z^2$ if and only if $h \in \rho({\mathfrak o}_{L,v}^\times \times \GL(2, {\mathfrak o}_{E,v}))H$. By Lemma \ref{zetaintegrals} we have that the map $H \to \mathbb{C}$ given by $h \mapsto Z(s, \pi(h) W_{un})$ is constant. Therefore,
	\begin{align*}
		B(1,\varphi,W,s) & = \int\limits_{H\backslash \SO(X)}  f_{Z^2}\big( h^{-1} (x_1,x_2) \big) Z(s, \pi(h) W_{un}) \, dh\\
		& = \int\limits_{H\backslash [ \SO(X) \cap H\rho({\mathfrak o}_{L,v}^\times \times \GL(2, {\mathfrak o}_{E,v}))]} Z(s, \pi(h) W_{un}) \, dh\\
		& = Z(s,W_{un}),
	\end{align*}
	since the Haar measure was chosen so that $\vol\big({H\backslash (\SO(X) \cap H\rho({\mathfrak o}_{L,v} \times \gl(2,{\mathfrak o}_E)}))\big) = 1$. It is observed in Theorem 11 of \cite{godement} that $Z(s, W_{un}) = L(s, \tau_0)$, which completes the proof in the non-split case. 

	This proof of the split case is similar.
  \end{proof}

\begin{lemma} \label{lemma:homspaces}
	Let $(\pi , V)$ be a representation of $\GSO(X)$, let $(\sigma , V \times V)$ be the induced representation of $\GO(X)$ obtained from $\pi$ as in Section \ref{sec:notation}. Let $(\Pi , W)$ be a representation of $\GSp(4 , L)$. Then we have the following $\C$-linear isomorphism
	\begin{align*}
		M: \Hom_{R'}( \mathcal S(X^2) \otimes V , W)   &\xrightarrow{ \ \sim \ } \Hom_{R} ( \mathcal S(X^2) \otimes (V \times V) , W)
	\end{align*}
	determined by $M(f)(\varphi \otimes (v_1 , v_2)) = f(\varphi \otimes v_1) + f(\omega(1 , s)\cdot \varphi \otimes v_2)$ for $\varphi \in \mathcal S(X^2), f \in \Hom_{R'}( \mathcal S(X^2) \otimes V , W),$ and $v_1,v_2 \in V$, and extended linearly. Here $(1 , s)$ is a non-trivial coset representative of $R/R'$. For example we could take $s$ the map that takes $x$ to $x^\ast$. Additionally the inverse map
	\begin{align*}
		N: \Hom_{R}( \mathcal S(X^2) \otimes (V \times V) , W)   &\xrightarrow{ \ \sim \ } \Hom_{R'} ( \mathcal S(X^2) \otimes V  , W)
	\end{align*}
	is given by $N(f)(\varphi \otimes v) = f(\varphi \otimes (v, 0))$ for $f \in \Hom_{R} ( \mathcal S(X^2) \otimes (V \times V) , W), \varphi \in \mathcal S(X^2)$, and $v \in V$.
\end{lemma}
\begin{proof}
	Let us start with proving that $M$ is well defined. Let $f \in \Hom_{R'} ( \mathcal S(X^2) \otimes V  , W), (g , h) \in R', \varphi \in \mathcal S(X^2),$ and $v_1,v_2 \in V$. We see that
	\begin{align*}
		M(f)\big( (g , h) \cdot (\varphi \otimes (v_1 , v_2)) \big) & = M(f)\big(  \omega(g , h) \cdot \varphi \otimes (\pi(h)\cdot v_1 , \pi(s h s^{-1}) \cdot v_2) \big)\\
		& = f(\omega (g  ,h) \cdot \varphi \otimes \pi(h)\cdot v_1) + f(\omega(1,s)\omega(g , h)\varphi \otimes \pi(s h s^{-1}) \cdot v_2)\\
		& =  f(\omega (g  ,h) \cdot \varphi \otimes \pi(h) \cdot v_1) + f(\omega(g,s h s^{-1})\omega(1 , s)\varphi \otimes \pi(s h s^{-1}) \cdot v_2)\\
		& = g \cdot f(\varphi \otimes v_1) + g \cdot f(\omega(1 , s) \cdot \varphi \otimes v_2)\\
		& = g\cdot M(f)(\varphi \otimes (v_1 , v_2)).
	\end{align*}
	Hence, we see that $M(f) \in \Hom_{R'} ( \mathcal S(X^2) \otimes (V \times V) , W)$ so it only remains to show that
	\begin{equation*}
		(1 , s) \cdot M(f)(\varphi \otimes (v_1 , v_2)) = M(f)(\omega(1 , s) \cdot \varphi \otimes \sigma(s)\cdot (v_1 , v_2)).
	\end{equation*}
	We calculate
	\begin{align*}
		M(f)\big( (1 , s) \cdot (\varphi \otimes (v_1 , v_2)) \big) & = M(f)\big( \omega(1 , s)\cdot \varphi \otimes(v_2 , v_1) \big)\\
		& = f(\omega(1 , s) \cdot \varphi  \otimes v_1) + f(\varphi \otimes v_2)
		\\
		& = M(f)(\varphi \otimes (v_2 , v_1))\\
		& = (1 , s)\cdot M(f)(\varphi \otimes (v_1 , v_2)).
	\end{align*}

	Now we have to prove that $M$ is one-to-one and onto. Let $f,f' \in \Hom_{R'} ( \mathcal S(X^2) \otimes V  , W)$ and assume that $M(f) = M(f')$. For all $\varphi \in \mathcal S(X^2)$ and for all $v_1 \in V$ we have that 
	\begin{align*}
		M(f)(\varphi \otimes (v_1 , 0)) &= M(f')(\varphi \otimes (v_1 , 0)), \quad \text{so that}\\
		 f( \varphi \otimes v_1) &= f'( \varphi \otimes v_1), \quad \text{and}\\
		  f & = f'.
	\end{align*} 
	
  Now we show that $N$ is well defined and one-to-one. Let $f \in \Hom_{R}( \mathcal S(X^2) \otimes (V \times V)  , W), (g,h) \in R', \varphi \in \mathcal S(X^2)$ and $v \in V$. Then $N(f) \in \Hom_{R'}( \mathcal S(X^2) \otimes V  , W)$ since
  \begin{align*}
    N(f) ((g,h) \cdot ( \varphi \otimes v ))  & = N(f) ((g , h) \cdot \varphi \otimes \pi(h)v)\\
    & = f((g,h) \varphi \otimes (\pi(h)v,0))\\
    & = f((g,h) \varphi \otimes \sigma(h)(v,0))\\
    & = g\cdot f(\varphi \otimes (v,0))\\
    & = g \cdot N(f)(\varphi \otimes v). 
  \end{align*} 
  Let $f,f' \in \Hom_{R} ( \mathcal S(X^2) \otimes (V \times V)  , W)$ and assume that $N(f) = N(f')$. Then for all $\varphi \in \mathcal S(X^2)$ and $v\in V$
	\begin{align*}
		N(f) & = N(f'), \quad \text{so that}\\
		f(\varphi \otimes (v,0)) & = f'(\varphi \otimes (v,0)), \\
		(1,s) \cdot f(\varphi \otimes (v,0)) & = (1,s)\cdot f'(\varphi \otimes (v,0)), \quad \text{and}\\
		f(\varphi \otimes (0,v)) & =  f'(\varphi \otimes (0,v)).
	\end{align*}
  We conclude that $f = f'$, so that $N$ is one-to-one. 

	Next we verify that $N = M^{-1}$.
	\begin{align*}
		\big[(M \circ N)(f)\big] (\varphi \otimes (v_1 , v_2)) & = N(f)(\varphi \otimes v_1) + N(f)(\omega(1 , s)\cdot\varphi \otimes v_2)\\
		& = f(\varphi \otimes v_1 , 0) + f(\omega(1 , s)\cdot \varphi \otimes (v_2 , 0))\\
		& = f(\varphi \otimes v_1 , 0) + f(\varphi \otimes (0 , v_2))\\
		& = f(\varphi \otimes (v_1 , v_2)).
		\intertext{So we have that $M\circ N = {\rm Id}$. On the other hand,}
		\big[ (N \circ M)(f) \big](\varphi \otimes v) &= M(f)(\varphi \otimes (v,0))\\
		& = f(\varphi \otimes v) + f(\omega(1 , s) \cdot \varphi \otimes 0)\\
		& = f(\varphi \otimes v).
	\end{align*}
	Since $M$ is an isomorphism, so is $N$. 

  \end{proof}

From Theorem \ref{Besselrepresentation} and Lemma \ref{lemma:homspaces} we can easily establish the following result.
\begin{corollary} \label{cor:piplusmap}
  Let $(\pi , V)$ be the representation of $\GSO(X)$ as in Section \ref{sec:notation} so that $V$ is either equal to $\mathcal W(\tau_1, \psi) \otimes \mathcal W(\tau_2,\psi)$ or equal to $\mathcal W(\tau_0, \psi_E)$. Let $(\sigma , V \times V) = \Ind_{\gso(X)}^{\go(X)} \pi$  and let $\pi^+$ be the canonical irreducible subrepresentation of $\sigma$. Let $\vartheta$ be as in \eqref{eq:localthetalift} and let $M$ be the map in lemma \ref{lemma:homspaces}. The composition 
  \begin{align*}
     M(\vartheta): \mathcal S(X^2) \otimes (V \times V) \to \Theta(V)
  \end{align*} 
  is a non-zero $R$-map. Furthermore, the restriction of $M(\vartheta)$ to $\pi^+$ is a non-zero $R$-map. 
\end{corollary}
\begin{proof}
  By part (a) of Theorem \ref{Besselrepresentation} we have that $\vartheta \in \Hom_{R'}( \mathcal S(X^2) \otimes V , \Theta(V))$ so that $M(\vartheta) \in \Hom_{R}( \mathcal S(X^2) \otimes (V \times V) , \Theta(V))$. Furthermore $M(\vartheta)$ is non-zero by Lemma \ref{bessle_trasformation_lemma}. Indeed, for $\varphi \in \mathcal S(X^2)$ and $W , W' \in V$ we have that
  \begin{align*}
    M(\vartheta)(\varphi \otimes (W, W')) & = \vartheta  (\varphi \otimes W) + \vartheta (\omega(1 , s) \cdot \varphi \otimes  W')
  \\
    & = B(\cdot , \varphi , W , s) + B(\cdot , \omega(1 , s) \varphi , T(W') , s)
  \\
    & =  B(\cdot , \varphi , W , s) + B(\cdot , \varphi , T(W') , s).
  \end{align*}
  Since $\tau_i$ has trivial central character we have that $\pi^{+}$ has space $V_{\pi^{+}} = \{(v , T(v)) \mid v \in V\}$, where $T$ is as in Lemma \ref{bessle_trasformation_lemma}. Using part (a) of Lemma \ref{bessle_trasformation_lemma} we calculate that $M(\vartheta)(\mathcal S(X^2) \otimes V_{\pi^+}) \ne 0$ since for $\varphi \in \mathcal S(X^2)$ and $W \in V$
  \begin{align*}
    M(\vartheta)(\varphi \otimes (W, \pi(s) W)) &= 2 B(\cdot , \varphi , W , s),
  \end{align*} 
  which is not zero by Theorem \ref{Besselrepresentation}.
  \end{proof}

\subsection{The Archimedean Case} % (fold)
\label{sec:intertwining_maps_archimedean}

In this section we make the  assumption that $L = \mathbb{R}$ and that $X = X_M$ is split, so that $E = \mathbb{R} \times \mathbb{R}$. Fix the Lie algebras ${\mathfrak h}$, ${\mathfrak g}$, ${\mathfrak r}$ of $\go(X(L)), \gsp(4,L),$ and $R(L)$, respectively, and maximal compact subgroups $ J \subset \go(X(L)), K \subset \gsp(4, L),$ and $F \subset R$. Let $(\omega , \mathcal S(X^2))$ be the $({\mathfrak r} , F)$-module given by the archimedean Weil representation. Let $\tau_1, \tau_2$ be infinite-dimensional, irreducible, admissible $({\mathfrak g} , K)$-modules for $\gl(2,\mathbb{R})$ that have the same central character, but are not isomorphic. Let $\mathcal W(\tau_i,\psi)$ be the Whittaker model for $\tau_i$. Set $V = \mathcal W(\tau_1, \psi) \otimes \mathcal W(\tau_2, \psi)$. Let $W = W_1 \otimes W_2 \in V$ and let $Z(s,W)$ be defined as in \eqref{eq:split_zeta_def}.

The primary purpose of this section is to define the intertwining map analogous to \eqref{eq:splitbesselintegral}, in the archimedean case, and to verify some necessary results. In particular we want to establish that the integral that defines the intertwining map is absolutely convergent, non-zero, and corresponds to an intertwining map in the theta correspondence. Let $g \in \gsp(4,\mathbb{R})^+ = \gsp(4,\mathbb{R})$ and $\varphi \in \mathcal S(X^2)$. Let $W \in V$. We define
\begin{equation}\label{eq:archbesselintegral}
  B(g,\varphi,W,s) = \int\limits_{H \backslash \so(X)} (\omega(g,hh')\varphi)(x_1,x_2) Z(s, \pi(hh')W) \,dh
\end{equation}
where $h'\in \gso(X)$ is chosen so that $\lambda(h') = \lambda(g).$

\begin{lemma} \label{lemma:bessel_abs_convergence}
  For $g \in \gsp(4,\mathbb{R}), \varphi \in \mathcal S(X^2), W = W_1\otimes W_2 \in V$ such that $Z(s,W)\ne 0$, and $s > \frac{3}{2}$ the integral in \eqref{eq:archbesselintegral}, defining $B(g,\varphi,W,s)$, converges absolutely. In fact
  \begin{equation}\label{eq:arch_abs_intertwing}
    \int\limits_{H(\mathbb{R}) \backslash \so(X, \mathbb{R})} |\big(\omega(g, hh') \varphi \big)(x_1,x_2)| \cdot  Z(\Re(s), |\pi(hh') W|) \, dh
  \end{equation}
  is finite. 
\end{lemma}
\begin{proof} 
  This proof appears in the unpublished manuscript \cite{Roberts_Bessel_manuscript}. Since we are in the split case we have that the integral \eqref{eq:arch_abs_intertwing} is
  \begin{align*}
    I&= \int_{H(\mathbb{R}) \backslash \so(X,\mathbb{R})} |\big(\omega(1, h) \varphi' \big)(x_1,x_2)|\cdot 
    \\
    & \hspace{2cm}\int\limits_{\mathbb{R}^\times} |x|^{\Re(s) - \frac{1}{2}}|W_1([\begin{smallmatrix}
      x & \\ & 1
    \end{smallmatrix}h_1)]| \,d^{\times}x \int\limits_{\mathbb{R}^\times} |x|^{\Re(s) - \frac{1}{2}}|W_2([\begin{smallmatrix}
      x & \\ & 1
    \end{smallmatrix}]h_2)| \,d^{\times}x \,dh
  \end{align*}
  where $h = \rho(h_1,h_2)$ and $\varphi' = \omega(g,h')\varphi$. Let $T$ be the group $\{[\begin{smallmatrix}
      y & \\ &y^{-1}
    \end{smallmatrix}] ,\,y \in \mathbb{R}^\times\}$
  and let $\Delta T = \{ (t , t^{-1})  \mid t \in T\}$. Then $\rho$ produces the following homeomorphism
  \begin{equation*}
    \rho: \Delta T\backslash (\SL(2,\mathbb{R}) \times \SL(2,\mathbb{R})) \xrightarrow{ \ \sim \ } H(\mathbb{R}) \backslash \so(X,\mathbb{R}).
  \end{equation*} 
  We have $\Delta \subset T\times T \subset \SL(2,\mathbb{R}) \times \SL(2,\mathbb{R})$. Moreover, $\Delta T$ and $T\times T$ are closed unimodular subgroups of $\SL(2,\mathbb{R})$ and $\SL(2,\mathbb{R})$ is unimodular. Therefore, we can write \eqref{eq:arch_abs_intertwing} as
  \begin{multline*}
I=\int\limits_{T\backslash \SL(2, \mathbb{R}) \times T\backslash \SL(2, \mathbb{R})} \int\limits_{T}  \bigg(|\big(\omega(1, \rho(th_1,h_2)) \varphi' \big)(x_1,x_2)| 
        \int\limits_{\mathbb{R}^\times} |x|^{\Re(s) - \frac{1}{2}}|W_1([\begin{smallmatrix}
          x & \\ & 1
        \end{smallmatrix}]th_1)| \,d^{\times}x \\\int\limits_{\mathbb{R}^\times} |x|^{\Re(s) - \frac{1}{2}}|W_2([\begin{smallmatrix}
          x & \\ & 1
        \end{smallmatrix}]h_2)| \,d^{\times}x \bigg) \,dt \,dh_1\,dh_2.
  \end{multline*}
  Let $N=\{[\begin{smallmatrix}
      1&z \\ &1
    \end{smallmatrix}], \, z \in \mathbb{R}\}$.
 Then, a standard integration formula gives
  \begin{align} \label{eq:Big_itegrand}
   I = & \int\limits_{\so(2,\mathbb{R}) \times \so(2,\mathbb{R})} \int\limits_{N\times N} \int\limits_{T}  \bigg(|\big(\omega(1, \rho(tn_1k_1,n_2k_2)) \varphi' \big)(x_1,x_2)|\cdot 
        \\
        & \hspace{.5cm}\int\limits_{\mathbb{R}^\times} |x|^{\Re(s) - \frac{1}{2}}|W_1([\begin{smallmatrix}
          x & \\ & 1
        \end{smallmatrix}]t_1n_1k_1)| \,d^{\times}x \int\limits_{\mathbb{R}^\times} |x|^{\Re(s) - \frac{1}{2}}|W_2([\begin{smallmatrix}
          x & \\ & 1
        \end{smallmatrix}]n_2k_2)| \,d^{\times}x \bigg) \,dt \,dn_1\,dn_2\,dk_1\,dk_2. \nonumber
  \end{align}
  This is an improvement since $N$ and $T$ are explicitly defined. Since $\mathcal W(\tau_i,\psi)$ is $\so(2,\mathbb{R})$-finite, the subspace spanned by $\so(2,\mathbb{R})$-translates of $W_i$, for $i \in \{1,2\}$, is finite dimensional vector space. Therefore we can choose a positive integer $n$ and functions $W_{1,i}, W_{2,j} \in \mathcal W(\pi,\psi)$, for $1 \le i,j \le n$, so thatc$W_1 = \sum_{i = 1}^n W_{1,i}$ and $W_2 = \sum_{j = 1}^n W_{2,j}$,
  where each $W_{1,i}$ and $W_{2,j}$ transforms according to a character of $\so(2, \mathbb{R})$. Let the functions $F_{i,j}: \so(X,\mathbb{R}) \to \mathbb{C}$ be the inner integrals of \eqref{eq:Big_itegrand} but with $W_{1,i}$ and $W_{2,j}$ replacing $W_{1}$ and $W_{2}$, respectively. With this \eqref{eq:arch_abs_intertwing} is 
  \begin{align*}
    \le & \sum_{\substack{1 \le i \le n \\ 1 \le j \le n}} \int\limits_{\so(2,\mathbb{R}) \times \so(2,\mathbb{R})} \int\limits_{N\times N} \int\limits_{T} F_{i,j}(\rho(tn_1k_1,n_2k_2)) \,dt \,dn_1\, dn_2 \, dk_1 \, dk_2.
  \end{align*}
  For any pair $(i,j)$, with $1\le i \le n$ and $1 \le j \le n$, we have
  \begin{align*}
    & \int\limits_{\so(2,\mathbb{R}) \times \so(2,\mathbb{R})} \int\limits_{N\times N} \int\limits_{T} F_{i,j}(\rho(tn_1k_1,n_2k_2)) \,dt \,dn_1\, dn_2 \, dk_1 \, dk_2
  \\
    = & \int\limits_{\so(2,\mathbb{R})\times \so(2,\mathbb{R})} \int\limits_{\mathbb{R}\times \mathbb{R}} \int\limits_{\mathbb{R^\times}} F_{i,j}(\rho([\begin{smallmatrix}
      t&\\&t^{-1}
    \end{smallmatrix}][\begin{smallmatrix}
      1 & z_1 \\ &1
    \end{smallmatrix}]k_1, [\begin{smallmatrix}
      1 & z_2 \\ &1
    \end{smallmatrix}]k_2 ) \, \frac{dt}{|t|} \, dz_1 \, dz_2 \, dk_1 \,dk_2
  \\
    =& \int\limits_{\mathbb{R}\times \mathbb{R}} \int\limits_{\mathbb{R}^\times} \int\limits_{\so(2,\mathbb{R})\times \so(2,\mathbb{R})} |(\omega(1, \rho(k_1,k_2)) \varphi')([\begin{smallmatrix}
      &t^{-1}\\&
    \end{smallmatrix}], -2 [\begin{smallmatrix}
      z_1t & z_1z_2 \\ t & tz_2
    \end{smallmatrix}])| \,dk_1\,dk_2 \cdot
    \\
    &  (\int\limits_{R^\times} |x|^{\Re(s) - \frac{1}{2}}|W_{1,i}([\begin{smallmatrix}
      x&\\&1
    \end{smallmatrix}][\begin{smallmatrix}
      t&\\&t^{-1}
    \end{smallmatrix}][\begin{smallmatrix}
      1&z_1\\&1
    \end{smallmatrix})]|\,d^\times x)
    (\int\limits_{R^\times} |x|^{\Re(s) - \frac{1}{2}}|W_{2,j}([\begin{smallmatrix}
      x&\\&1
    \end{smallmatrix}][\begin{smallmatrix}
      1&z_2\\&1
    \end{smallmatrix}])|\,d^\times x) \, \frac{dt}{|t|} \,dz_1 \,dz_2.
    \end{align*}
  Since each $W_{u,i}$ is in $\mathcal W(\pi, \psi)$ and transforms according to a character of $\so(2, \mathbb{R})$, it can be show that there are positive constants $C_{1,i}, C_{2,j}$ so that the above is
  \begin{align*}
    I\le & C_{1,i} C_{2,j} \int\limits_{\so(2,\mathbb{R})\times \so(2,\mathbb{R})} \int\limits_{\mathbb{R}\times \mathbb{R}} \int\limits_{\mathbb{R}^\times}  |(\omega(1, \rho(k_1,k_2)) \varphi')([\begin{smallmatrix}
      &t^{-1}\\&
    \end{smallmatrix}], -2 [\begin{smallmatrix}
      z_1t & z_1z_2 \\ t & tz_2
    \end{smallmatrix}])| \,\frac{dt}{|t|} \,dz_1 \,dz_2 \,dk_1\,dk_2.
  \end{align*} 
  Define $C = \sum_{1 \le i,j \le n} C_{1,i} C_{2,j}$ so that \eqref{eq:arch_abs_intertwing} is 
  \begin{align*}
  I  \le C \int\limits_{\so(2,\mathbb{R})\times \so(2,\mathbb{R})} \int\limits_{\mathbb{R}\times \mathbb{R}} \int\limits_{\mathbb{R}^\times}  |(\omega(1, \rho(k_1,k_2)) \varphi')([\begin{smallmatrix}
      &t^{-1}\\&
    \end{smallmatrix}], -2 [\begin{smallmatrix}
      z_1t & z_1z_2 \\ t & tz_2
    \end{smallmatrix}])| \,\frac{dt}{|t|} \,dz_1 \,dz_2 \,dk_1\,dk_2.
  \end{align*}

  Let $L,M,$ and $N$ be positive integers which we will determine momentarily. Define a polynomial on $X$ by
  \begin{equation*}
    Q([\begin{smallmatrix}
      x_1 & x_2 \\ x_3 & x_4
    \end{smallmatrix}] , [\begin{smallmatrix}
      x_1' & x_2' \\ x_3' & x_4'
    \end{smallmatrix}]) = (1 + (\frac{1}{2} x'_3)^{2N} + x_2^{2M})(1 + (\frac{1}{2}x_1'x_2)^{2L} + (\frac{1}{2}x_2x_4')^{2L})
  \end{equation*}
  where $x_i$ and $x'_i$ in $L$ for $i \in \{1,2\}$. Let $||\cdot||_{Q}$ be the semi-norm corresponding to $Q$, on the space of $\mathcal S(X^2)$, given explicitly by
  \begin{equation*}
    ||\varphi||_Q = \sup_{x \in X^2}|Q(x)\varphi (x)|.
  \end{equation*}
  Then we have that the integral from  \eqref{eq:arch_abs_intertwing} is bounded as
  \begin{align*}
   I \le & C\int\limits_{\so(2,\mathbb{R})^2} ||  \omega(1, \rho(k_1,k_2)) \varphi'  ||_{Q} \int\limits_{\mathbb{R} \times \mathbb{R}} \int\limits_{\mathbb{R}^\times} \frac{|t|^{-2\Re(s - \frac{1}{2}) - 1}}{Q([\begin{smallmatrix}
      &t^{-1}\\&
    \end{smallmatrix}], -2 [\begin{smallmatrix}
      -z_1t & -z_1z_2 \\ t & tz_2
    \end{smallmatrix}])} \,dt\,dz_1\,dz_2 \,dk_1\,dk_2
  \\
    = & C\int\limits_{\so(2,\mathbb{R})^2} ||  \omega(1, \rho(k_1,k_2)) \varphi'  ||_{Q} \int\limits_{\mathbb{R} \times \mathbb{R}} \frac{1}{1 + |z_1|^{2L} + |z_2|^{2L}} \,dz_1 \,dz_2 \int\limits_{\mathbb{R}^\times} \frac{|t|^{2M - 2\Re(s - \frac{1}{2}) - 1}}{1 + |t|^{2(N + M)} + |t|^{2M}} \,dt   \,dk_1\,dk_2. \end{align*}
We choose $L$ to be large enough so that the above integral over $\mathbb{R} \times \mathbb{R}$ converges. Also, we choose $M$ to be large enough so that $2M - 2(\Re(s - \frac{1}{2})) - 1$ is positive and then choose $N$ large enough so that the above integral over $\mathbb{R}^\times$ converges. Therefore, there is a constant $C' \in \mathbb{R}$ so that $I\le  C' \int_{\so(2,\mathbb{R})^2} ||  \omega(1, \rho(k_1,k_2)) \varphi'  ||_{Q} \,dk_1\,dk_2$.
 
  Define an action of $\so(2,\mathbb{R})^2$ on polynomials on $X_{M}$ by $(k_1,k_2)P(y) = R(\rho(k_1,k_2)^{-1}y)$. There exist an integer $m$ and polynomials $\{Q_i\}_{i = 1}^m$ which all transform according to a character of $\so(2,\mathbb{R})^2$ such that $Q = \sum_{i = 1}^m Q_i$.
  Furthermore, for all $(k_1,k_2) \in \so(2,\mathbb{R})^2$ , we have $$|| \omega(1, \rho(k_1,k_2))\varphi' ||_{Q} \le \sum_{i = 1}^m || \varphi' ||_{Q_i}.$$
  Therefore we conclude that integral \eqref{eq:arch_abs_intertwing} is bounded above by
  $C' \vol(\so(X,\mathbb{R})) \sum_{i = 1}^m ||  \omega(g ,h') \varphi  ||_{Q_i}$ which is finite.
    \end{proof}

\begin{lemma}\label{lemma:archimedean_nonvanishing}
  Suppose that $W \in \mathcal V$ such that $Z(s, W) \ne 0$ and for $s \gg M$. There exists $\varphi \in \mathcal S(X^2)$ such that $B(1,\varphi,W,s) \ne 0$. 
\end{lemma}
\begin{proof}
  There certainly exists a smooth rapidly decreasing function $\bar\varphi:X^2 \to \mathbb{C}$ so that the Bessel integral $B(1,\bar\varphi,W,s) \ne 0$. Since $\mathcal S(X^2)$ is dense in the Schwartz space on $X^2$ there exists a sequence $\{\varphi_i\}_{i \ge 1} \subset \mathcal S(X^2)$ so that $\varphi_i \to \bar\varphi$. As in the proof of Lemma \ref{lemma:bessel_abs_convergence} we have that
  \begin{equation*}
    |B(1,\bar\varphi,W,s) - B(1,\varphi_n,W,s)| \le C C' \vol(2,\mathbb{R})^2 \sum_{i = 1}^m || \varphi - \varphi_n ||_{Q_i}.
  \end{equation*}
  Since $B(1,\bar\varphi,W,s) \ne 0$ there is some integer $n$ so that $B(1,\varphi_n,W,s) \ne 0$.
  \end{proof}

Now we are sure that $B(\cdot, \varphi,W,s)$ is well defined and non-zero, for some choice of $\varphi \in \mathcal S(X^2)$, $W \in V$, and $s \in \mathbb{C}$. We define $\Theta(V)$ to be the $({\mathfrak g}, K)$-module, for $\gsp(4,L)$, of smooth functions generated by $B(\cdot,\varphi,W,s)$ for all choices of $\varphi \in \mathcal S(X^2) , W \in V$, and $\Re(s) > \frac{3}{2}$. We define the map
\begin{align}\label{eq:localthetalift_arch}
  \vartheta: \mathcal S(X^2) \otimes V &\to \Theta(V)\\
  \varphi \otimes W & \mapsto B(\cdot, \varphi, W, s). \nonumber
\end{align}
If $B \in \Theta(V)$ then by Remark \ref{remark:arch} we are justified in using  Lemma \ref{bessle_trasformation_lemma} to verify that 
\begin{align} \label{eq:extended_bessel_arch}
B(tg)  &= |t_1/t_2|^{\frac{1}{2} - s}B(g) \quad \text{for $t = \left[\begin{smallmatrix}
  t_1&&&\\&t_2&&\\&&t_2&\\&&&t_1
\end{smallmatrix}\right]$ for $t_1,t_2 \in L^\times$, and}\\
\label{eq:extended_bessel2_arch}
B(bg) & = \psi(b_2)B(g)\quad \text{for $b = [\begin{smallmatrix}
  1 & B \\ &1
\end{smallmatrix}]$ where $B = [\begin{smallmatrix}
  b_1 & b_2 \\ b_2 & b_3
\end{smallmatrix}] \in \mat(2,L)$.}  
\end{align}
With this in place we can extend the results of Theorem \ref{Besselrepresentation} to the archimedean case.
\begin{theorem}
\label{Besselrepresentation_arch}
Let $s \in \mathbb{C}$ with $\Re(s) > \frac{3}{2}$.  Let $V = \mathcal W(\tau_1 , \psi) \otimes \mathcal W(\tau_2, \psi)$. Then
\begin{enumerate}
  \item[(a)] The map $\vartheta$ from \eqref{eq:localthetalift_arch} is a non-zero $({\mathfrak r} , F')$-map.
  \item[(b)] The image of $\varphi$ lies inside $\mathcal B(\gsp(4, \mathbb{R}) , \psi)$.
\end{enumerate}
It follows that $\Theta(V)$ is a $({\mathfrak r} , F')$-module sitting inside $\mathcal B(\gsp(4,\mathbb{R}), \psi)$.
\end{theorem}
\begin{proof}
  Since $\gsp(4,\mathbb{R})^+ = \gsp(4,\mathbb{R})$ it follows that $\vartheta$ is an $R$-map from the simple calculation: if $(g_0,h_0) \in R' , (g , h') \in R , \varphi \in \mathcal S(X^2), W \in V,$ and $\Re(S) > \frac{3}{2}$ then 
  \begin{align*}
    B(g, \omega(g_0,h_0) \varphi, \pi(h_0) W, s) & = \int\limits_{H \backslash \so(X)} \omega(gg_0, h(h'h_0)) \varphi(x_1,x_2) Z(s, \pi(h (h'h_0))W) \, dh\\
    & = B(gg_0, \varphi, W, s)
  \end{align*}
  since $(gg_0, h'h_0) \in R$. It follows that $\vartheta$ induces a $({\mathfrak r}\ , F')$-equivariant map.

  Part (b) follows from the definition of $\Theta(V)$ and from \eqref{eq:extended_bessel_arch} and \eqref{eq:extended_bessel2_arch}.
  \end{proof}

Lastly, we need to extend the results of Corollary \ref{cor:piplusmap} to the case when $L = \mathbb{R}$.

\begin{lemma}
  \label{lemma:homspaces_arch}
  Let $(\pi , V)$ be a $({\mathfrak h} ,  J_1)$-module, let $(\sigma , V \times V)$ be the induced $({\mathfrak h},J)$-module obtained from $\pi$. Let $(\Pi , W)$ be a $({\mathfrak g} , K)$-module. Then we have the following $\C$-linear isomorphism
  \begin{align*}
    M: \Hom_{({\mathfrak r} , F')}( \mathcal S(X^2) \otimes V , W)   &\xrightarrow{ \ \sim \ } \Hom_{({\mathfrak r} , F)} ( \mathcal S(X^2) \otimes (V \times V) , W)
  \end{align*}
  determined by $M(f)(\varphi \otimes (v_1 , v_2)) = f(\varphi \otimes v_1) + f(\omega(1 , s)\cdot \varphi \otimes v_2)$ for $\varphi \in \mathcal S(X^2), f \in \Hom_{({\mathfrak r} , K')}( \mathcal S(X^2) \otimes V , W),$ and $v_1,v_2 \in V$, and extended linearly. Here $(1 , s)$ is a non-trivial coset representative of $R/R'$. For example we could take $s$ the map that takes $x$ to $x^\ast$. Additionally the inverse map
  \begin{align*}
    N: \Hom_{({\mathfrak r} , F)}( \mathcal S(X^2) \otimes (V \times V) , W)   &\xrightarrow{ \ \sim \ } \Hom_{({\mathfrak r}  , F')} ( \mathcal S(X^2) \otimes V  , W)
  \end{align*}
  is given by $N(f)(\varphi \otimes v) = f(\varphi \otimes (v, 0))$ for $f \in \Hom_{({\mathfrak r} , K)} ( \mathcal S(X^2) \otimes (V \times V) , W), \varphi \in \mathcal S(X^2)$, and $v \in V$.
\end{lemma}
\begin{proof}
  This is similar to the proof of Lemma \ref{lemma:homspaces}.  
  \end{proof}
\begin{corollary}
  \label{cor:piplusmap_arch}
  Let $(\pi , V)$ be the $({\mathfrak h} , J')$-module associated to the representation of $\GSO(X)$ as in Section \ref{sec:notation} so that $V$ is equal to $\mathcal W(\tau_1, \psi) \otimes \mathcal W(\tau_2,\psi)$. Let $(\sigma , V \times V)$ be the $({\mathfrak h} , J)$-module, associated to the representation of $\GO(X)$ that is isomorphic to $\Ind_{\gso(X)}^{\go(X)} \pi$. We are justified in applying this result to the real case. Let $\pi^+$ be the canonical irreducible subrepresentation of $\sigma$. Let $\vartheta$ be as in \eqref{eq:localthetalift_arch} and let $M$ be the map in \ref{lemma:homspaces_arch}. The composition 
  \begin{align*}
     M(\vartheta): \mathcal S(X^2) \otimes (V \times V) \to \Theta(V)
  \end{align*} 
  is a non-zero $({\mathfrak r} , F)$-map. Furthermore, the restriction of $M(\vartheta)$ to $\pi^+$ is a non-zero $({\mathfrak r} , F)$-map. 
\end{corollary}
\begin{proof}
  It is clear, from what is presented in Corollary \ref{cor:piplusmap}, that $M(\vartheta)|_{\pi^+}$ is a non-zero $({\mathfrak r} , F)$-equivariant map. 
  \end{proof}

% section intertwining_maps_the_archimedean_case (end)

\section{The explicit Schwartz functions}\label{sec:testdata}
In this section, we construct explicit Schwartz functions, $\varphi \in \mathcal S(X^2)$, which produce a nonzero paramodular invariant vector. We will work separately in the case that $E/L$ is split, inert, and tamely ramified. Unfortunately, when $E/L$ is wildly ramified we do not find any such Schwartz functions. For our global application, we note that the only real quadratic number field for which we do not attain an explicit local lift at every place is $\mathbb{Q}(\sqrt{2})$. 

 In the split case, assume that $W \in  V$ is $\Gamma_0({\mathfrak p}^{n_1})\times \Gamma_0({\mathfrak p}^{n_2})$-invariant for some non-negative integers $n_1,n_2$. In the nonsplit case assume that $W \in V$ is $\Gamma_0({\mathfrak P}^n)$-invariant for some non-negative integers $n$. For any $\varphi \in \mathcal S(X^2)$, let $B(\cdot, \varphi, W, s)$ be as in Section \ref{sec:intertwiningmaps}. 

There are three important considerations. First, we want the support of $\varphi$ to be simple, with regard to $Z(s, W)$. To be specific, if $h \in H\backslash \so(X)$ so that $(h^{-1}x_1,h^{-1}x_2) \in \supp (\varphi)$ then we want $Z(s,\pi(h)W)$ to be simple to compute. When $L$ has odd residual characteristic we are able to make choices so that $Z(s, \varphi(h)W) = Z(s, W)$, for all such $h$. Our second consideration is invariance under the Weil representation for $(g,h) \in R \cap (K({\mathfrak p}^N) \times \gso(X))$. When $E/L$ is unramified these first two goals have almost perfect overlap, while if $E/L$ is tamely ramified there is still enough overlap to create good candidate Schwartz functions. 

If we make a natural choice of Schwartz function in the ramified case we do not get full paramodular invariance in the Weil representation, but we do get invariance under a rather large subgroup of $K({\mathfrak p}^N)$. We can sum over coset representatives of this subgroup to produce a fully paramodular invariant Schwartz function. Of course, we must verify that our choice of $\varphi$ does not make $B(\cdot, \varphi, W, s) = 0$. More specifically, we are able to verify that $B(1, \varphi, W, s)$ is a nonzero multiple of $Z(s,W)$.

\subsection{The Split Case} % (fold)
\label{sec:the_split_case_intertwinning_map}

Suppose that $X = X_{M}$ as in Section \ref{sec:notation}. Let $\tau_1$ and $\tau_2$ be irreducible admissible representations of $\gl(2,L)$ with trivial central character. Suppose the space of $\tau_i$ is its Whittaker model $\mathcal W_{\tau_i}$ and further suppose that there are non-negative integers $n_i$ such that $W_i \in \mathcal W_{\tau_i}$ that is invariant under $\Gamma_0({\mathfrak p}^{n_i})$ for $i \in \{1,2\}$. Set $W = W_1 \otimes W_2$. Let $\varphi_1$ and $\varphi_2$ be in $\mathcal S(X)$ and let $\varphi = T(\varphi_1 \otimes \varphi_2) \in \mathcal S(X^2)$. Set $N = n_1 + n_2$, which we will show is the paramodular level of the lift.  Using the Lemma \ref{seesaw_lemma} and Lemma \ref{lemma:paramodular_generators} we can easily calculate the action of the generators of $\K({\mathfrak p}^N)$ on $\varphi$. With these calculations in mind we can specify $\varphi_1$ and $\varphi_2$ so that $B(\cdot, \varphi, W,s)$ is paramodular invariant and non-zero. Choose
\begin{equation}
\label{eq:phi_choice_split}
	\varphi_1 = f_{\left[\begin{smallmatrix}
		{\mathfrak p}^{n_2} & {\mathfrak o}_L\\ {\mathfrak p}^N & {\mathfrak p}^{n_1}
	\end{smallmatrix}\right]} \quad \text{and} \quad \varphi_2 = f_{\mat(2,{\mathfrak o}_L)},
\end{equation} 
where $f_S$ is the characteristic function of a set S.

\begin{lemma}
\label{weil_paramodular_invariance_lemma_split}
	Let $\varphi$ be as in \eqref{eq:phi_choice_split}. For every $(k,h) \in R$ such that $k \in \K({\mathfrak p}^N)$ we have that $\omega(k,h)\varphi = \varphi$.
\end{lemma}

\begin{proof} Let us go through the generators listed in Lemma \ref{lemma:paramodular_generators}.
	\begin{enumerate}
		\item[(a)] Let  $A = [\begin{smallmatrix}
		a_1&a_2\\
		a_3&a_4
	\end{smallmatrix}] \in \Gamma_0({\mathfrak p}^N) \subset \gl(2,{\mathfrak o}_L)$ so that $a_1,a_4 \in {\mathfrak o}_L^\times, a_3 \in {\mathfrak p}^N$ and $a_2 \in {\mathfrak o}_L$. Note that $A^{-1} = \frac{1}{\det(A)}\left[\begin{smallmatrix}
		a_4&-a_2\\
		-a_3&a_1
	\end{smallmatrix}\right] \in \Gamma_0({\mathfrak p}^{N})$. If $(x,y) \in \supp(\varphi)$ then 
	\begin{align*}
		a_1x + a_3y \in [\begin{smallmatrix}
			{\mathfrak p}^{n_2} & {\mathfrak o}_L\\ {\mathfrak p}^N & {\mathfrak p}^{n_1}
		\end{smallmatrix}] \quad \text{and} \quad a_2x + a_4y \in \mat(2,{\mathfrak o}_L)
	\end{align*}
	so that $(a_1x + a_3y , a_2x + a_4y) \in \supp(\varphi)$. On the other hand, if $(a_1x + a_3y , a_2x + a_4y) \in \supp(\varphi)$ then 
	\begin{align*}
		& A^{-1}\cdot (a_1x + a_3y , a_2x + a_4y)
  \\ 
    =& \frac{1}{\det(A)}(a_4(a_1x + a_3y) + -a_3(a_2x + a_4y), -a_2(a_1x + a_3y) + a_1(a_2x + a_4y))
  \\
		=& ( x, y).
	\end{align*} Since $A^{-1} \in \Gamma_0({\mathfrak p}^N)$ we conclude that $(x,y) \in \supp(\varphi)$.
	So, $(x,y) \in \supp(\varphi)$ if and only if $(a_1x + a_3y , a_2x + a_4y) \in \supp(\varphi)$. Since $\det(A) = a_1a_4 - a_2a_3 \in {\mathfrak o}_L^\times$ and $\chi_{E/L}$ is trivial in the split case we see that 
	\begin{align*}
		\omega([\begin{smallmatrix}
			A&\\
			&{}^tA^{-1}
		\end{smallmatrix}] , 1)\varphi(x,y) & =  \chi_{E/L}(\det A)|\det A|^2 \varphi(a_1x + a_3y, a_2x + a_4y)\\
		& = \varphi(x,y).
	\end{align*}

Let $u \in {\mathfrak o}_L^\times$ and set
\begin{equation*}
	g_u = \left[\begin{smallmatrix}
      1&&&\\
      &1&&\\
      &&u&\\
      &&&u
    \end{smallmatrix}\right] \quad \text{and} \quad h_u = \rho([\begin{smallmatrix}
		u&\\
		&1
	\end{smallmatrix}], [\begin{smallmatrix}
		1&\\
		&1
	\end{smallmatrix}])
\end{equation*}
so that $(g_u,h_u) \in R$. Since $u \in {\mathfrak o}_L^\times$ we have $(h_u^{-1} x, h_u^{-1} y) \in \supp(\varphi)$ if an only if $(x,y) \in \supp(\varphi)$. Hence,
\begin{align*}
	\omega(g_u,h_u)\varphi(x,y) & = |u|^{-2} \omega(1,1) \varphi(h_u^{-1} x, h_u^{-1} y)\\
	& = \varphi(x,y).
\end{align*} 

\item[(b)] Let $B  = [\begin{smallmatrix}
	b_1&b_2\\b_2&b_3
\end{smallmatrix}] \in [\begin{smallmatrix}
	{\mathfrak p}^{-N}&{\mathfrak o}_L\\
	{\mathfrak o}_L & {\mathfrak o}_L
\end{smallmatrix}]$. If $(x,y) \in \supp(\varphi)$ then $\langle x , x \rangle \in {\mathfrak p}^N$ and $\langle x,y \rangle, \langle y , y \rangle \in {\mathfrak o}_L$. Hence,
\begin{align*}
	\omega([\begin{smallmatrix}
		1&B\\&1
	\end{smallmatrix}], 1) \varphi(x,y) & = \psi(b_1 \langle x , x \rangle + 2b_2 \langle x , y \rangle + b_3 \langle y , y \rangle) \varphi(x,y)\\
	& = \varphi(x,y).
\end{align*}

\item[(c)] In the proof of Lemma \ref{Haar_measure_constant_lemma} we verified that $\mathcal F_1(\varphi_2) = \varphi_2$. 
Now, using Lemma \ref{seesaw_lemma} we find that 
\begin{align*}
	\omega(s_2,1) \varphi(x,y) &= \omega_1(1,1)\varphi_1(x) \omega_1([\begin{smallmatrix}
		&1\\-1&
	\end{smallmatrix}],1)\varphi_2(y)\\
	& = \varphi_1(x) \mathcal F_1(\varphi_2)(y)\\
	& = \varphi(x,y).
\end{align*}

\item[(d)] To determine $\omega(t_N,1)$ we need the following preliminary calculation. Let $\varphi_1^{a}(x) = \varphi_1(ax)$. We claim that that $\mathcal F_1(\varphi_1^{\varpi^N})(x) = q^{2N} \varphi_1(x)$. The simpliest way to do this is to use \eqref{eq:ip_expand_split} along with the corresponding considerations for the Haar measure made in Lemma \ref{Haar_measure_constant_lemma}. With these in mind we calculate that
\begin{align*}
		\mathcal F_1(\varphi_1^{\varpi^N})(x)& = \int\limits_{X} \varphi_1(\varpi^N y) \psi(2  \langle x , y \rangle) \,dy = \int \limits_{\left[ \begin{smallmatrix}
		{\mathfrak p}^{-n_1} & {\mathfrak p}^{-N}\\
		{\mathfrak o}_L & {\mathfrak p}^{-n_2}
	\end{smallmatrix} \right]} \psi(2 \langle x , y \rangle) \,dy\\
	& = \int\limits_{{\mathfrak p}^{-n_1}} \psi(x_4y_1) \,dy_1 \int\limits_{{\mathfrak p}^{-N}} \psi(-x_3y_2) \,dy_2 \int\limits_{{\mathfrak o}_L} \psi(-x_2y_3) \,dy_3 \int\limits_{{\mathfrak p}^{-n_2}} \psi(x_1y_4) \,dy_4.  
\end{align*}
By Lemma \ref{simplegausschilemma}, we conclude that the above is  $\mathcal F_1(\varphi_1^{\varpi^N})(x) = q^{2N} \varphi_1(x).$ Now, again using the seesaw embedding in Lemma \ref{seesaw_lemma} we have
\begin{align*}
	\omega(t_n, 1)\varphi(x,y) &= \omega_1([\begin{smallmatrix}
		&\varpi^{-N}\\
		-\varpi^N&
	\end{smallmatrix}],1)\varphi_1(x) \varphi_2(y)\\
	& = q^{-2N} \mathcal F_1(\varphi_1^{\varpi^{N}})(x) \varphi_2(y)\\ 
	&= \varphi(x,y).
\end{align*}
	\end{enumerate}
  This completes the proof.
	  \end{proof}

\begin{corollary}
\label{cor:bessel_invariance_split}
	For all $k \in \K({\mathfrak p}^N)$ and all $g \in \gsp(4,L)$ we have that $B(gk,\varphi,W,s) = B(g,\varphi,W,s)$.
\end{corollary}
\begin{proof}
	When $\lambda(k) = 1$ the result follows from Lemma \ref{weil_paramodular_invariance_lemma_split}, the definition of $B(\cdot, \varphi,W,s)$, and the fact that for every $u \in {\mathfrak o}_L^\times$ there is some $b \in \Gamma_0({\mathfrak p}^n)$ such that $\det(b) = u$. For a general $k \in K({\mathfrak p}^N)$ we have to do a bit more work. Suppose that $\lambda(k) = u \in {\mathfrak o}_L^\times$ and let $g \in \gsp(4,L)$ and let $g_1 \in \gsp(4,L)$ be as in \eqref{eq:g0} then 
	\begin{align*}
		B(gk) & = |\lambda(g)|^{-s + \frac{1}{2}} B(g_2k_1) = |\lambda(g)|^{-s + \frac{1}{2}} B(g_2)
	\end{align*}
	where 
	\begin{equation*}
		k_1 = \left[\begin{smallmatrix}
			u^{-1}&&&\\
			&1&&\\
			&&1&\\
			&&&u^{-1}
		\end{smallmatrix}\right] k \quad \text{and} \quad g_2 =\left[\begin{smallmatrix}
			u^{-1}&&&\\
			&1&&\\
			&&1&\\
			&&&u^{-1}
		\end{smallmatrix}\right] g_1 \left[\begin{smallmatrix}
			u&&&\\
			&1&&\\
			&&1&\\
			&&&u
		\end{smallmatrix}\right].
	\end{equation*}
	It remains to verify that $B(g_2) = B(g_1)$. To do this we only need to check that $\omega(g_1,1) \varphi = \omega(g_2 , 1) \varphi$ for the $\varphi$ chosen in \eqref{eq:phi_choice_split} and for each of the generators with similitude factor equal to 1, namely the the elements of $\gsp(4,L)$ found in \eqref{weilrep2}, \eqref{weilrep3}, and \eqref{weilrep4}.
	\begin{enumerate}
		\item[(a)] Suppose that
		\begin{equation*}
			g_1 = \begin{bmatrix} 
		    [\begin{smallmatrix} a_1 & a_2 \\ a_3 & a_4 \end{smallmatrix}]&  \\ 
		    &{[\begin{smallmatrix} a_1 & a_2 \\ a_3 & a_4 \end{smallmatrix}]^t}^{-1} 
		    \end{bmatrix} \quad \text{so that} \quad g_2 = \begin{bmatrix} 
		    [\begin{smallmatrix} a_1 & a_2u^{-1} \\ a_3u & a_4 \end{smallmatrix}]&  \\ 
		    &{[\begin{smallmatrix} a_1 & a_2u^{-1} \\ a_3u & a_4 \end{smallmatrix}]^t}^{-1} 
		    \end{bmatrix}.
		\end{equation*}
		With $\varphi$ chosen as in \eqref{eq:phi_choice_split} it is clear that $\omega(g_1,1) \varphi = \omega(g_2 , 1) \varphi$ by examining \eqref{weilrep2}. 

		\item[(b)] Suppose that
		\begin{equation}
			g_1 = \left[\begin{smallmatrix}
				1&&b_2&b_2\\
				&1&b_2&b_3\\
				&&1&\\
				&&&1
			\end{smallmatrix}\right] \quad \text{so that} \quad g_2 = \left[\begin{smallmatrix}
				1&&b_1u^{-1}&b_2\\
				&1&b_2&b_3u\\
				&&1&\\
				&&&1
			\end{smallmatrix}\right].
		\end{equation}
		With $\varphi$ chosen as in \eqref{eq:phi_choice_split} it is clear that $\omega(g_1,1) \varphi = \omega(g_2 , 1) \varphi$ by examining \eqref{weilrep3}

		\item[(c)] Suppose that $g_1 = J$ so that 
		$g_2  = J \left[\begin{smallmatrix}
				u&&&\\
				&u^{-1}&&\\
				&&u^{-1}&\\
				&&&u
			\end{smallmatrix}\right]$.
		For $\varphi$ as chosen in \eqref{eq:phi_choice_split} we have that $\varphi(u^{-1}x , u y) = \varphi(x,y)$ so we conclude that
		\begin{align*}
			\omega(g_2) \varphi(x,y) & = \omega(J,1) \omega(\left[\begin{smallmatrix}
				u&&&\\
				&u^{-1}&&\\
				&&u^{-1}&\\
				&&&u
			\end{smallmatrix}\right]) \varphi(x,y)= \omega(J,1) \varphi(u^{-1}x , uy)= \omega(J,1) \varphi(x,y).
		\end{align*}
	\end{enumerate}
  This completes the proof.
  \end{proof}

\begin{lemma}
\label{lemma:H_coset_split}
	Let $x_1,x_2$ and $H$ be as in Section \ref{sec:notation}. Let $h = \rho(h_1,h_2) \in \so(X)$ then $h^{-1}(x_1,x_2) \in \supp(\varphi)$ if and only if there is some $h' = \rho(h_1',h_2') \in H$ such that $h_1'h_1 \in \Gamma_0({\mathfrak p}^{n_1})$ and $h_2'h_2 \in \left[  \begin{smallmatrix}
		2^{-1} &\\ & 1
	\end{smallmatrix}\right] \Gamma_0({\mathfrak p}^{n_2})$.
\end{lemma}
\begin{proof}
	Let $m \in \mathbb{Z}_{\ge 0}$ be such that $2 {\mathfrak o}_L = {\mathfrak p}^m$. First, suppose that $h = \rho(h_1,h_2)$ where $ h_1\in \Gamma_0({\mathfrak p}^{n_1})$ and $h_2 \in \left[  \begin{smallmatrix}
		2^{-1} &\\ & 1
	\end{smallmatrix}\right] \Gamma_0({\mathfrak p}^{n_2})$ where $ h_1 = [\begin{smallmatrix}
		a_1 & b_1\\ c_1 & d_1
	\end{smallmatrix}],$ and $h_2 = [\begin{smallmatrix}
		a_2/2 & b_2/2\\ c_2& d_2
	\end{smallmatrix}]$, for some $a_i, b_i , c_i , d_i \in L$, for $i \in \{1,2\}$. Then we have  
		$h^{-1} (x_1) = [\begin{smallmatrix}
			d_1c_2 & d_1d_2\\
			-c_1c_2 & -c_1d_2
		\end{smallmatrix}] \in \supp(\varphi_1)$ and $h^{-1}(x_2) = [\begin{smallmatrix}
			b_1a_2 & b_1b_2\\
			-a_1a_2 & -a_1b_2
		\end{smallmatrix}] \in \supp(\varphi_2)$.
  On the other hand, assume that $h^{-1}(x_1,x_2) \in \supp(\varphi)$. Then we have the following congruences $
  		c_1c_2 \in {\mathfrak p}^N, c_1d_2 \in {\mathfrak p}^{n_1}, d_1c_2 \in {\mathfrak p}^{n_2}, d_1d_2 \in {\mathfrak o}_L, a_1a_2\in {\mathfrak o}_L, a_1b_2\in {\mathfrak o}_L, b_1a_2\in {\mathfrak o}_L,$ and $ b_1b_2 \in {\mathfrak o}_L.
	$ Choose $r_1 \in \mathbb{Z}$ maximally so that $c_1 \in {\mathfrak p}^{r_1}$ and $d_1 \in {\mathfrak p}^{r_1 - n_1}$. Since at least one of $c_1$ and $d_1$ is not equal to $0$ we know that $c_2 \in {\mathfrak p}^{N - r_1} = {\mathfrak p}^{n_2 - (r_1 - n_1)}$ and $d_2 \in {\mathfrak p}^{n_1 - r_1} = {\mathfrak p}^{0 - (r_1 - n_1)}$. Similarly, choose $r_2 \in \mathbb{Z}$ maximally so that $a_1,b_1 \in {\mathfrak p}^{r_2}$. Since at least one of $a_1$ and $b_1$ is not equal to $0$ we know that $a_2,b_2 \in {\mathfrak p}^{-r_2-m}$.
	Therefore $h_1 \in [\begin{smallmatrix}
			{\mathfrak p}^{r_2}& {\mathfrak p}^{r_2}\\
			{\mathfrak p}^{r_1}& {\mathfrak p}^{r_1 - n_1}
		\end{smallmatrix}]$ and $h_2 \in [\begin{smallmatrix}
			2^{-1} {\mathfrak p}^{-r_2}& 2^{-1}{\mathfrak p}^{-r_2}\\
			{\mathfrak p}^{N - r_1} & {\mathfrak p}^{n_1 - r_1}
		\end{smallmatrix})] \cap \so(X)$.
We choose $h_1' = [\begin{smallmatrix}
			\varpi^{-r_2} & \\
			& \varpi^{n_1 - r_1}
		\end{smallmatrix}]$ and $h_2' =  [\begin{smallmatrix}
			\varpi^{r_2}&\\
			& \varpi^{r_1 - n_1}
		\end{smallmatrix}]$ so that $h' = \rho(h_1',h_2') \in H$
and easily see that that $h_1'h_1 \in \Gamma_0({\mathfrak p}^{n_1})$ and that $h_2'h_2 \in   \left[\begin{smallmatrix}
			2^{-1}&\\&1
		\end{smallmatrix} \right]\Gamma_0({\mathfrak p}^{n_2})$
\end{proof}
The following theorem establishes our main local result in the split case. 
\begin{theorem}\label{MT_split}
	Suppose that $E/L$ is split and that $W \in \mathcal W_{\tau_1} \otimes \mathcal W_{\tau_2}$ is $\Gamma_0({\mathfrak p}^{n_1}) \times \Gamma_0({\mathfrak p}^{n_2})$-invariant. Let $\varphi$ be as in \eqref{eq:phi_choice_split} and assume that $\Re(s)> M$, then the intertwining map $B$ defined in \eqref{eq:splitbesselintegral} is non-zero and $\K({\mathfrak p}^N)$-invariant. In particular $B(1,\varphi,W,s) \ne 0$.
\end{theorem}
\begin{proof}
By Corollary \ref{cor:bessel_invariance_split} We have already show that $B(\cdot,\varphi,W,s)$ is paramodular invariant. Lemma \ref{lemma:H_coset_split} can be used determine the support of $B(1, \varphi, W, s)$. Indeed, we see that
\begin{align*}
	B(1,\varphi,W,s) & = \int\limits_{H \backslash \so(X)} \omega(1,h) \varphi(x_1,x_2) Z(s,\pi(h)W) = \int\limits_{H \backslash \rho(\Gamma_0({\mathfrak p}^{n_1}) \times \left[\begin{smallmatrix}
			2^{-1}&\\&1
		\end{smallmatrix}\right]\Gamma_0({\mathfrak p}^{n_2}))} Z(s,\pi(h) W)\\
	& = \vol[H\backslash \rho(\Gamma_0({\mathfrak p}^{n_1}) \times \left[\begin{smallmatrix}
			2^{-1}&\\&1
		\end{smallmatrix}\right]\Gamma_0({\mathfrak p}^{n_2}))] \cdot |2|^{s - \frac{1}{2}} \cdot Z(s, W)\neq 0.
\end{align*}
Note that the additional constant $|2|^{s - \frac{1}{2}}$ comes from an application of Lemma \ref{zetaintegrals} part (a). 
  \end{proof}

% section the_split_case (end)

\subsection{The Inert Case} % (fold)
\label{sec:the_inert_case_intertwinning_mapc}

We now let $X=X_{ns}$ and apply the same method to prove that if $E/L$ is inert that there exists a $\varphi \in \mathcal S(X^2)$ such that $B(\cdot,\varphi,W,s) \ne 0$, in particular we will show that that $B(1) \ne 0$. Suppose that $1 \ne \delta \in {\mathfrak o}_L$ is square-free and that the field extensions $E = L(\sqrt{\delta})$ is inert. In particular this means that $\delta \in {\mathfrak o}_L^\times$. Let $\tau_0$ be an irreducible admissible representation of $\gl(2,E)$ with trivial central character. We assume that the space of $\tau_0$ is its Whittaker model $\mathcal W_{\tau_0}$ and that there is some $W \in \mathcal W_{\tau_0}$ that is $\Gamma_0({\mathfrak p}^n)$-invariant, for some non-negative integer $n$. Set $N = 2n$ which will be the paramodular level of $B(\cdot, \varphi, W, s)$. With this information we choose Schwartz function $\varphi = T(\varphi_1 \otimes \varphi_2)$ where 
\begin{equation}
	\label{eq:phi_choice_inert}
	\varphi_1 = f_{\left[\begin{smallmatrix}
		{\mathfrak p}^n & {\mathfrak o}_L\\
		{\mathfrak p}^N & {\mathfrak p}^n
	\end{smallmatrix}\right] \cap X} \quad \text{and} \quad \varphi_2 = f_{\mat(2,{\mathfrak o}_E) \cap X}, 
\end{equation} 
where $f_S$ is the characteristic function of a set S.
Evidently we see that we can also write the support of the $\varphi_i$ as
\begin{equation}
	\supp(\varphi_1) = \{[\begin{smallmatrix}
		x_1 + x_2 \sqrt{\delta} & x_3\\ x_4 & x_1 - x_2 \sqrt{\delta}
	\end{smallmatrix}] \mid x_1 \in {\mathfrak p}^{n}, x_2 \in {\mathfrak p}^{n}, x_3 \in {\mathfrak o}_L , x_4 \in {\mathfrak p}^{N}\}
\end{equation}
and
\begin{equation}
	\supp(\varphi_2) = \{[\begin{smallmatrix}
		x_1 + x_2 \sqrt{\delta} & x_3\\ x_4 & x_1 - x_2 \sqrt{\delta}
	\end{smallmatrix}] \mid x_1,x_2,x_3,x_4 \in {\mathfrak o}_L\}.
\end{equation}
\begin{lemma}
	\label{weil_paramodular_invariance_lemma_inert}
	Let $\varphi$ be as chosen above. Then, for every $(k,h) \in R$ such that $k \in \K({\mathfrak p}^N)$ we have that $\omega(k,h)\varphi = \varphi$.
\end{lemma}
\begin{proof}
	Just as in the split case it suffices to check this for each of the of the generators of $\K({\mathfrak p}^N)$, as listed in Lemma \ref{lemma:paramodular_generators}.

\begin{enumerate}
	\item[(a)] Let  $A = [\begin{smallmatrix}
		a_1&a_2\\
		a_3&a_4
	\end{smallmatrix}] \in \Gamma_0({\mathfrak p}^N) \subset \gl(2,{\mathfrak o}_L)$ for some choices of $a_i \in {\mathfrak o}_L$. Note that $A^{-1} = \frac{1}{\det(A)}\left[\begin{smallmatrix}
		a_4&-a_2\\
		-a_3&a_1
	\end{smallmatrix}\right] \in \Gamma_0({\mathfrak p}^{N})$. If $(x,y) \in \supp(\varphi)$ then 
		$a_1x + a_3y \in \left[\begin{smallmatrix}
			{\mathfrak p}^{n} & {\mathfrak o}_L\\ {\mathfrak p}^N & {\mathfrak p}^{n}
		\end{smallmatrix}\right]$ and a$ a_2x + a_4y \in \mat(2,{\mathfrak o}_L)$
	so that $(a_1x + a_3y , a_2x + a_4y) \in \supp(\varphi)$. On the other hand, if $(a_1x + a_3y , a_2x + a_4y) \in \supp(\varphi)$ then 
	\begin{align*}
		A^{-1} & \cdot (a_1x + a_3y , a_2x + a_4y)\\ &= \frac{1}{\det(A)}(a_4(a_1x + a_3y) + -a_3(a_2x + a_4y), -a_2(a_1x + a_3y) + a_1(a_2x + a_4y))\\
		& = (x,y).
	\end{align*} Since $A^{-1} \in \Gamma_0({\mathfrak p}^N)$ we conclude that $(x,y) \in \supp(\varphi)$.
	So, we have just proved that $(x,y) \in \supp(\varphi)$ if and only if $(a_1x + a_3y , a_2x + a_4y) \in \supp(\varphi)$. Recall that when $E/L$ is unramified that $\norm_L^E: {\mathfrak o}_E^\times \to {\mathfrak o}_L^\times$ is surjective. Since $\det(A) \in {\mathfrak o}_E^\times$ we see that 
	\begin{align*}
		\omega([\begin{smallmatrix}
			A&\\
			&{}^tA^{-1}
		\end{smallmatrix}] , 1)\varphi(x,y) & =  \chi_{E/L}(\det(A))|\det(A)|^2 \varphi(a_1x + a_3y, a_2x + a_4y)\\
		& = \varphi(x,y).
	\end{align*}

	Let $u \in {\mathfrak o}_L^\times$ such that $g_u = \left[\begin{smallmatrix}
		1&&&\\
		&1&&\\
		&&u&\\
		&&&u
	\end{smallmatrix}\right] \in \gsp(4)^+$. Choose $u_E \in {\mathfrak o}_E^\times$ such that $\norm_L^E(u_E) = u$ and set $h_u = \rho(1,[\begin{smallmatrix}
		u_E&\\&1
	\end{smallmatrix})]$. Then $(g_u,h_u) \in R$. Since $u_E$ is a unit in ${\mathfrak o}_E$ we can conclude that
	\begin{align*}
	 	\omega(g_u,h_u)\varphi(x,y) &= |u|^{-2} \varphi(h_u^{-1}x , h_u^{-1}y)\\
	 	& = \varphi(x,y)
	\end{align*} 
	for all $(x,y) \in X^2$.

	 \item[(b)] Suppose that $B = [\begin{smallmatrix}
	 	b_1&b_2\\ b_2&b_3
	 \end{smallmatrix}] \in [\begin{smallmatrix}
	 	{\mathfrak p}^{-N} & {\mathfrak o}_L\\
	 	{\mathfrak o}_L & {\mathfrak o}_L
	 \end{smallmatrix}]$. Then for all $(x,y) \in \supp(\varphi)$ we can easily calculate that 
	\[
		b_1 \langle x,x \rangle + 2b_2 \langle x , y \rangle + b_3 \langle y,y \rangle \in {\mathfrak o}_L, \quad \text{so that}
	\]
	\begin{align*}
		\omega([\begin{smallmatrix}
			1 & B \\ & 1
		\end{smallmatrix}],1) \varphi(x,y) & = \psi(b_1 \langle x,x \rangle + 2b_2 \langle x , y \rangle + b_3 \langle y,y \rangle) \varphi(x,y)\\ & = \varphi(x,y).
	\end{align*}

	\item[(c)] As in part $(c)$ in the proof of the the split case (Lemme \ref{weil_paramodular_invariance_lemma_split}) it suffices to show that $\mathcal F_1(\varphi_2) = \varphi_2$, which we did in the proof of Lemma \ref{Haar_measure_constant_lemma}.

	\item[(d)] Again, similar to the split case it suffices to show that $\mathcal F_1(\varphi_1^{\varpi^N}) = q^{2N} \varphi_1$. For a generic element $y \in X$ we write $y = \left[ \begin{smallmatrix}
		y_1 + y_2 \sqrt{\delta} & y_3\\
		y_4 & y_1 - y_2 \sqrt{\delta}
	\end{smallmatrix}  \right]$ with $y_i \in L$ for $i \in \{1,2,3,4\}$. Using \eqref{eq:ip_expand_nonsplit} and the corresponding considerations for the Haar measure in Lemma \ref{Haar_measure_constant_lemma} we can calculate this as follows:
	\begin{align*}
		(\mathcal F_1 \varphi_1^{\varpi^N})(x) & = \int\limits_X \varphi_1(\varpi^N x) \psi(2 \langle x , y \rangle) \,dy = \int\limits_{ \left [ \begin{smallmatrix}
			{\mathfrak p}^{-n} & {\mathfrak p}^{-N}\\
			{\mathfrak o}_E & {\mathfrak p}^{-n}
		\end{smallmatrix} \right] \cap X} \psi(2 \langle x , y  \rangle ) \,dy\\
		& = \int\limits_{{\mathfrak p}^{-n}} \psi(2x_1y_1) \,dy_1 \int\limits_{{\mathfrak p}^{-n}} \psi(-2x_2y_2 \delta) \,dy_2 \int\limits_{{\mathfrak p}^{-N}} \psi(-x_4y_3) \,dy_3 \int\limits_{{\mathfrak o}_L} \psi(-x_3y_4 \delta) \,dy_4 .
	\end{align*} 
    \end{enumerate}
    Using Lemma \ref{simplegausschilemma} we can conclude that the above is 
    $(\mathcal F_1 \varphi_1^{\varpi^N})(x) = q^{2N} \varphi_2(x)$
  \end{proof}

\begin{corollary}
	\label{cor:Bessel_invariance_inert}
	For all $k \in \K({\mathfrak p}^N)$ and all $g \in \gsp(4,L)$ we have that $B(gk,\varphi,W,s) = B(g,\varphi,W,s)$.
\end{corollary}
\begin{proof}
	The proof is very similar to the proof of Corollary \ref{cor:bessel_invariance_split}. 
  \end{proof}

\begin{lemma} \label{lemma:H_coset_inert} 
	Let $x_1,x_2$ and $H$ be as in Section \ref{sec:notation} and let $m \in \mathbb{Z}$ be such that $2 {\mathfrak o}_L = {\mathfrak p}^m$ and set $\bar m = \lfloor \frac{m}{2} \rfloor$. Let $t \in L^\times$ and $h_0 \in \gl(2,E)$ and set $h = \rho(t, h_0) \in \so(X)$. Then, $h^{-1}(x_1,x_2) \in \supp(\varphi)$ if and only if there is some $t' \in L^\times$ and $h'_0 \in \gl(2, E)$ such that $h' = \rho(t', h'_0) \in H$, $t't = 1$, and $h'_0h_0 = \left[ \begin{smallmatrix}
		\varpi^{-\bar m} & \\ & 1
	\end{smallmatrix} \right] A$, for some $A \in \Gamma_0({\mathfrak p}^n)$.
\end{lemma}
\begin{proof}
	First suppose that there is such an $h'$ and an $A = \left[\begin{smallmatrix}
      \varpi^{-\bar m}a& \varpi^{- \bar m}b\\c&d
    \end{smallmatrix}\right] \in \left[ \begin{smallmatrix}
		\varpi^{-\bar m} & \\ & 1
	\end{smallmatrix} \right] \Gamma_0({\mathfrak p}^n)$ such that $h'h = \rho(1,A)$. We calculate
	\begin{align} 
		h^{-1} x_1 & = h^{-1} (h'^{-1} x_1)= (h'h)^{-1} x_1= A x_1 \alpha(A)^* \label{eq:hmat1}
		= \sqrt{\delta}[\begin{smallmatrix}
			d \alpha(c)  & d \alpha(d) \\
			-c \alpha(c)  & -c \alpha(d) 
		\end{smallmatrix}]
		\intertext{and,}
		h^{-1} x_2 & = h^{-1} (h'^{-1} x_2) = (h'h)^{-1} x_2= A x_2 \alpha(A)^*\label{eq:hmat2}
		 = \frac{\sqrt{\delta}}{\delta} \left[\begin{smallmatrix}
			\varpi^{m - 2\bar m}b \alpha(a)  & \varpi^{m - 2\bar m} b \alpha(b) \\
			-\varpi^{m - 2\bar m} a \alpha(a)  & -\varpi^{m - 2\bar m} a \alpha(b) 
		\end{smallmatrix}\right].
	\end{align}
	Since $m - 2\bar m \ge 0$ we conclude that $h^{-1}(x_1,x_2) \in \supp(\varphi)$. 

	On the other hand, let $h = \rho(1,A)$ with $A = \left[\begin{smallmatrix}
		\varpi^{-\bar m}a& \varpi^{- \bar m}b\\c&d
	\end{smallmatrix} \right] \in \gl(2,E)$ and assume that $h^{-1}(x_1,x_2) \in \supp (\varphi)$. From \eqref{eq:hmat1} and \eqref{eq:hmat2}, we immediately find the following congruences
	\begin{align*}
		& d \alpha(c) \in {\mathfrak p}^n, d \alpha (d) \in {\mathfrak o}_L, c \alpha(c) \in {\mathfrak p}^{N}\\
		& b \alpha(a), a \alpha(a), b \alpha(b) \in {\mathfrak p}^{2\bar m - m} = \begin{cases}
			{\mathfrak o}_L; & m \ \text{is even} \\
			{\mathfrak p}^{-1}; & m \ \text{is odd}.
		\end{cases}
	\end{align*}
	Since $|\alpha(x)| = |x|$ for all $x \in E$ we conclude that $a,b,d \in {\mathfrak o}_L$ and $c \in {\mathfrak p}^n$. Let $t \in L^\times$ and let $A$ be as above. Set $h = \rho(t,A)$ and suppose that $h^{-1}(x_1,x_2) \in \supp(\varphi)$. If $\nu(t)$ is odd then $\nu(c \alpha (c))$ is odd, which is impossible. Then $t = \varpi^{2k}u$ for some $k \in \mathbb{Z}$ and some $u \in {\mathfrak o}_L^\times$. Set $h' = \rho(t^{-1},[\begin{smallmatrix}
		\varpi^{k}&\\&\varpi^{k}u
	\end{smallmatrix}]) \in H$. A priori, $h'h \in \rho(1,\gl(2,E))$ and $(h'h)^{-1}(x_1,x_2) = h^{-1}(x_1,x_2) \in \supp(\varphi)$. By the above calculation when $t = 1$ we can conclude that $h'h \in \rho(1,  \left[ \begin{smallmatrix}
		\varpi^{-\bar m} & \\ & 1
	\end{smallmatrix} \right] \Gamma_0({\mathfrak p}^n))$.
  \end{proof}
  The following theorem establishes our main local result in the inert case. 
\begin{theorem} \label{MT_inert}
	Suppose that $E/L$ is inert and that $W \in \mathcal W_{\tau_0}$ has $\Gamma_0({\mathfrak p}^n)$-invariance. Let $\varphi$ be  as in \eqref{eq:phi_choice_inert} and assume that $s \gg M$, then the intertwining map $B$ defined in \eqref{eq:nonsplitbesselintegral} is non-zero and $K({\mathfrak p}^N)$-invariant. In particular $B(1,\varphi,W,s) \ne 0$.
\end{theorem}
\begin{proof}
By Corollary \ref{cor:Bessel_invariance_inert} We have already show that $B(\cdot,\varphi,W,s)$ is paramodular invariant. Lemma \ref{lemma:H_coset_inert} can be used to determine the support of $B(1,\varphi,W,s)$. Indeed, we see that
\begin{align*}
	B(1,\varphi,W,s) &= \int\limits_{H\backslash \so(X)} \omega(1,h) \varphi(x_1,x_2) Z(s, \pi(h)W) \, dh= \int\limits_{H \backslash \rho(1, \left[ \begin{smallmatrix}
		\varpi^{-\bar m} & \\ & 1
	\end{smallmatrix} \right] \Gamma_0({\mathfrak p}^n))} Z(s,\pi(h)W) \,dh\\
	& = \vol[H \backslash \rho(1, \left[ \begin{smallmatrix}
		\varpi^{-\bar m} & \\ & 1
	\end{smallmatrix} \right] \Gamma_0({\mathfrak p}^n))] \cdot |\varpi|^{\bar m(s - \frac{1}{2}) } \cdot Z(s,W) \ne 0.
\end{align*}
Note that the additional constant $|\varpi|^{\bar m(s - \frac{1}{2})}$ comes from an application of Lemma \ref{zetaintegrals} part (a). 
  \end{proof}

% section the_inert_case (end)

\subsection{The Ramified Case} % (fold)
\label{sec:the_ramified_case}

Suppose that $\delta \in {\mathfrak o}_L$ is square-free and that the field extensions $E = L(\sqrt{\delta})$ is ramified, so that $\delta \in {\mathfrak p}$. Let $\tau_0$ be an irreducible admissible representation of $\gl(2,E)$ with trivial central character. We assume that the space of $\tau_0$ is its Whittaker model $\mathcal W_{\tau_0}$ and that there is some $W \in \mathcal W_{\tau_0}$ that is invariant under $\Gamma_0({\mathfrak P}^n)$ for some non-negative integer $n$. Set $N = n + 2$, which we will prove is the paramodular level of $B(\cdot, \varphi, W, s)$. In the previous two sections we were able to find Schwartz functions for which $B(g,\varphi, W,s) \ne 0$ is paramodular invariant by inspection. In the ramified case we take a more systematic approach. Define $\tilde\varphi(x,y)  = T(\varphi_1(x) \otimes \varphi_2(y))$, where
	\begin{align}
 \label{eq:tilde_phi}
 \varphi_1(x) & = \chi(x_3)f_{\left[\begin{smallmatrix}
			{\mathfrak P}^{n + 1} &  {\mathfrak P}\\
			\varpi_E^{2n + 1} {\mathfrak o}_E^\times & {\mathfrak P}^{n + 1}
		\end{smallmatrix}\right] \cap X} (x) = \chi(x_3) f_{{\mathfrak P}^{n + 1}}(x_1) f_{{\mathfrak o}_L}(x_2) f_{\varpi_L^n {\mathfrak o_L^\times}}(x_3), 
	\\  
		\varphi_2(y) & = \chi(y_3) f_{\left[ \begin{smallmatrix}
			{\mathfrak P}^{-1} & {\mathfrak P}^{-1} \\ \varpi_E^{-1} {\mathfrak o}_E^\times & {\mathfrak P}^{-1}
		\end{smallmatrix} \right] \cap X} (y),
	 = \chi(y_3) f_{{\mathfrak P}^{-1}}(y_1) f_{{\mathfrak p}^{-1}} (y_2) f_{\varpi^{-1}_L {\mathfrak o}_L^\times}(y_3), \nonumber
\end{align}
with $x = \left[\begin{smallmatrix}
	x_1 & x_2 \sqrt{\delta} \\ x_3 \sqrt{\delta} & \alpha(x_1)
\end{smallmatrix}\right]$ and $y= \left[\begin{smallmatrix}
	y_1 & y_2 \sqrt{\delta} \\ y_3 \sqrt{\delta} & \alpha(y_1)
\end{smallmatrix}\right]$. 

Let $T$ be the subgroup of $\GSp(4,L)$ generated by
\begin{equation*}
	\{ \left[\begin{smallmatrix}
      A&\\&{}^tA^{-1}
    \end{smallmatrix}\right], \left[\begin{smallmatrix}
        1&B\\&1
      \end{smallmatrix}\right], \left[\begin{smallmatrix}
          1&\\
          C&1
        \end{smallmatrix}\right] \mid A \in \Gamma_{0}({\mathfrak p}^{N}), B \in \left[\begin{smallmatrix}
            {\mathfrak p}^{-N + 1} & {\mathfrak o}_L\\ {\mathfrak o}_L & {\mathfrak p} \end{smallmatrix}\right] \cap \sym(2), C \in \left[\begin{smallmatrix}
                  {\mathfrak p}^N & {\mathfrak p}^{N-1}\\ {\mathfrak p}^{N - 1} & {\mathfrak o}_L
                \end{smallmatrix}\right]\cap \sym(2)  \}.
\end{equation*}
The following lemma computes the Fourier transforms of $\varphi_1$ and $\varphi_2$.
\begin{lemma} \label{lemma:Fourier_identites}
	Let $\varphi_1$ and $\varphi_2$ be as above. Then
	\begin{align*}
		\mathcal F_1(\varphi_1)(\varpi^{-N} x) & = q_E^{-N} \mathcal F_1(\varphi_1^{\varpi^{N}})(x)  = q_E^{-N} \chi(x_2) f_{\left[\begin{smallmatrix}
					{\mathfrak P}^{n + 2} & \varpi_E {\mathfrak o}_E^\times \\ {\mathfrak P}^{2n + 3} & {\mathfrak P}^{n + 2}
				\end{smallmatrix}\right] \cap X}(x)\\
	\intertext{where $\varphi_1^a(x) = \varphi_1(ax)$, and}
		\mathcal F_1(\varphi_2)(x) & = \chi(x_2) f_{\left [\begin{smallmatrix}
					{\mathfrak o}_E & \varpi^{-1}_E {\mathfrak o}_E^\times \\ {\mathfrak P} & {\mathfrak o}_E
				\end{smallmatrix} \right] \cap X}(x). 
	\end{align*}
\end{lemma}
\begin{proof}
	First note that, since $E/L$ is ramified,  $\psi_E^{\varpi_E^{-1}}$ has conductor ${\mathfrak o}_E$. Evaluations of character sums and Gauss sums in this proof rely on the relevant formulas proven in Section \ref{sec:notation}. We begin with the easier calculation:
	\begin{align*}
		\mathcal F_1(\varphi_2) (x) & = \int\limits_{X} \varphi_2(y) \psi(2 \langle x , y \rangle) \, dy
	\\ 
		& = k\int\limits_{{\mathfrak P}^{-1}} \psi_E(x_1 \alpha (y_1)) \,dy_1 \int\limits_{{\mathfrak p}^{-1}} \psi(-\delta x_3y_3) \,dy_2 \int\limits_{\varpi_L^{-1} {\mathfrak o}_L^\times} \chi(y_3) \psi(-\delta x_2 y_3) \,dy_3
	\\
		& = k \int\limits_{{\mathfrak P}^{-1}} \psi_E^{\varpi^{-1}_E}( \varpi_E x_1 \alpha (y_1)) \,dy_1 \int\limits_{{\mathfrak p}^{-1}} \psi(-\delta x_3y_2) \,dy_2 \int\limits_{\varpi_L^{-1} {\mathfrak o}_L^\times} \chi(y_3) \psi(-\delta x_2 y_3) \,dy_3
	\\
		& = k'  \chi(x_2) f_{{\mathfrak o}_E}(x_1) f_{\varpi_L^{-1}{\mathfrak o}_L^\times}(x_2) f_{{\mathfrak o}_L}(x_3)  
	\\
		& = k' \chi(x_2) f_{\left[\begin{smallmatrix}
					{\mathfrak o}_E & \varpi^{-1}_E {\mathfrak o}_E^\times \\ {\mathfrak P} & {\mathfrak o}_E
				\end{smallmatrix}\right] \cap X}(x).
    \end{align*}
	We know that $(k')^4 = 1$ since $\mathcal F_1(\mathcal F_1 \varphi_2)(x) = \varphi_2(-x) = \chi(-1)\varphi_2(x)$. For $\varphi_1$ we see that the first equality is true because $[\begin{smallmatrix}
			& \varpi_L^{-N}\\ -\varpi_L^N &
		\end{smallmatrix}] = [\begin{smallmatrix}
			&1\\-1&
		\end{smallmatrix}][\begin{smallmatrix}
			\varpi_L^{-N} & \\ & \varpi_L^N
		\end{smallmatrix}] = [\begin{smallmatrix}
			\varpi_L^N & \\
			& \varpi_L^{-N}
		\end{smallmatrix}][\begin{smallmatrix}
			&-1\\1&
		\end{smallmatrix}]$. For the second equality we calculate that
		\begin{align*}\mathcal F_1(\varphi_1) (x) & = \int\limits_{X} \varphi_1(y) \psi(2 \langle x , y \rangle) \,dy
	\\
		& = k \int\limits_{{\mathfrak P}^{n + 1}} \psi_E^{\varpi^{-1}_E}( \varpi_E x_1 \alpha (y_1)) \,dy_1 \int\limits_{{\mathfrak o}_L} \psi(-\delta x_3y_2) \,dy_2 \int\limits_{\varpi_L^{n} {\mathfrak o}_L^\times} \chi(y_3) \psi(-\delta x_2 y_3) \,dy_3
	\\
		& = q_E^{-N} k' \chi(x_2) f_{{\mathfrak P}^{-n-2}}(x_1) f_{\varpi_L^{-n - 2}{\mathfrak o}_L^\times}(x_2) f_{{\mathfrak p}^{-1}}(x_3).
	\intertext{Therefore we have that}
		\mathcal F_1(\varphi_1)(\varpi_L^{-N}x) & = q_E^{-N} k' \chi(x_2) f_{{\mathfrak P}^{n+2}}(x_1) f_{{\mathfrak o}_L^\times}(x_2) f_{{\mathfrak p}^{n+1}}(x_3).
	\\
		& = q_E^{-N}\chi(x_2) f_{\left[\begin{smallmatrix}
					{\mathfrak P}^{n + 2} & \varpi_E {\mathfrak o}_E^\times \\ {\mathfrak P}^{2n + 3} & {\mathfrak P}^{n + 2}
				\end{smallmatrix}\right] \cap X}(x).
	\end{align*} 
	So, both equalities are verified.
  \end{proof}
\begin{lemma}
	Let $\tilde \varphi$ be as defined before \eqref{eq:tilde_phi}. For every $g \in T$ we have that $\omega(g, 1)\tilde\varphi = \tilde\varphi$. 
\end{lemma}
\begin{proof}
\begin{enumerate}
	\item[(a)] Let $A = [\begin{smallmatrix}
			a_1 & a_2\\ a_3 & a_4
		\end{smallmatrix}] \in \Gamma_0({\mathfrak p}^{N})$. Then $(x,y) \in \supp (\tilde\varphi)$ if and only if $(a_1x + a_3y, a_2x + a_4y) \in \supp (\tilde\varphi)$. Furthermore, Assuming that $(x,y) \in \supp (\tilde\varphi)$, we have that
	\begin{align*}
		& \chi((a_1x_3 + a_3y_3)(a_2x_3 + a_4y_3)) \\
		= & \chi(a_1a_2x_3^2 + 2a_2a_3 x_3y_3 + a_3a_4y_3^2 + x_3y_3 (\det A))\\
		= & \chi\left[ x_3y_3  \det A) \big( 1 + (x_3y_3\det A)^{-1} (a_1a_2x_3^2 + 2a_2a_3 x_3y_3 + a_3a_4y_3^2) \big) \right]\\
		= & \chi( x_3y_3\det A)
	\end{align*}
	since $(x_3y_3\det A)^{-1} (a_1a_2x_3^2 + 2a_2a_3 x_3y_3 + a_3a_4y_4^2) \in {\mathfrak p}$. Therefore
	\begin{align*}
		\omega([\begin{smallmatrix}
			A&\\&{}^tA^{-1}
		\end{smallmatrix}], 1) \tilde\varphi(x,y) & = \chi(\det A) \tilde\varphi(a_1 x + a_3 y, a_2x + a_4 y)\\
		& =  \tilde\varphi(x,y).
	\end{align*}
	
	\item[(b)] Let $B = [\begin{smallmatrix}
			b_1 & b_2 \\ b_2 & b_3
		\end{smallmatrix}] \in [\begin{smallmatrix}
			{\mathfrak p}^{-N + 1} & {\mathfrak o}_L\\
			{\mathfrak o}_L & {\mathfrak p}
		\end{smallmatrix}]$. If $(x , y ) \in \supp (\tilde\varphi)$ then we get the following congruences
		\begin{align*}
			\langle x , x \rangle \in {\mathfrak p}^{n + 1} = {\mathfrak p}^{N - 1}, \quad
			\langle y , y \rangle &\in {\mathfrak p}^{-1}, \quad \text{and} \quad
			\langle x , y \rangle \in {\mathfrak o}_L
		\intertext{so that,}
			 b_1 \langle x , x \rangle + 2b_2 \langle x , y \rangle + b_3 \langle y , y \rangle & \in {\mathfrak o}_L. 	
		\end{align*} 
		Therefore it is now clear that for all $(x,y) \in X^2$ we have
		\begin{align*}
			\omega([\begin{smallmatrix}
				1&B\\&1
			\end{smallmatrix}],1) \tilde\varphi(x,y)& = \psi(\langle b_1 \langle x , x \rangle + 2b_2 \langle x , y \rangle + b_3 \langle y , y \rangle)\tilde\varphi(x,y)\\ &= \tilde\varphi(x,y)
		\end{align*}

	\item[(c)] Let $C = [\begin{smallmatrix}
		c_1&c_2\\c_2&c_3
	\end{smallmatrix}] \in [\begin{smallmatrix}
		{\mathfrak p}^N & {\mathfrak p}^{N - 1}\\ {\mathfrak p}^{N - 1} & {\mathfrak o}_L
	\end{smallmatrix}]$ and set $C' = [\begin{smallmatrix}
		-c_1\varpi^{-2N} & -c_2 \varpi^{-N} \\ -c_2\varpi^{-N} & -c_3
	\end{smallmatrix}].$  If $(x,y) \in \supp (\omega(t_Ns_2,1) \tilde\varphi) = \supp (\mathcal F_1(\tilde\varphi_1)(\varpi^{-N}\cdot) \otimes \mathcal F_1(\tilde\varphi_2)(\cdot))$ then we have the congruences
	\begin{align*}
		\langle x , x \rangle  \in {\mathfrak p}^{n + 2}, \quad
		\langle y , y \rangle  &\in {\mathfrak o}_L, \quad \text{and} \quad
		\langle x , y \rangle  \in {\mathfrak p}
	\intertext{so that,}
		-c_1 \varpi^{-2N} & \langle x , x \rangle - 2 c_2 \varpi^{-N} \langle x , y \rangle - c_3 \langle y , y \rangle \in {\mathfrak o}_L.
	\end{align*}
	Therefore for all $(x,y) \in X^2$ we find that
	\begin{align*}
		& \ \ \ \ \omega([\begin{smallmatrix}
			1&\\C&1
		\end{smallmatrix}],1)\tilde\varphi(x,y) \\ & = \omega((t_Ns_2)^{-1} [\begin{smallmatrix}
			1&C' \\ &1
		\end{smallmatrix}]t_Ns_2 , 1)\tilde\varphi(x,y)\\
		& = \omega((t_Ns_2)^{-1} , 1) \psi(-c_1 \varpi^{-2N}  \langle x , x \rangle - 2 c_2 \varpi^{-N} \langle x , y \rangle - c_3 \langle y , y \rangle) [\omega(t_Ns_2 , 1)\tilde\varphi](x,y) \\
		& = \omega((t_Ns_2)^{-1} , 1) [\omega(t_Ns_2 , 1)\tilde\varphi](x,y)\\
		& = \tilde\varphi(x,y).
	\end{align*}
\end{enumerate}
	We have verified the claim for all the generators of $T$, which means that we are done. 
  \end{proof}

\begin{lemma} \label{lemma:klingen_invariance_cosets}
	The set $\{ s_2 , \left[\begin{smallmatrix}
			1&&&\\
			&1&&v\\
			&&1&\\
			&&&1
		\end{smallmatrix}\right] \mid  u,v \in {\mathfrak o}_L /{\mathfrak p}\}$ is a complete list of coset representatives of $\Kl({\mathfrak p}^N) / [\Kl({\mathfrak p}^N)  \cap T]$.
\end{lemma}
\begin{proof}
Let $k \in \Kl({\mathfrak p}^N)$. Then, there exists $a,b,c, a',b',c' \in {\mathfrak o}_L, u \in {\mathfrak o}_L^\times$ and $s = [\begin{smallmatrix}
		s_1 & s_2 \\ s_3 & s_4
	\end{smallmatrix}] \in \SL(2,{\mathfrak o}_L)$ such that
	\begin{equation} \label{eq:klingen_iwahori}
	 	k = \left[\begin{smallmatrix}
	 		1&a&c&b\\
	 		&1&b&\\
	 		&&1&\\
	 		&&-a&1
	 	\end{smallmatrix}\right] 
	 	\left[\begin{smallmatrix}
	 		u_1 &&&\\
	 		&s_1&&s_2\\
	 		&&u^{-1}&\\
	 		&s_3&&s_4
	 	\end{smallmatrix}\right]
	 	\left[\begin{smallmatrix}
	 		1 &&&\\
	 		\varpi^Na'&1&&\\
	 		\varpi^Nc'&\varpi^Nb'&1&-\varpi^Na'\\
	 		\varpi^Nb'&&&1
	 	\end{smallmatrix}\right]
	\end{equation} 
	by the Iwahori factorization (\cite{Roberts_Schmidt2007}, for example). Clearly the matrix on the right is in $T \cap \Kl({\mathfrak p}^N)$. First, assume that $s_4 \in {\mathfrak o}_L^\times$. In that case we find that
	\begin{align*}
		& \left[\begin{smallmatrix}
	 		1&a &c &b \\
	 		&1&b &\\
	 		&&1&\\
	 		&&-a &1
	 	\end{smallmatrix} \right]
	 	\left[\begin{smallmatrix}
	 		u_1 &&&\\
	 		&s_1&&s_2\\
	 		&&u^{-1}&\\
	 		&s_3&&s_4
	 	\end{smallmatrix}\right] (T\cap \Kl({\mathfrak p}^N))
	\\
		= &\left[\begin{smallmatrix}
	 		1&a &c &b \\
	 		&1&b &\\
	 		&&1&\\
	 		&&-a &1
	 	\end{smallmatrix}\right]
	 	\left[\begin{smallmatrix}
	 		1&&&\\
	 		&1&&s_4^{-1}s_2\\
	 		&&1&\\
	 		&&&1
	 	\end{smallmatrix}\right]
	 	\left[\begin{smallmatrix}
	 		u&&&\\
	 		&s_1 - s_4^{-1}s_2s_3&&\\
	 		&&u^{-1}&\\
	 		&s_3&&s_4
	 	\end{smallmatrix}\right](T\cap \Kl({\mathfrak p}^N))
	\\
		= &\left[\begin{smallmatrix}
	 		1&a &c &b \\
	 		&1&b &\\
	 		&&1&\\
	 		&&-a &1
	 	\end{smallmatrix}\right]
	 	\left[\begin{smallmatrix}
	 		1&&&\\
	 		&1&&s_4^{-1}s_2\\
	 		&&1&\\
	 		&&&1
	 	\end{smallmatrix} \right](T\cap \Kl({\mathfrak p}^N))
	\\
		=& \left[\begin{smallmatrix}
			1&&c +a b &b \\
			&1&b &&\\
			&&1&\\
			&&&1
		\end{smallmatrix}\right]
		\left[\begin{smallmatrix}
			1&a &&\\
			&1&&\\
			&&1&\\
			&&-a &1
		\end{smallmatrix}\right]
		\left[\begin{smallmatrix}
			1&&&\\
			&1&&s_4^{-1}s_2\\
			&&1&\\
			&&&1
		\end{smallmatrix}\right] (T\cap \Kl({\mathfrak p}^N))
	\\
		= & \left[\begin{smallmatrix}
			1&&c +a b &b \\
			&1&b&&\\
			&&1&\\
			&&&1
		\end{smallmatrix}\right]
		\left[\begin{smallmatrix}
			1&& a ^2s_4^{-1}s_2 & a s_4^{-1}s_2\\
			&1&a s_4^{-1}s_2 & s_4^{-1}s_2\\
			&&1&\\
			&&&1
		\end{smallmatrix}\right]
		\left[\begin{smallmatrix}
			1&a&&\\
			&1&&\\
			&&1&\\
			&&-a&1
		\end{smallmatrix}\right](T\cap \Kl({\mathfrak p}^N))
	\\
		= & \left[\begin{smallmatrix}
			1&& c  + a b  + a^2s_4^{-1}s_2 & b + as_4^{-1}s_2\\
			&1&b + as_4^{-1}s_2 & s_4^{-1}s_2\\
			&&1&\\
			&&&1
		\end{smallmatrix}\right] (T\cap \Kl({\mathfrak p}^N))
	\\
		= & \left[\begin{smallmatrix}
			1&& &\\
			&1&&s_4^{-1}s_2\\
			&&1&\\
			&&&1
		\end{smallmatrix}\right] (T\cap \Kl({\mathfrak p}^N)).
	\end{align*}
	Now assume that $s_4 \in {\mathfrak p}$ so that $s_2 \in {\mathfrak o}_L^\times$. In this case we find that
	\begin{align*}
		& \left[\begin{smallmatrix}
	 		1&a &c &b \\
	 		&1&b &\\
	 		&&1&\\
	 		&&-a &1
	 	\end{smallmatrix} \right]
	 	\left[\begin{smallmatrix}
	 		u_1 &&&\\
	 		&s_1&&s_2\\
	 		&&u^{-1}&\\
	 		&s_3&&s_4
	 	\end{smallmatrix}\right] (T\cap \Kl({\mathfrak p}^N))
	\\
		= & \left[\begin{smallmatrix}
	 		1&a &c &b \\
	 		&1&b &\\
	 		&&1&\\
	 		&&-a &1
	 	\end{smallmatrix} \right]
	 	\left[\begin{smallmatrix}
	 	 	1&&&\\
	 	 	&&&1\\
	 	 	&&1&\\
	 	 	&-1&&
	 	\end{smallmatrix}\right]
	 	\left[\begin{smallmatrix}
	 	 	1&&&\\
	 	 	&1&&-s_2^{-1}s_4\\
	 	 	&&1&\\
	 	 	&&&1
	 	\end{smallmatrix}\right]
	 	\left[\begin{smallmatrix}
	 		u&&&\\
	 		&-s_3 + s_2^{-1}s_1s_4 &&\\
	 		&&u^{-1}&\\
	 		&s_1&&s_2 	
	 	\end{smallmatrix} \right](T\cap \Kl({\mathfrak p}^N))
	\\
		= & \left[\begin{smallmatrix}
	 		1&a &c &b \\
	 		&1&b &\\
	 		&&1&\\
	 		&&-a &1
	 	\end{smallmatrix} \right]
	 	\left[\begin{smallmatrix}
	 	 	1&&&\\
	 	 	&&&1\\
	 	 	&&1&\\
	 	 	&-1&&
	 	\end{smallmatrix}\right] (T\cap \Kl({\mathfrak p}^N))
	\\
		=& \left[\begin{smallmatrix}
	 	 	1&&&\\
	 	 	&&&1\\
	 	 	&&1&\\
	 	 	&-1&&
	 	\end{smallmatrix}\right]
	 	\left[\begin{smallmatrix}
	 		1&-b &c &a \\
	 		&1&a &\\
	 		&&1&\\
	 		&&b &1
	 	\end{smallmatrix}\right](T\cap \Kl({\mathfrak p}^N))
	\\
		=& \left[\begin{smallmatrix}
	 	 	1&&&\\
	 	 	&&&1\\
	 	 	&&1&\\
	 	 	&-1&&
	 	\end{smallmatrix} \right]
	 	\left[\begin{smallmatrix}
	 		1&&c -b a &a \\
	 		&1&a &\\
	 		&&1&\\
	 		&&&1
	 	\end{smallmatrix}\right]
	 	\left[\begin{smallmatrix}
	 		1&-b &&\\
	 		&1&&\\
	 		&&1&\\
	 		&&b &1
	 	\end{smallmatrix}\right] (T\cap \Kl({\mathfrak p}^N))
	\\
		=& \left[\begin{smallmatrix}
	 	 	1&&&\\
	 	 	&&&1\\
	 	 	&&1&\\
	 	 	&-1&&
	 	\end{smallmatrix} \right]
	 	(T\cap \Kl({\mathfrak p}^N)).
	\end{align*}
	Which gives us exactly the cosets we described. 
  \end{proof}

\begin{lemma} \label{lemma:paramodoular_invariance_cosets}
	The following is an exact list of coset representatives of $\K({\mathfrak p}) / [\K({\mathfrak p}^N)  \cap T]$:
	\begin{equation*}
		\{ \left[\begin{smallmatrix}
			1&&\varpi^{-N}u&\\
			&1&&v\\
			&&1&\\
			&&&1
		\end{smallmatrix}\right], s_2 \left[\begin{smallmatrix}
			1&&\varpi^{-N}u&\\
			&1&&\\
			&&1&\\
			&&&1
		\end{smallmatrix}\right], t_N \left[\begin{smallmatrix}
			1&&&\\
			&1&&v\\
			&&1&\\
			&&&1
		\end{smallmatrix}\right], t_Ns_2 \mid u,v \in {\mathfrak o}_L/{\mathfrak p}  \}
	\end{equation*}
\end{lemma}
\begin{proof}
	Use the Lemma \ref{lemma:klingen_invariance_cosets} and the decomposition of $\K({\mathfrak p}^N)/ \Kl({\mathfrak p}^N)$ found in \cite{Roberts_Schmidt2007}, Lemma 3.3.1. The later is reproduced in \eqref{eq:KKlcosets}.
  \end{proof}
Define the Schwartz function
\begin{equation}\label{eq:phi_ramified}
	\varphi(x,y) = \sum_{g \in \K({\mathfrak p}^N)/ T \cap \K({\mathfrak p}^N)} \omega(g,1) \tilde\varphi(x,y).
\end{equation}
It is clear from the above discussion that $ \varphi$ is invariant under all of $K_1({\mathfrak p}^N)$. The lemma below shows that $\varphi\neq0$
\begin{lemma}
	We have that
	\begin{align} \label{eq:phi_ramified_explicit}
		 \varphi & = \varphi^{(1)} + \varphi^{(2)} + \varphi^{(3)} + \varphi^{(4)}
	\intertext{where}
	 \nonumber \varphi^{(1)}(x,y)& =  q_L^2  f_{{\mathfrak p}^{n + 2}}(\langle x , x \rangle) \varphi_1 (x) \ f_{{\mathfrak o}_L} (\langle y , y \rangle) (y) \varphi_2(y),
	\\ \nonumber
		\varphi^{(2)}(x,y) & = q_L \chi(y_2) f_{{\mathfrak p}^{n + 2}}(\langle x , x \rangle) \varphi_1(x) \ f_{\left[\begin{smallmatrix}
			{\mathfrak o}_E & \varpi_E^{-1}{\mathfrak o}_E^\times\\
			{\mathfrak P} & {\mathfrak o}_E	
		\end{smallmatrix}\right] \cap X}(y),
	\\ \nonumber
		\varphi^{(3)}(x,y)& = q \chi(x_2) f_{\left[\begin{smallmatrix}
			{\mathfrak P}^{n + 2} & \varpi_E {\mathfrak o}_E^\times \\ {\mathfrak P}^{2n + 3} & {\mathfrak P}^{n + 2}
		\end{smallmatrix}\right]}(x) \ f_{{\mathfrak o}_L} (\langle y , y \rangle) \varphi_2(y), \ \ \text{and}
	\\\nonumber
		\varphi^{(4)}(x,y) &  = \chi(x_2y_2) f_{\left[\begin{smallmatrix}
			{\mathfrak P}^{n + 2} & \varpi_E {\mathfrak o}_E^\times \\ {\mathfrak P}^{2n + 3} & {\mathfrak P}^{n + 2}
		\end{smallmatrix}\right]}(x)f_{\left[\begin{smallmatrix}
			{\mathfrak o}_E & \varpi_E^{-1}{\mathfrak o}_E^\times\\
			{\mathfrak P} & {\mathfrak o}_E	
		\end{smallmatrix}\right] \cap X}(y).
	\end{align}
	Moreover, the supports of the four summands in this formula are pairwise disjoint. 
\end{lemma}
\begin{proof}
	 We calculate 
	\begin{align*}
		\sum_{u \in {\mathfrak o}_L/{\mathfrak p}} \psi(\varpi^{-N}u \langle x , x \rangle)\varphi_1(x) & = q_L f_{{\mathfrak p}^N}(\langle x , x \rangle) \varphi_1(x),
	\\
		\sum_{v \in {\mathfrak o}_L/{\mathfrak p}}\psi(v \langle y , y \rangle) \varphi_2(y) & = q_L f_{{\mathfrak o}_L}(\langle y , y \rangle)\varphi_2(y).
	\end{align*}
	Therefore, applying Lemma \ref{lemma:Fourier_identites}, we have
	\begin{align*}
		\varphi^{(1)}=\sum_{u,v \in ({\mathfrak o}_L/{\mathfrak p})^2} \omega(\left[\begin{smallmatrix}
			1&&\varpi^{-N}u&\\
			&1&&v\\
			&&1&\\
			&&&1
		\end{smallmatrix}\right])\varphi(x,y) & = q_L^2 f_{{\mathfrak p}^N}(\langle x , x \rangle) \varphi_1(x) f_{{\mathfrak o}_L}(\langle y , y \rangle) \varphi_2(y)
	\\
		\varphi^{(2)}=\sum_{u \in {\mathfrak o}_L/{\mathfrak p}} \omega(s_2 \left[\begin{smallmatrix}
			1&&\varpi^{-N}u&\\
			&1&&\\
			&&1&\\
			&&&1
		\end{smallmatrix}\right])\varphi(x,y) & = q_L \chi(y_2) f_{{\mathfrak p}^N}(\langle x , x \rangle)\varphi_1(x) \ f_{\left[\begin{smallmatrix}
			{\mathfrak o}_E & \varpi^{-1}{\mathfrak o}_E^\times\\
			{\mathfrak P} & {\mathfrak o}_E	
		\end{smallmatrix}\right] \cap X}(y)
	\\
		\varphi^{(3)}=\sum_{v \in {\mathfrak o}_L/{\mathfrak p}} \omega (t_N\left[\begin{smallmatrix}
					1&&&\\
					&1&&v\\
					&&1&\\
					&&&1
				\end{smallmatrix}\right]) \varphi(x,y) & = q \chi(x_2) f_{\left[\begin{smallmatrix}
			{\mathfrak P}^{N} & \varpi_E {\mathfrak o}_E^\times \\ {\mathfrak P}^{2N - 1} & {\mathfrak P}^{N}
		\end{smallmatrix}\right]}(x) \ f_{{\mathfrak o}_L} (\langle y , y \rangle) \varphi_2(y)
	\\
	\\
		\varphi^{(4)}=\omega(t_Ns_2)\varphi(x,y)& = \chi(x_2y_2) f_{\left[\begin{smallmatrix}
			{\mathfrak P}^{N} & \varpi_E {\mathfrak o}_E^\times \\ {\mathfrak P}^{2N - 1} & {\mathfrak P}^{N}
		\end{smallmatrix}\right]}(x)f_{\left[\begin{smallmatrix}
			{\mathfrak o}_E & \varpi^{-1}{\mathfrak o}_E^\times\\
			{\mathfrak P} & {\mathfrak o}_E	
		\end{smallmatrix}\right] \cap X}(y).
	\end{align*}
	Since the cosets of $\K({\mathfrak p}^N)/(\K({\mathfrak p}^N) \cap T)$ are represented exactly once in the above sums we have proven that \ref{eq:phi_ramified_explicit} is valid. 
  \end{proof}
Recall the choice of $x_1 , x_2$ and $H$ from Section \ref{sec:notation}. We will now determine for which $h \in H\backslash\SO(X)$ we have $(h^{-1}x_1, h^{-1}x_2) \in \supp ( \varphi)$. This will help us evaluate the value of the intertwining map, as in \eqref{eq:splitbesselintegral}, evaluated at $ \varphi$.

\begin{lemma} \label{lemma:H_coset_ramified}
	Let $h \in \SO(X)$. Then 
	\begin{enumerate}
		\item[(a)] $(h^{-1}x_1 , h^{-1}x_2) \in \supp(\varphi^{(1)})$ if and only if there exists some $h' \in H$ such that $h'h = \rho(t,A)$ for some $t \in {\mathfrak o}_L^\times$ and $A \in [\begin{smallmatrix}
			{\mathfrak o}_E & {\mathfrak o}_E \\ \varpi^n {\mathfrak o}_E^\times& {\mathfrak o}_E^\times				
		\end{smallmatrix}]$ such that $\norm_L^E(\det A) = t^2$,

		\item[(b)] $(h^{-1}x_1 , h^{-1}x_2) \in \supp(\varphi^{(2)})$ if and only if $n = 0$ and there exists some $h' \in H$ such that $h'h = \rho(t,A)$ for some $t \in {\mathfrak o}_L^\times$ and $A \in [\begin{smallmatrix}
			{\mathfrak o}_E & {\mathfrak o}_E^\times \\ {\mathfrak o}_E^\times&  {\mathfrak P}
		\end{smallmatrix}]$ such that $\norm_L^E(\det A) = t^2$,

		\item[(c)] $(h^{-1}x_1 , h^{-1}x_2) \in \supp(\varphi^{(3)})$ if and only if there exists some $h' \in H$ such that $h'h = \rho(t,A)$ for some $t \in {\mathfrak o}_L^\times$ and $A \in [\begin{smallmatrix}
			{\mathfrak o}_E^\times & {\mathfrak o}_E \\ {\mathfrak P}^{n + 1} & {\mathfrak o}_E^\times				
		\end{smallmatrix}]$ such that $\norm_L^E(\det A) = t^2$, and

		\item[(d)] $(h^{-1}x_1 , h^{-1}x_2) \in \supp(\varphi^{(4)})$ for no $h \in \SO(X)$.
	\end{enumerate} 

	In particular we can assume that $t$ is one of the two representatives of $L^\times/\norm_L^E(E^\times)$. 
\end{lemma}
\begin{proof}
	Let's start with a preliminary observation. Let $A = [\begin{smallmatrix}
		a&b\\c&d
		\end{smallmatrix}]$ 
	 and suppose that there is some $h' \in H$ so that $h'h = \rho(1,A^{-1})$. 
	Then we calculate that 
	\begin{align}
		h^{-1} x_1 & = h^{-1} (h'^{-1} x_1) = (h'h)^{-1} x_1= t^{-1} A x_1 \alpha(A)^*= t^{-1} \left[\begin{smallmatrix}
			a \alpha(c)\sqrt{\delta}  & a \alpha(a)\sqrt{\delta} \\
			-c \alpha(c)\sqrt{\delta}  & -c \alpha(a)\sqrt{\delta} \label{eq:h_mat_x1}
		\end{smallmatrix}\right]\\
	\intertext{and,}
		h^{-1} x_2 & = h^{-1} (h'^{-1} x_2)= (h'h)^{-1} x_2= t^{-1} A x_2 \alpha(A)^*
		= 2t^{-1}\left[\begin{smallmatrix}
			-b \alpha(d)\frac{\sqrt{\delta}}{\delta}  &  b \alpha(b)\frac{\sqrt{\delta}}{\delta} \\
			-d \alpha(d)\frac{\sqrt{\delta}}{\delta}  &  d \alpha(b) \frac{\sqrt{\delta}}{\delta}
		\end{smallmatrix}\right]. \label{eq:h_mat_x2}
	\end{align}
	Also notice that for $i = 1,2$ and $g \in \GSO(X)$ we have that $0 =\langle x_i, x_i \rangle = \langle g^{-1} x_i , g^{-1} x_i \rangle$. Considering the current application, in \eqref{eq:phi_ramified_explicit} we can ignore the  restriction of the domain that $\langle x , x \rangle \in {\mathfrak p}^N$ and that $\langle y , y \rangle \in {\mathfrak o}_L$.   

	Suppose that $h = \rho(s, A')$ for some $s \in L^\times$ and some $A' \in \gl(2,E)$. We distinguish an element of the stabilizer $h' = \rho( \varpi_L^{-\nu_E(s)} , [\begin{smallmatrix}
			\varpi_E^{\nu_E(s)} & \\ & \varpi_E^{\nu_E(s)} 
		\end{smallmatrix}])\in H.$ So that, $h'h = \rho (t, A)$ for some $t \in {\mathfrak o}_L^\times$ and $A \in \gl(2,E)$. In particular we may assume that $t$ is a coset representative of the $L/L^1$. Since $(h'h)^{-1}x_i = h^{-1}x_i$, in the following arguments we can restrict our attention to $h = \rho(t,A)$. 

		% {There is a mistake above, it is not always possible to find such an $h' \in H$. It may be the best you can do is fine an $h' \in H$ so that $h'h = \rho(r , A)$ where $r \in {\mathfrak o}_L^\times$ so that $r$ is not a norm. Luckily it appears that the following argument works for any $r \in {\mathfrak o}^\times$ in the first coordinate of $\rho(r,A)$. It comes back into the bessel integral, but since both components of the $\varphi^{(i)}$ will have an extra $\chi(r)$ it should not change the value of the bessel integral.}

	\begin{enumerate}
		\item[(a)] Suppose that $(h^{-1}x_1 , h^{-1}x_2) \in \supp( \varphi^{(1)})$. By  \eqref{eq:phi_ramified_explicit} we have that $(h^{-1}x_1 , h^{-1}x_2) \in \supp( \varphi^{(1)})$ if and only if $h^{-1}x_1 \in \supp(\varphi_1)$ and $h^{-1}x_2 \in \supp (\varphi_2)$. These two conditions imply that $A \in [\begin{smallmatrix}
				{\mathfrak o}_E & {\mathfrak o}_E \\ \varpi^n {\mathfrak o}_E^\times & {\mathfrak o}_E^\times
			\end{smallmatrix}]$. 
				On the other hand, assume that $A\in [\begin{smallmatrix}
				{\mathfrak o}_E & {\mathfrak o}_E \\ \varpi^n {\mathfrak o}_E^\times & {\mathfrak o}_E^\times
			\end{smallmatrix}]$, then by \eqref{eq:h_mat_x1} and \eqref{eq:h_mat_x2} we can see that $h^{-1}x_1 \in \supp(\varphi_1)$ and $h^{-1}x_2 \in \supp(\varphi_2)$. We conclude that $(h^{-1}x_1 , h^{-1}x_2) \in \varphi^{(1)}$.

		\item[(b)] Suppose that $(h^{-1}x_1 , h^{-1}x_2) \in \supp( \varphi^{(2)})$. By \eqref{eq:phi_ramified_explicit} we have that $(h^{-1}x_1 , h^{-1}x_2) \in \supp( \varphi^{(1)})$ if and only if $h^{-1}x_1 \in \supp(\varphi_1)$ and $h^{-1} x_2 \in [\begin{smallmatrix}
			{\mathfrak o}_E & \varpi_E^{-1} {\mathfrak o}_E^\times \\ {\mathfrak P} & {\mathfrak o}_E
		\end{smallmatrix}]$. These two conditions imply that $A \in [\begin{smallmatrix}
				{\mathfrak o}_E & {\mathfrak o}_E^\times \\ \varpi^n_E{\mathfrak o}_E^\times & {\mathfrak P}
			\end{smallmatrix}]$.
		Since $\rho(t,A) \in \SO(X)$, this can only happen when $n = 0$.
		On the other hand, assume that $A \in [\begin{smallmatrix}
				{\mathfrak o}_E & {\mathfrak o}_E^\times \\ \varpi^n_E{\mathfrak o}_E^\times & {\mathfrak P}
			\end{smallmatrix}]$, then by \eqref{eq:h_mat_x1} and \eqref{eq:h_mat_x2} we can see that $h^{-1}x_1 \in \supp(\varphi_1)$ and $h^{-1}x_2 \in [\begin{smallmatrix}
			{\mathfrak o}_E & \varpi_E^{-1} {\mathfrak o}_E^\times \\
			{\mathfrak P} & {\mathfrak o}_E
		\end{smallmatrix}] \cap X$. We conclude that $(h^{-1}x_1 , h^{-1}x_2) \in \varphi^{(2)}$.

		\item[(c)] Suppose that $(h^{-1}x_1 , h^{-1}x_2) \in \supp( \varphi^{(3)})$. By \eqref{eq:phi_ramified_explicit}, this can happen if and only if $h^{-1}x_1 \in [\begin{smallmatrix}
			{\mathfrak P}^{n + 2} & \varpi_E {\mathfrak o}_E^\times \\ {\mathfrak P}^{2n + 3} & {\mathfrak P}^{n + 2}
		\end{smallmatrix}]$ and $h^{-1}x_2 \in \supp(\varphi_2)$. These two conditions imply that $A \in [\begin{smallmatrix}
				{\mathfrak o}_E^\times & {\mathfrak o}_E \\ {\mathfrak P}^{n + 1} & {\mathfrak o}_E^\times
			\end{smallmatrix}]$
		On the other hand, assume that $A \in [\begin{smallmatrix}
				{\mathfrak o}_E^\times & {\mathfrak o}_E \\ {\mathfrak P}^{n + 1} & {\mathfrak o}_E^\times
			\end{smallmatrix}]$, then by \eqref{eq:h_mat_x1} and \eqref{eq:h_mat_x2} we can see that $h^{-1}x_1 \in [\begin{smallmatrix}
			{\mathfrak P}^{n + 2} & \varpi_E {\mathfrak o}_E^\times \\ {\mathfrak P}^{2n + 3} & {\mathfrak P}^{n + 2}
		\end{smallmatrix}]$ and $h^{-1}x_2 \in \supp (\varphi_2)$. We conclude that $(h^{-1}x_1 , h^{-1}x_2) \in \varphi^{(3)}$.

		\item[(d)] Suppose that $(h^{-1}x_1 , h^{-1}x_2) \in \supp( \varphi^{(4)})$. This can happen if and only if $h^{-1}x_1 \in [\begin{smallmatrix}
			{\mathfrak P}^{n + 2} & \varpi_E {\mathfrak o}_E^\times \\ {\mathfrak P}^{2n + 3} & {\mathfrak P}^{n + 2}
		\end{smallmatrix}]$ and $h^{-1}x_2 \in [\begin{smallmatrix}
			{\mathfrak o}_E & \varpi_E^{-1} \\ {\mathfrak P} & {\mathfrak o}_E
		\end{smallmatrix}]$. These two conditions imply that $A \in [\begin{smallmatrix}
				{\mathfrak o}_E^\times & {\mathfrak o}_E^\times \\ {\mathfrak P}^{n + 1} & {\mathfrak P}
			\end{smallmatrix}]$,   which is impossible since $\rho(t,A) \in \SO(X)$.
	\end{enumerate}
  \end{proof}

\begin{corollary} \label{cor:H_coset_ramified} 
	Let $h \in \SO(X)$. Then $(h^{-1}x_1 , h^{-1}x_2) \in \supp( \varphi)$ if and only if there exists some $h' \in H$ such that $h'h = \rho(t,A)$ for $t$ a representative of $L^\times/\norm_L^E(E^\times)$ and some $A \in \Gamma_0({\mathfrak P}^n)$ such that $\norm_L^E(\det A) = t^2$. 
\end{corollary}
\begin{proof}
	If $n = 0$ then 
	\begin{align*}
		\Gamma_0({\mathfrak o}_E) = \gl(2,{\mathfrak o}_E) = ([\begin{smallmatrix}
					{\mathfrak o}_E^\times & {\mathfrak o}_E^\times \\ {\mathfrak o}_E^\times & {\mathfrak P}
				\end{smallmatrix}] \sqcup [\begin{smallmatrix}
					{\mathfrak P} & {\mathfrak o}_E^\times \\ {\mathfrak o}_E^\times & {\mathfrak o}_E^\times
				\end{smallmatrix}] \sqcup [\begin{smallmatrix}
					{\mathfrak P} & {\mathfrak o}_E^\times \\ {\mathfrak o}_E^\times & {\mathfrak P}
				\end{smallmatrix}] \sqcup ([\begin{smallmatrix}
					{\mathfrak o}_E^\times & {\mathfrak o}_E \\ {\mathfrak o}_E^\times & {\mathfrak o}_E^\times 
				\end{smallmatrix}] \cap \gl(2,{\mathfrak o}_E)) \sqcup [\begin{smallmatrix}
					{\mathfrak o}_E ^\times & {\mathfrak o}_E \\ {\mathfrak P} & {\mathfrak o}_E^\times
				\end{smallmatrix}]).
	\end{align*}

	If $n > 0$ then 
	\begin{align*}
		\Gamma_0({\mathfrak P}^n) = ([\begin{smallmatrix}
			{\mathfrak o}_E^\times & {\mathfrak o}_E \\ \varpi^{n}_E {\mathfrak o}_E^\times & {\mathfrak o}_E^\times
		\end{smallmatrix}] \sqcup [\begin{smallmatrix}
			{\mathfrak o}_E^\times & {\mathfrak o}_E \\ {\mathfrak P}^{n + 1} & {\mathfrak o}_E^\times
		\end{smallmatrix}]).
	\end{align*}
  Examining Lemma \ref{lemma:H_coset_ramified} gives the result.
  \end{proof}
To explicitly calculate the $B(1, \varphi, W ,s )$ we will need the volumes of the subsets of $\gl(2,{\mathfrak o}_E)$ that appear in the proof of Corollary \ref{cor:H_coset_ramified}, in terms of the volume of some subgroups of $\SO(X)$. When $n = 0$ set $\Gamma = \rho({\mathfrak o}_L^\times , \Gamma_0({\mathfrak P})) \cap \SO(X)$. When $n > 0$ set $\Gamma = \rho({\mathfrak o}_L^\times , \Gamma_0({\mathfrak P}^n)) \cap \SO(X)$. 

\begin{lemma}
	Table 1 summarizes volumes of certain subsets of $H\backslash \big(\rho({\mathfrak o}_L^\times, \gl(2, {\mathfrak o}_E)) \cap \SO(X)\big)$, up to a positive constant.
\end{lemma}
\begin{center}
\begin{table}
\label{volumestab}
\caption{Volumes of subsets of $\rho({\mathfrak o}_L^\times , \gl(2,{\mathfrak o}_E))$}
\renewcommand{\arraystretch}{1.5}
\setlength{\tabcolsep}{10pt}
\begin{tabular}{c|c|c}
$\rho({\mathfrak o}_L^\times, \cdot) \cap \so(X)$ & Volume in \ $\rho({\mathfrak o}_L^\times, \gl(2,{\mathfrak o}_E)) \cap \so(X)$ & Support in $\varphi^{(i)}$\\
\hline
$\left[ \begin{smallmatrix}
	{\mathfrak o}_E^\times & {\mathfrak o}_E^\times \\ {\mathfrak o}_E^\times & {\mathfrak P}
\end{smallmatrix} \right]$  & $(1 - \frac{1}{q})\vol(\Gamma)$ & 2
\\
\vspace{2pt}
$\left[ \begin{smallmatrix}
	{\mathfrak P} & {\mathfrak o}_E^\times \\  {\mathfrak o}_E^\times & {\mathfrak o}_E^\times
\end{smallmatrix} \right]$  & $(1 - \frac{1}{q})\vol(\Gamma)$ & 1
\\
\vspace{2pt}
$\left[ \begin{smallmatrix}
	{\mathfrak P} & {\mathfrak o}_E^\times \\  {\mathfrak o}_E^\times & {\mathfrak P}
\end{smallmatrix} \right]$  & $\frac{1}{q}\vol(\Gamma)$ & 2
\\
\vspace{2pt}
$\left[ \begin{smallmatrix}
	{\mathfrak o}_E^\times & {\mathfrak o}_E \\  {\mathfrak o}_E^\times & {\mathfrak o}_E^\times
\end{smallmatrix} \right] \cap \gl(2,{\mathfrak o}_E)$  & $1 - (3 - \frac{1}{q})\vol(\Gamma)$ & 1
\\
\vspace{2pt}
$\left[ \begin{smallmatrix}
	{\mathfrak o}_E^\times & {\mathfrak o}_E \\  {\mathfrak P} & {\mathfrak o}_E^\times
\end{smallmatrix} \right]$ & $\vol(\Gamma)$ & 3
\\
\hline
\vspace{2pt}
$\left[ \begin{smallmatrix}
	{\mathfrak o}_E^\times & {\mathfrak o}_E \\  \varpi^{n}_E {\mathfrak o}_E^\times & {\mathfrak o}_E^\times
\end{smallmatrix} \right]$  & $(1 - \frac{1}{q})\vol(\Gamma)$ & 1
\\
\vspace{2pt}
$\left[ \begin{smallmatrix}
	{\mathfrak o}_E^\times & {\mathfrak o}_E \\  {\mathfrak P}^{n + 1} & {\mathfrak o}_E^\times
\end{smallmatrix} \right]$  & $\frac{1}{q}\vol(\Gamma)$ & 3
 \end{tabular}
\vspace{.75cm}
\end{table}
\end{center}
\begin{proof}
	Recall the Haar measures of $\SO(X)$ and $H$ which were chosen in Section \ref{sec:intertwiningmaps}. Suppose that $S$ is a subgroup of $\SO(X)$ and $f_S(x)$ is the characteristic function of $S$. We have that 
	\begin{align*}
		\int\limits_{\SO(X)} f_S(g) \, dg = \int\limits_{H\backslash \SO(X)} \ (\int\limits_H f_S(hg)\, dh) \, dg= \int\limits_{H \backslash \SO(X)} \vol_H(H \cap S) f_{HS}(G) \, dh,     
	\end{align*}
which is $\vol_H(H\cap S) \vol_{H\backslash \SO(X)} (H\backslash HS)$.	Therefore, we have the formula for the quotient measure:
	\begin{equation} \label{eq:quotient_volume}
		\vol_{H\backslash G}(H \backslash HS) = \frac{\vol_{\SO(X)}(S)}{\vol_H(H\cap S)}.
	\end{equation}

	We begin with the case that $n >0$. First, note that the following is an exact list of coset representatives:
	\begin{equation*}
		\big(\rho({\mathfrak o}_L^\times, \Gamma_0({\mathfrak P}^{n + 1})) \cap \so(X)\big) \backslash \Gamma = \{\rho(1, [\begin{smallmatrix}
			1&\\x&1
		\end{smallmatrix}]) \mid x \in {\mathfrak o}_E/{\mathfrak P}\}.
	\end{equation*}
  Therefore, the volume of $\rho({\mathfrak o}_L^\times, \Gamma_0({\mathfrak P}^{n + 1})) \cap \so(X)$ is $\frac{1}{q} \vol(\Gamma)$. By \eqref{eq:quotient_volume} we have that
	\begin{align*}
		\vol(H\backslash \rho&({\mathfrak o}_L^\times, \Gamma_0({\mathfrak P}^{n + 1})) = \vol(H\backslash H\rho({\mathfrak o}_L^\times, \Gamma_0({\mathfrak P}^{n + 1})) = \frac{\vol(\rho({\mathfrak o}_L^\times, \Gamma_0({\mathfrak P}^{n + 1}))}{\vol(\rho({\mathfrak o}_L^\times, \Gamma_0({\mathfrak P}^{n + 1}) \cap H)}\\& = \frac{\vol(\rho({\mathfrak o}_L^\times, \Gamma_0({\mathfrak P}^{n + 1}))}{\vol(H)}= \frac{1}{q} \vol(\Gamma).
	\end{align*}
	Since $H\backslash \rho({\mathfrak o}_L^\times, \left[\begin{smallmatrix}
		{\mathfrak o}_E^\times & {\mathfrak o}_E \\ \varpi^n_E {\mathfrak o}_E^\times & {\mathfrak o}_E^\times
	\end{smallmatrix}\right]) \cap \so(X)$ is the complement of $H \backslash \rho({\mathfrak o}_L^\times, \Gamma_0({\mathfrak P}^{n + 1})) \cap \so(X)$ inside of $H\backslash \Gamma$, we find that that 
	\[\vol(H\backslash \rho({\mathfrak o}_L^\times, \left[\begin{smallmatrix}	
		{\mathfrak o}_E^\times & {\mathfrak o}_E \\ \varpi^n_E {\mathfrak o}_E^\times & {\mathfrak o}_E^\times
		\end{smallmatrix}\right]) \cap \so(X)) = \vol(H\backslash \Gamma) - \vol(H\backslash \rho({\mathfrak o}_L^\times, \Gamma_0({\mathfrak P}^{n+1}))\cap \so(X)) = (1 - \frac{1}{q})\vol(\Gamma).\]

	Now assume that $n = 0$. The following is a complete list of coset representatives
	\[
		\rho({\mathfrak o}_L^\times , \left[\begin{smallmatrix}
			{\mathfrak o}_E^\times & {\mathfrak P} \\ {\mathfrak P} & {\mathfrak o}_E^\times
		\end{smallmatrix}\right]) \cap \so(X) \backslash \Gamma = \{[\begin{smallmatrix}
			1&x\\&1
		\end{smallmatrix}] \mid x \in {\mathfrak o}_E/{\mathfrak P}\}.
	\] 
	So, similarly to the above case we find that 
	\begin{align*}
		\vol(H\backslash \rho&({\mathfrak o}_L^\times, \left[\begin{smallmatrix}
			{\mathfrak o}_E^\times & {\mathfrak P} \\ {\mathfrak P} & {\mathfrak o}_E^\times
		\end{smallmatrix}\right]) \cap \so(X)) = \vol(H\backslash H\rho({\mathfrak o}_L^\times, \left[\begin{smallmatrix}
			{\mathfrak o}_E^\times & {\mathfrak P} \\ {\mathfrak P} & {\mathfrak o}_E^\times
		\end{smallmatrix}\right]) \cap \so(X))
	= \frac{\vol(\rho({\mathfrak o}_L^\times, \left[\begin{smallmatrix}
			{\mathfrak o}_E^\times & {\mathfrak P} \\ {\mathfrak P} & {\mathfrak o}_E^\times
		\end{smallmatrix}\right]) \cap \so(X))}{\vol(\rho({\mathfrak o}_L^\times,\left[\begin{smallmatrix}
			{\mathfrak o}_E^\times & {\mathfrak P} \\ {\mathfrak P} & {\mathfrak o}_E^\times
		\end{smallmatrix}\right]) \cap H)}
	\\
		& = \frac{\vol(\rho({\mathfrak o}_L^\times, \left[\begin{smallmatrix}
			{\mathfrak o}_E^\times & {\mathfrak P} \\ {\mathfrak P} & {\mathfrak o}_E^\times
		\end{smallmatrix}\right]) \cap \so(X))}{\vol(H)}
	= \frac{1}{q} \vol(\Gamma).
	\end{align*}
	Since $H\backslash \rho({\mathfrak o}_L^\times, \left[\begin{smallmatrix}
		{\mathfrak o}_E^\times & {\mathfrak o}_E^\times \\ {\mathfrak P} & {\mathfrak o}_E^\times
	\end{smallmatrix}\right]) \cap \so(X))$ is the complement of $H \backslash \rho({\mathfrak o}_L^\times, \left[\begin{smallmatrix}
		{\mathfrak o}_E^\times & {\mathfrak P} \\ {\mathfrak P} & {\mathfrak o}_E^\times
	\end{smallmatrix}\right]) \cap \so(X))$ inside of $H\backslash \Gamma$, we find that that 
	\[
		\vol(H\backslash \rho({\mathfrak o}_L^\times, \left[\begin{smallmatrix}
			{\mathfrak o}_E^\times & {\mathfrak o}_E^\times \\ {\mathfrak P} & {\mathfrak o}_E^\times
		\end{smallmatrix}\right]) \cap \so(X)) = \vol(H\backslash \Gamma) - \vol(H\backslash \rho({\mathfrak o}_L^\times, \left[\begin{smallmatrix}
			{\mathfrak o}_E^\times & {\mathfrak P} \\ {\mathfrak P} & {\mathfrak o}_E^\times
		\end{smallmatrix}\right]) \cap \so(X)) = (1 - \frac{1}{q})\vol(\Gamma).
	\]
  \noindent Lastly we have that $H\backslash \rho({\mathfrak o}_L^\times, \left[\begin{smallmatrix}
		{\mathfrak o}_E^\times & {\mathfrak o}_E \\ {\mathfrak o}_E^\times & {\mathfrak o}_E^\times
	\end{smallmatrix}\right] \cap \gl(2,{\mathfrak o}_E)) \cap \so(X)$ is the complement of $H \backslash \rho({\mathfrak o}_L^\times, \left[\begin{smallmatrix}
		{\mathfrak o}_E^\times & {\mathfrak o}_E \\ {\mathfrak P} & {\mathfrak o}_E^\times
	\end{smallmatrix}\right] \sqcup \left[\begin{smallmatrix}
		{\mathfrak P} & {\mathfrak o}_E^\times \\ {\mathfrak o}_E^\times & {\mathfrak P}
	\end{smallmatrix}\right] \sqcup \left[\begin{smallmatrix}
		{\mathfrak P} & {\mathfrak o}_E^\times \\ {\mathfrak o}_E^\times & {\mathfrak o}_E^\times
	\end{smallmatrix}\right] \sqcup \left[\begin{smallmatrix}
		{\mathfrak o}_E^\times & {\mathfrak o}_E^\times \\ {\mathfrak o}_E^\times & {\mathfrak P}
	\end{smallmatrix}\right]) \cap \so(X))$ inside of $H\backslash \rho( {\mathfrak o}_L^\times, \gl(2,{\mathfrak o}_E))\cap \so(X)$. Therefore 
	\begin{align*}
		\vol(H\backslash \rho&({\mathfrak o}_L^\times, \left[\begin{smallmatrix}
		{\mathfrak o}_E^\times & {\mathfrak o}_E \\ {\mathfrak o}_E^\times & {\mathfrak o}_E^\times
	\end{smallmatrix}\right])\cap \so(X)) = 1 - \vol(H \backslash \rho({\mathfrak o}_L^\times, \left[\begin{smallmatrix}
		{\mathfrak o}_E^\times & {\mathfrak o}_E \\ {\mathfrak P} & {\mathfrak o}_E^\times
	\end{smallmatrix}\right] \sqcup \left[\begin{smallmatrix}
		{\mathfrak P} & {\mathfrak o}_E^\times \\ {\mathfrak o}_E^\times & {\mathfrak P}
	\end{smallmatrix}\right] \sqcup \left[\begin{smallmatrix}
		{\mathfrak P} & {\mathfrak o}_E^\times \\ {\mathfrak o}_E^\times & {\mathfrak o}_E^\times
	\end{smallmatrix}\right] \sqcup \left[\begin{smallmatrix}
		{\mathfrak o}_E^\times & {\mathfrak o}_E^\times \\ {\mathfrak o}_E^\times & {\mathfrak P}
	\end{smallmatrix}\right])\cap \so(X))
\\
	& = 1 - \vol(H\backslash \Gamma) - \vol(H \backslash \rho({\mathfrak o}_L^\times, \left[\begin{smallmatrix}
			{\mathfrak o}_E^\times & {\mathfrak P} \\ {\mathfrak P} & {\mathfrak o}_E^\times
		\end{smallmatrix}\right])\cap \so(X)) - 2 \vol(H\backslash \rho({\mathfrak o}_L^\times, \left[\begin{smallmatrix}
			{\mathfrak o}_E^\times & {\mathfrak o}_E^\times \\ {\mathfrak P} & {\mathfrak o}_E^\times
		\end{smallmatrix}\right]) \cap \so(X))\\&= 1 - (3 - \frac{1}{q}) \vol(\Gamma).
	\end{align*}
  This completes the proof.
  \end{proof}
  
The following result completes the proof of the main theorem for the tamely ramified case.
\begin{theorem} \label{MT_ramified}
Suppose that $L$ has odd residual characteristic, that $E/L$ is ramified, and that $W \in \mathcal W_{\tau_0}$ has $\Gamma_0({\mathfrak p}^n)$-invariance. Let $\varphi$ be chosen as in \eqref{eq:phi_ramified} and assume that $s \gg M$. Then, the function $B(\cdot,  \varphi, W, s)$ as defined in \eqref{eq:nonsplitbesselintegral} is non-zero and $\K({\mathfrak P}^N)$-invariant, where $N = n + 2$. In particular $B(1,\varphi , W , s) \ne 0$.	
\end{theorem}
\begin{proof}
	By construction we have that $\tilde \varphi$ is paramodular invariant under the action of the Weil representation. It follows that $B(\cdot, \tilde\varphi , W , s)$ is also paramodular invariant. We write a generic element of $H\backslash \SO(X)$ as $h = \rho(1 , [\begin{smallmatrix}
		a& b \\ c & d		
	\end{smallmatrix}])$. Using Lemma \ref{lemma:H_coset_ramified} to determine the support of the intertwining map we write
	\begin{align*}
		B(1, \varphi^{(1)}, W ,s) & = \int\limits_{H \backslash \SO(X)} \omega(1,h) \varphi^{(1)}(x_1 , x_2) Z(s, \pi(h) W) \,dh
	\\
		& = \int\limits_{H \backslash \rho({\mathfrak o}_L^\times, \big[\begin{smallmatrix} {\mathfrak o}_E & {\mathfrak o}_E \\ \varpi_E^{n} {\mathfrak o}_E^\times & {\mathfrak o}_E^\times \end{smallmatrix}\big]) \cap \so(X)}  \omega(1,h)   \varphi^{(1)}(x_1 , x_2) Z(s,\pi(h)W) \,dh
	\\
		& = q_L^2 \int\limits_{H \backslash \rho({\mathfrak o}_L^\times, \big[\begin{smallmatrix} {\mathfrak o}_E & {\mathfrak o}_E \\ \varpi_E^{n} {\mathfrak o}_E^\times & {\mathfrak o}_E^\times \end{smallmatrix}\big]) \cap \so(X)}   \chi(2c \alpha(c) d \alpha(d) \delta^{-1}) Z(s , \pi(h)W) \, dh
	\\
		& = q_L^2\vol(H \backslash \rho({\mathfrak o}_L^\times, \big[\begin{smallmatrix} {\mathfrak o}_E & {\mathfrak o}_E \\ \varpi_E^{n} {\mathfrak o}_E^\times & {\mathfrak o}_E^\times \end{smallmatrix}\big]) \cap \so(X)) \chi(2\delta) Z(s,W)
	\intertext{also,}
		B(1 , \varphi^{(3)} , W , s) & = \int\limits_{H \backslash \SO(X)} \omega(1,h) \varphi^{(3)}(x_1 , x_2) Z(s, \pi(h) W) \,dh
	\\
		& = \int\limits_{H \backslash \rho({\mathfrak o}_L^\times, \big[\begin{smallmatrix} {\mathfrak o}_E^\times & {\mathfrak o}_E \\ {\mathfrak P}^{n + 1} & {\mathfrak o}_E^\times \end{smallmatrix}\big]) \cap \so(X)}  \omega(1,h)   \varphi^{(3)}(x_1 , x_2) Z(s,\pi(h)W) \,dh
	\\
		& = q_L\int\limits_{H \backslash \rho({\mathfrak o}_L^\times, \big[\begin{smallmatrix} {\mathfrak o}_E^\times & {\mathfrak o}_E \\ {\mathfrak P}^{n + 1} & {\mathfrak o}_E^\times \end{smallmatrix}\big]) \cap \so(X)}   \chi(-2d \alpha(d) a \alpha(a) \delta^{-1}) Z(s , \pi(h)W) \, dh
	\\
		& = q_L\vol(H \backslash \rho({\mathfrak o}_L^\times, \big[\begin{smallmatrix} {\mathfrak o}_E^\times & {\mathfrak o}_E \\ {\mathfrak P}^{n + 1} & {\mathfrak o}_E^\times \end{smallmatrix}\big])\cap \so(X)) \chi(2) Z(s,W).
	\intertext{If $n = 0$ then we also have a contribution from}
		B(1 , \varphi^{(2)} , W , s) & = \int\limits_{H \backslash \SO(X)} \omega(1,h) \varphi^{(2)}(x_1 , x_2) Z(s, \pi(h) W) \,dh
	\\
		& = \int\limits_{H \backslash \rho({\mathfrak o}_L^\times, \big[\begin{smallmatrix} {\mathfrak P} & {\mathfrak o}_E^\times \\ {\mathfrak o}_E^\times & {\mathfrak o}_E  \end{smallmatrix}\big]) \cap \so(X)}  \omega(1,h)   \varphi^{(2)}(x_1 , x_2) Z(s,\pi(h)W) \,dh
	\\
		& = q_L\int\limits_{H \backslash \rho({\mathfrak o}_L^\times, \big[\begin{smallmatrix} {\mathfrak P} & {\mathfrak o}_E^\times \\ {\mathfrak o}_E^\times & {\mathfrak o}_E  \end{smallmatrix}\big]) \cap \so(X)}   \chi(-2d \alpha(d) b \alpha(b) \delta^{-1}) Z(s , \pi(h)W) \, dh
	\\
		& = q_L\vol(H \backslash \rho({\mathfrak o}_L^\times, \big[\begin{smallmatrix} {\mathfrak P} & {\mathfrak o}_E^\times \\ {\mathfrak o}_E^\times & {\mathfrak o}_E  \end{smallmatrix}\big] \cap \so(X))) \chi(2) Z(s,W).
	\end{align*} 
	Therefore, from Lemma \ref{lemma:H_coset_ramified}, we find that
	\begin{align*}
		B(1, \varphi, W,s)  = &B(1, \varphi^{(1)} + \varphi^{(2)} + \varphi^{(3)}, W,s)
	\\
		= &\chi(2) Z(s,  W) \bigg[\chi(\delta) q_L\vol(H \backslash \rho(1, \big[\begin{smallmatrix} {\mathfrak o}_E & {\mathfrak o}_E \\ \varpi^n_E {\mathfrak o}_E^\times & {\mathfrak o}_E^\times \end{smallmatrix}\big]) \cap \so(X)) \\
		&  +  q_L^2\vol(H \backslash \rho(1, \big[\begin{smallmatrix} {\mathfrak o}_E^\times & {\mathfrak o}_E \\ {\mathfrak P}^{n + 1} & {\mathfrak o}_E^\times \end{smallmatrix}\big]) \cap \so(X)) +\epsilon(n)  q_L\vol(H \backslash \rho(1, \big[\begin{smallmatrix} {\mathfrak P} & {\mathfrak o}_E^\times \\ {\mathfrak o}_E^\times & {\mathfrak o}_E  \end{smallmatrix}\big]) \cap \so(X))\bigg]
	\end{align*}
	where $\epsilon(0) = 1$ and is equal to zero elsewhere.

	When $n > 0$ we can use the volumes in Table \ref{volumestab} to calculate that 
	\begin{align*}
		B(1, \varphi, W , s) & = \chi(2) Z(s,  W) \bigg[\chi(\delta) q_L\vol(H \backslash \rho(1, \big[\begin{smallmatrix} {\mathfrak o}_E^\times & {\mathfrak o}_E 
		\\ \varpi^n_E {\mathfrak o}_E^\times & {\mathfrak o}_E^\times \end{smallmatrix}\big]) \cap \so(X)) \\ & \hspace{3cm} +  q_L^2\vol(H \backslash \rho(1, \big[\begin{smallmatrix} {\mathfrak o}_E^\times & {\mathfrak o}_E \\ {\mathfrak P}^{n + 1} & {\mathfrak o}_E^\times \end{smallmatrix}\big]) \cap \so(X)) \bigg]
	\\
		& = \chi(2) Z(s,  W) \vol(\Gamma) \bigg[\chi(\delta) (1 - \frac{1}{q}) + \frac{1}{q} \bigg] \ne 0
	\end{align*}
	since $q \ne 2$. When $n = 0$ we have 
	\begin{align*}
		B(1,  \varphi, W , s) & = \chi(2) Z(s,  W) \bigg[\chi(\delta) q_L\vol(H \backslash \rho(1, \big[\begin{smallmatrix} {\mathfrak o}_E & {\mathfrak o}_E \\{\mathfrak o}_E^\times & {\mathfrak o}_E^\times \end{smallmatrix}\big]) \so(X)) \\ &\hspace{3cm} +  q_L^2\vol(H \backslash \rho(1, \big[\begin{smallmatrix} {\mathfrak o}_E^\times & {\mathfrak o}_E \\ {\mathfrak P} & {\mathfrak o}_E^\times \end{smallmatrix}\big]) \so(X))
	\\
		& \hspace{3cm} +  q_L\vol(H \backslash \rho(1, \big[\begin{smallmatrix} {\mathfrak P} & {\mathfrak o}_E^\times \\ {\mathfrak o}_E^\times & {\mathfrak o}_E  \end{smallmatrix}\big]) \cap \so(X))\bigg]
	\\
		& = \chi(2) Z(s,W) \bigg[ \chi(\delta)q \big( 1 - (3 -\frac{1}{q}) \vol(\Gamma) + (1 - \frac{1}{q})\vol(\Gamma) \big) + q^2(\vol(\Gamma)) + q(\vol(\Gamma))\bigg]
	\\
		& = \chi(2) Z(s,W) \bigg[ \chi(\delta)q + \vol(\Gamma)(\chi(\delta)(4q - 2) + q^2 + q)\bigg]
	\\
		& \ne 0
	\end{align*}
	since $\vol(\Gamma) = \frac{1}{q + 1}$.
  \end{proof}

% section ramified_case (end)

% chapter a_choice_of_shwartz_function (end)

%\begin{thebibliography}{6}
\bibliography{explicittheta.bib}

\providecommand{\bysame}{\leavevmode\hbox to3em{\hrulefill}\thinspace}
\providecommand{\MR}{\relax\ifhmode\unskip\space\fi MR }
\providecommand{\MRhref}[2]{%
  \href{http://www.ams.org/mathscinet-getitem?mr=#1}{#2}
}
\providecommand{\href}[2]{#2}
\begin{thebibliography}{FLHS15}

\bibitem[BK14]{BK}
Armand Brumer and Kenneth Kramer, \emph{Paramodular abelian varieties of odd
  conductor}, Trans. Amer. Math. Soc. \textbf{366} (2014), no.~5, 2463--2516.
  \MR{3165645}

\bibitem[Bum97]{bump}
Daniel Bump, \emph{Automorphic forms and representations}, Cambridge Studies in
  Advanced Mathematics, vol.~55, Cambridge University Press, Cambridge, 1997.
  \MR{1431508 (97k:11080)}

\bibitem[BZ76]{Bernshtein_Zelevinskii_1976}
I.~N. Bern{\v{s}}te{\u\i}n and A.~V. Zelevinski{\u\i}, \emph{Representations of
  the group {$GL(n,F),$} where {$F$} is a local non-{A}rchimedean field},
  Uspehi Mat. Nauk \textbf{31} (1976), no.~3(189), 5--70. \MR{0425030 (54
  \#12988)}

\bibitem[FLHS15]{RLS}
Nuno Freitas, Bao~V. Le~Hung, and Samir Siksek, \emph{Elliptic curves over real
  quadratic fields are modular}, Invent. Math. \textbf{201} (2015), no.~1,
  159--206. \MR{3359051}

\bibitem[Gan23]{ganIHESnotes}
{Wee Teck} Gan, \emph{Explicit constructions of automorphic forms: Theta
  correspondence and automorphic descent}, arXiv e-prints (2023),
  arXiv:2303.14919.

\bibitem[Gel75]{gelbart}
Stephen Gelbart, \emph{Automorphic forms and representations of adele groups},
  Department of Mathematics, University of Chicago, Chicago, Ill., 1975,
  Lecture Notes in Representation Theory. \MR{0419361 (54 \#7382)}

\bibitem[God70]{godement}
R~Godement, \emph{Notes on jacquet-langlands' theory}, The Institute for
  Advanced Study, 1970.

\bibitem[How79]{Howe79}
R.~Howe, \emph{{$\theta $}-series and invariant theory}, Automorphic forms,
  representations and {$L$}-functions ({P}roc. {S}ympos. {P}ure {M}ath.,
  {O}regon {S}tate {U}niv., {C}orvallis, {O}re., 1977), {P}art 1, Proc. Sympos.
  Pure Math., vol. XXXIII, Amer. Math. Soc., Providence, RI, 1979,
  pp.~275--285. \MR{546602}

\bibitem[JL70a]{Jacquet1970}
H.~Jacquet and R.~P. Langlands, \emph{Automorphic forms on gl (2)}, Springer
  Berlin Heidelberg, 1970.

\bibitem[JL70b]{JacquetLanglandsgl2}
\bysame, \emph{Automorphic forms on {${\rm GL}(2)$}}, Lecture Notes in
  Mathematics, Vol. 114, Springer-Verlag, Berlin-New York, 1970. \MR{0401654
  (53 \#5481)}

\bibitem[JLR12]{Johnson-Leung_Roberts2012}
Jennifer Johnson-Leung and Brooks Roberts, \emph{Siegel modular forms of degree
  two attached to {H}ilbert modular forms}, J. Number Theory \textbf{132}
  (2012), no.~4, 543--564. \MR{2887605}

\bibitem[JLRS23]{JohnsonLeung2023}
Jennifer Johnson-Leung, Brooks Roberts, and Ralf Schmidt, \emph{Stable klingen
  vectors and paramodular newforms}, Springer Nature Switzerland, 2023.

\bibitem[Knu91]{Knus}
Max-Albert Knus, \emph{Quadratic and {H}ermitian forms over rings}, Grundlehren
  der Mathematischen Wissenschaften [Fundamental Principles of Mathematical
  Sciences], vol. 294, Springer-Verlag, Berlin, 1991, With a foreword by I.
  Bertuccioni. \MR{1096299 (92i:11039)}

\bibitem[Rob01]{Roberts2001}
Brooks Roberts, \emph{Global {$L$}-packets for {${\rm GSp}(2)$} and theta
  lifts}, Doc. Math. \textbf{6} (2001), 247--314 (electronic). \MR{1871665
  (2003a:11059)}

\bibitem[Rob03]{Roberts_Bessel_manuscript}
\bysame, \emph{Epsilon factors for some representations of gsp(2) and bessel
  coefficients}, 2003, Unpublished Manuscript.

\bibitem[RS07]{Roberts_Schmidt2007}
Brooks Roberts and Ralf Schmidt, \emph{Local newforms for {GS}p(4)}, Lecture
  Notes in Mathematics, vol. 1918, Springer, Berlin, 2007. \MR{2344630
  (2008g:11080)}

\bibitem[Wal80]{wald1}
J.-L. Waldspurger, \emph{Correspondance de {S}himura}, J. Math. Pures Appl. (9)
  \textbf{59} (1980), no.~1, 1--132. \MR{577010}

\bibitem[Wal81]{wald2}
\bysame, \emph{Sur les coefficients de {F}ourier des formes modulaires de poids
  demi-entier}, J. Math. Pures Appl. (9) \textbf{60} (1981), no.~4, 375--484.
  \MR{646366}

\bibitem[Wei64]{Weil64}
Andr\'{e} Weil, \emph{Sur certains groupes d'op\'{e}rateurs unitaires}, Acta
  Math. \textbf{111} (1964), 143--211. \MR{165033}

\bibitem[Yos79]{Yoshida1979}
Hiroyuki Yoshida, \emph{Weil's representations of the symplectic groups over
  finite fields}, J. Math. Soc. Japan \textbf{31} (1979), no.~2, 399--426.
  \MR{527552 (82j:20084)}

\end{thebibliography}
\bibliographystyle{amsalpha}
%\end{thebibliography}
\end{document}